\documentclass{amsart}

\usepackage[letterpaper, margin=1.4in]{geometry}

\usepackage[backend=biber,style=alphabetic,sorting=nyt,minalphanames=10,maxalphanames=10,maxnames=15, doi=false,isbn=false,url=false,giveninits=true]{biblatex}
\renewbibmacro{in:}{} 
\usepackage{amssymb}
\usepackage{amssymb, latexsym, MnSymbol}
\usepackage{color, xcolor}
\usepackage{amscd, graphicx, psfrag, tikz}
\usetikzlibrary{calc,patterns}
\usepackage[mathscr]{euscript}
\usepackage[all]{xy}

\usepackage{enumitem}
\usepackage{stmaryrd}
\usepackage{url,hyperref}

\hypersetup{
    hidelinks,
    colorlinks=false,
    linkcolor=black,
    citecolor=black,
    urlcolor=black
}

\usepackage{comment}

\newcommand{\Ga}{\Gamma}
\newcommand{\Si}{ {\Sigma} }

\newcommand{\si}{ {\sigma} }

\newcommand{\bC}{ {\mathbb{C}} }
\newcommand{\bD}{\mathbb{D}}
\newcommand{\bE}{\mathbb{E}}
\newcommand{\bF}{\mathbb{F}}

\newcommand{\bK}{\mathbb{K}}
\newcommand{\bL}{\mathbb{L}}

\newcommand{\bP}{\mathbb{P}}
\newcommand{\bQ}{\mathbb{Q}}
\newcommand{\bR}{\mathbb{R}}
\newcommand{\bS}{\mathbb{S}}
\newcommand{\bT}{\mathbb{T}}

\newcommand{\bZ}{\mathbb{Z}}

\newcommand{\cA}{\mathcal{A}}

\newcommand{\cD}{\mathcal{D}}
\newcommand{\cE}{\mathcal{E}}
\newcommand{\cF}{\mathcal{F}}

\newcommand{\cI}{\mathcal{I}}
\newcommand{\cK}{\mathcal{K}}
\newcommand{\cL}{\mathcal{L}}
\newcommand{\cM}{\mathcal{M}}
\newcommand{\cO}{\mathcal{O}}

\newcommand{\cS}{\mathcal{S}}
\newcommand{\cT}{\mathcal{T}}
\newcommand{\cX}{\mathcal{X}}
\newcommand{\cV}{\mathcal{V}}

\newcommand{\cHom}{\mathcal{H}om}

\newcommand{\age}{\mathrm{age}}
\newcommand{\Aut}{\mathrm{Aut}}
\newcommand{\CR}{ {\mathrm{CR}} }
\newcommand{\orb}{\mathrm{orb}}

\newcommand{\Hom}{\mathrm{Hom}}
\newcommand{\End}{\mathrm{End}}
\newcommand{\Jac}{ {\mathrm{Jac}} }
\newcommand{\Ker}{\mathrm{Ker}}
\newcommand{\Res}{\mathrm{Res}}
\newcommand{\Spec}{\mathrm{Spec}}
\newcommand{\Pic}{\mathrm{Pic}}
\renewcommand{\Re}{\mathrm{Re}}
\renewcommand{\Im}{\mathrm{Im}}
\newcommand{\can}{ {\mathrm{can}} }
\newcommand{\eff}{ {\mathrm{eff}} }
\newcommand{\ext}{ {\mathrm{ext}} }
\newcommand{\ev}{\mathrm{ev}}
\newcommand{\val}{ {\mathrm{val}} }
\newcommand{\vir}{ {\mathrm{vir}} }

\newcommand{\inv}{\mathrm{inv}}
\newcommand{\Int}{\mathrm{Int}}
\newcommand{\pt}{\mathrm{pt}}

\newcommand{\Nef}{{\mathrm{Nef}}}
\newcommand{\NE}{{\mathrm{NE}}}

\newcommand{\Picst}{\Pic^\mathrm{st}}

\newcommand{\FM}{\mathrm{FM}}

\renewcommand{\Box}{\mathrm{Box}}

\newcommand{\depth}{{\mathrm{depth}}}
\newcommand{\comb}{{\mathrm{comb}}}

\newcommand{\Ext}{\mathrm{Ext}}
\newcommand{\mir}{\mathrm{mir}}
\newcommand{\Coh}{\mathrm{Coh}}
\newcommand{\BM}{\mathrm{BM}}

\newcommand{\bfS}{\mathbf{S}}
\newcommand{\be}{\mathbf{e}}
\newcommand{\bff}{\mathbf{f}}

\newcommand{\bu}{\mathbf{u}}

\newcommand{\one}{\mathbf{1}}
\newcommand{\bmu}{\boldsymbol{\mu}}
\newcommand{\btau}{\boldsymbol{\tau}}
\newcommand{\bsi}{{\boldsymbol{\si}}}
\newcommand{\brho}{{\boldsymbol{\rho}}}
\newcommand{\bgamma}{{{\boldsymbol{\gamma}}}}

\newcommand{\bSi}{\mathbf{\Si}}

\newcommand{\fg}{\mathfrak{g}}
\newcommand{\fl}{\mathfrak{l}}
\newcommand{\fm}{\mathfrak{m}}

\newcommand{\fn}{\mathfrak{n}}
\newcommand{\fp}{\mathfrak{p}}
\newcommand{\fr}{\mathfrak{r}}
\newcommand{\fs}{\mathfrak{s}}

\newcommand{\fI}{\mathfrak{I}}

\newcommand{\w}{\mathsf{w}}
\newcommand{\su}{\mathsf{u}}
\newcommand{\sv}{\mathsf{v}}
\newcommand{\sw}{\mathsf{w}}

\newcommand{\sfa}{\mathsf{a}}
\newcommand{\sfb}{\mathsf{b}}

\newcommand{\hH}{\hat{H}}

\newcommand{\hx}{\hat{x}}
\newcommand{\hy}{\hat{y}}
\newcommand{\hxi}{\hat{\xi}}
\newcommand{\htheta}{\hat{\theta}}

\newcommand{\hGa}{\hat{\Gamma}}
\newcommand{\hpsi}{\hat{\psi}}

\newcommand{\hubsi}{\hat{\underline {\boldsymbol{\sigma}}}}
\newcommand{\hubrho}{\hat{\underline {\boldsymbol{\rho}}}}

\newcommand{\tC}{\widetilde{C}}

\newcommand{\tF}{\widetilde{F}}

\newcommand{\tI}{\widetilde{I}}

\newcommand{\tM}{\widetilde{M}}
\newcommand{\tN}{\widetilde{N}}
\newcommand{\tQ}{\widetilde{Q}}

\newcommand{\tS}{\widetilde{S}}

\newcommand{\tb}{\widetilde{b}}

\newcommand{\tp}{\widetilde{p}}

\newcommand{\tbsi}{{\widetilde{\bsi}}}

\newcommand{\tbT}{\widetilde{\bT}}

\newcommand{\tch}{\widetilde{\mathrm{ch}} }

\newcommand{\tNef}{\widetilde{\Nef}}
\newcommand{\tNE}{\widetilde{\NE}}

\newcommand{\vGa}{\vec{\Gamma}}
\newcommand{\bGa}{\mathbf{\Gamma}}
\newcommand{\Mbar}{\overline{\cM}}

\newcommand{\spa}{ {\ \ \,} }

\newcommand{\Cbar}{\overline{C}}
\newcommand{\Dbar}{\bar{D}}

\newcommand{\Ibar}{\overline{I}}

\newcommand{\ST}{ {S_{\bT'}} }
\newcommand{\RT}{ {R_{\bT'}} }
\newcommand{\bST}{ {\bar{S}_{\bT'}} }
\newcommand{\bSTQ}{ \bST \llbracket \tQ,\tau'' \rrbracket }

\newcommand{\bSt}{ {\bar{S}_{\bT'}} }

\newcommand{\nov}{\Lambda_{\mathrm{nov}}}
\newcommand{\novT}{\bar{\Lambda}^{\bT'}_{\mathrm{nov}} }

\newcommand{\XX}{X_{\bff}}
\newcommand{\YY}{Y_{\bff}}

\newcommand{\chZ}{\check{Z}}

\newcommand{\chcX}{\check{\cX}}

\newcommand{\inner}[1]{\left\langle  #1 \right\rangle}
\newcommand{\ceil}[1]{\lceil  #1 \rceil}
\newcommand{\floor}[1]{\lfloor  #1 \rfloor}
\newcommand{\formal}[1]{\left\llbracket #1  \right\rrbracket}
\newcommand{\double}[1]{\left\llangle #1 \right\rrangle}

\newcommand{\uchi}{{\underline{\chi}}}

\newtheorem{dummy}{dummy}[section]
\newtheorem{lemma}[dummy]{Lemma}
\newtheorem{theorem}[dummy]{Theorem}

\newtheorem{corollary}[dummy]{Corollary}
\newtheorem{proposition}[dummy]{Proposition}
\newtheorem{remark}[dummy]{Remark}
\newtheorem{definition}[dummy]{Definition}

\newtheorem{convention}[dummy]{Convention}

\begin{document}

\title{Remodeling Conjecture with Descendants}

\author{Bohan Fang}
\address{Bohan Fang, Beijing International Center for Mathematical Research, Peking University, 5 Yiheyuan Road, Beijing 100871, China}
\email{bohanfang@gmail.com}

\author{Chiu-Chu Melissa Liu}
\address{Chiu-Chu Melissa Liu, Department of Mathematics, Columbia University, 2990 Broadway, New York, NY 10027}
\email{ccliu@math.columbia.edu}

\author{Song Yu}
\address{Song Yu, Yau Mathematical Sciences Center, Tsinghua University, Haidian District, Beijing 100084, China}
\email{song-yu@tsinghua.edu.cn}

\author{Zhengyu Zong}
\address{Zhengyu Zong, Department of Mathematical Sciences, Tsinghua University, Haidian District, Beijing 100084, China}
\email{zyzong@mail.tsinghua.edu.cn}

\begin{abstract}

We formulate and prove the Remodeling Conjecture with descendants, which is a version of all-genus equivariant descendant mirror symmetry for semi-projective toric Calabi-Yau 3-orbifolds with integral structures. We construct an isomorphism between the $K$-group of equivariant coherent sheaves on the toric Calabi-Yau 3-orbifold with support bounded in a direction and a certain integral relative first homology group of the equivariant mirror curve. Under this isomorphism, we prove the equivariant mirror symmetric Gamma conjecture which equates quantum cohomology central charges of coherent sheaves and oscillatory integrals along corresponding relative 1-cycles. As a consequence in the non-equivariant setting, we prove a conjecture of Hosono which equates central charges of compactly supported coherent sheaves and period integrals of integral 3-cycles on the Hori-Vafa mirror 3-fold. Furthermore, we establish a correspondence between all-genus equivariant descendant Gromov-Witten invariants with $K$-theoretic framings and oscillatory integrals (Laplace transforms) of the Chekhov-Eynard-Orantin topological recursion invariants along relative 1-cycles on the equivariant mirror curve.

\end{abstract}
\maketitle

\setcounter{tocdepth}{1}

\tableofcontents


\section{Introduction} \label{sec:Intro} 

\subsection{Background and motivation}
Mirror symmetry relates the A-model on a K\"ahler manifold or orbifold $\cX$, defined by the symplectic structure, to the B-model on the mirror defined by the complex structure. In particular, mirror symmetry relates the Gromov-Witten theory of $\cX$, which can be viewed as a mathematical theory of the topological A-model, to the topological B-model on the mirror. The integral structure of the A-model variation of Hodge structure, in general given by B-branes (coherent sheaves), is expected to match the integral structure of the B-model variation of Hodge structure, in general given by A-branes (Lagrangian submanifolds).


\subsubsection{Mirror symmetric Gamma conjecture for toric Fano orbifolds}\label{sect:MSGammaFano}

When $\cX$ is an $r$-dimensional projective toric Fano orbifold, the mirror is a Landau-Ginzburg model $((\bC^*)^r, W)$ where the superpotential $W: (\bC^*)^r \to \bC$ is a family of Laurent polynomials defined by the toric geometry of $\cX$. The genus-zero Gromov-Witten theory of $\cX$ is related to oscillatory integrals on the mirror, which, when $\cX$ is a manifold, is manifested in the equation
\begin{equation}\label{eqn:MSGammaFano}
    \int_\gamma e^{-W(X,q)/z} \frac{dX_1}{X_1} \wedge \cdots \wedge \frac{dX_r}{X_r} = \int_{\cX} J(\btau, -z)\hGa_{\cX}^z \tch_{z}(\cE).
\end{equation}
Here, the left-hand side is an oscillatory integral taken along a cycle $\gamma \in H_r((\bC^*)^r, \Re(W) \gg 0; \bZ)$ and is a function in the complex parameter $q$ and a formal variable $z$. The right-hand side involves the following ingredients:
\begin{itemize}
    \item $J(\btau,z)$ is the \emph{$J$-function} \cite{Givental98} from the genus-zero Gromov-Witten theory of $\cX$, parameterized by the stringy K\"{a}hler parameter $\btau$ of $\cX$;
    
    \item $\hGa_{\cX}^z$ is the ($z$-modified) \emph{Gamma class} of $\cX$, which is a characteristic class of the tangent bundle $T\cX$ defined by the Euler $\Gamma$-function;
    
    \item $\tch_{z}(\cE)$ is the ($z$-modified twisted) Chern character of a coherent sheaf $\cE$ on $\cX$.
\end{itemize}
In particular, Equation \eqref{eqn:MSGammaFano} implies that the Gamma class of $\cX$ determines the asymptotics of the oscillatory integrals on the mirror. It is termed the \emph{mirror symmetric Gamma conjecture} for $\cX$ in \cite{Iritani23}. The appearance of the Gamma class in the B-model periods has been observed since the early studies of mirror symmetry \cites{HKTY95,Libgober99}.

In general, Equation \eqref{eqn:MSGammaFano} builds on an explicit mirror map that relates the complex parameter $q$ and the stringy K\"{a}hler parameter $\btau$. Under the mirror map, the B-model variation of Hodge structures (VHS) defined by the Gauss-Manin connection is identified with the A-model VHS defined by quantum cohomology \cite{Iritani09,CCIT15}. A significance of \eqref{eqn:MSGammaFano} is the underlying correspondence between the integral cycles $\gamma$, which give the integral structure in the B-model VHS, and the coherent sheaves $\cE$ viewed as elements of the $K$-group $K(\cX)$, which give an integral structure in the A-model VHS \cite{Iritani09,KKP08}. Put differently, Equation \eqref{eqn:MSGammaFano} identifies the central charge of (the A-brane) $\gamma$, defined by the oscillatory integral, with the quantum cohomology central charge of (the B-brane) $\cE$ defined by the right-hand side via the Gamma class. The correspondence between the cycle $\gamma$ and the sheaf $\cE$ is expected to be consistent with well-studied mirror constructions, including homological mirror symmetry and Strominger-Yau-Zaslow (SYZ) $T$-duality \cites{SYZ96,LYZ00}. For projective toric Fano orbifolds, Iritani \cite{Iritani09} proved the orbifold and equivariant generalization of \eqref{eqn:MSGammaFano} and the first author \cite{Fang20} provided an explicit description of the correspondence between the cycles and the sheaves which was verified to be consistent with toric SYZ mirror duality \cite{Abouzaid09,fltz12}.

\subsubsection{Mirror symmetric Gamma conjecture for toric Calabi-Yau 3-orbifolds}\label{sect:MSGammaCY}

In this paper, we consider the mirror symmetric Gamma conjecture for a semi-projective toric Calabi-Yau 3-orbifold $\cX$, which must be non-compact. In this local Calabi-Yau setting, one may consider the B-model on the \emph{Hori-Vafa mirror} $(\check{\cX}_q, \Omega_q)$ \cite{HV00}, where 
$$
    \check{\cX}_q =\{ (u,v, X, Y)\in\bC^2\times (\bC^*)^2: uv=H(X,Y,q) \}
$$
is a non-compact Calabi-Yau 3-fold, and 
$$
    \Omega_q := \Res_{\check{\cX}_q} \left(\frac{1}{H(X,Y,q)-uv} du\wedge dv \wedge \frac{dX}{X}\wedge \frac{dY}{Y} \right)
$$
is a holomorphic 3-form on $\check{\cX}_q$. When $\cX$ is a manifold, we have the following analog of \eqref{eqn:MSGammaFano} \cite{Hosono06}:
\begin{equation}\label{eqn:MSGammaCY}
    \int_\gamma \Omega_q = \int_{\cX} J(\btau, -z)\hGa_{\cX}^z \tch_{z}(\cE),
\end{equation}
which can be generalized to the orbifold case. Here, the left-hand side is a period of $\Omega_q$ along a 3-cycle $\gamma \in H_3(\check{\cX};\bZ)$. Since $\cX$ is non-compact, for the right-hand side to be defined, a compact-support condition is imposed on $\cE$. A long-standing conjecture of Hosono \cite{Hosono06} predicts a canonical isomorphism
\begin{equation}\label{eqn:HosonoCycleIso}
    K^c(\cX) \to H_3(\check{\cX};\bZ),
\end{equation}
where $K^c(\cX)$ is the $K$-group of $\cX$ with support on the \emph{core} (see Section \ref{sect:IntroSupport}), under which \eqref{eqn:MSGammaCY} holds. The isomorphism is also expected to be consistent with SYZ $T$-duality and homological mirror symmetry \cite{Seidel10,AAK16,GM18}. Gross-Matessi \cite{GM18} proposed an explicit correspondence between the cycles and sheaves above and a precise homological mirror symmetry conjecture when $\cX$ is a manifold. However, verifying \eqref{eqn:MSGammaCY} and in particular computing the periods have been challenging in general. Iritani \cite{Iritani23} studied the case $\cX = K_F$ for a smooth toric Fano surface $F$ and consider the structure sheaf $\cE = \cO_F$ of the zero section. Han \cite{Han} studied a dual version of Hosono's conjecture (see Remark \ref{rem:Han}). We note that \eqref{eqn:MSGammaCY} has also been considered for (compact) Batyrev mirror Calabi-Yau pairs in \cite{AGIS20,Yamamoto24}.

We approach Hosono's conjecture via the reduction of the Hori-Vafa B-model to a theory on the \emph{mirror curve} \cite{HV00}, which is the affine curve
$$
    C_q =\{ (X, Y)\in (\bC^*)^2: H(X,Y,q) = 0\}.
$$
The periods of $\Omega_q$ on $\check{\cX}_q$ can be reduced to the periods of the (multivalued) holomorphic 1-form $ydx$ on $C_q$, where $x = -\log X$ and $y = -\log Y$. We further consider the equivariant situation with respect to the Calabi-Yau 2-torus $\bT'$ of $\cX$ using the Landau-Ginzburg model $(\tC_q, \hx)$, where $\tC_q$ is the pullback of $C_q$ under the universal covering $\bC^2_{x,y} \to (\bC^*)^2_{X, Y}$ and
$$
    \hx = u_1x + u_2y
$$
is a holomorphic function on $\tC_q$ viewed as the $\bT'$-equivariant superpotential, with $u_1, u_2$ the equivariant parameters. Set $\Phi = y d \left( x + \frac{u_2}{u_1} y \right)$. In this paper, we prove the following $\bT'$-equivariant version of \eqref{eqn:MSGammaCY}:
\begin{equation}\label{eqn:MSGammaEquiv}
    2\pi\sqrt{-1}\int_\gamma e^{-\hx/z}\Phi = \int_{\cX} J_{\bT'}(\btau, -z)\hGa_{\cX}^z \tch_{z,\bT'}(\cE),
\end{equation}
where the left-hand side is an oscillatory integral along a relative 1-cycle $\gamma$ in $H_1(\tC_q,  \Re (\hx) \gg 0; \bZ)$. This is also the equivariant analog of \eqref{eqn:MSGammaFano}. We prove this statement by constructing an explicit correspondence between the relative 1-cycles on $\tC_q$ and $\bT'$-equivariant coherent sheaves $\cE$ on $\cX$ with support bounded in a prescribed direction, and computing the respective central charges. Such sheaves form a $K$-group $K^+_{\bT'}(\cX)$ that interpolates the usual $K$-group and the $K$-group with compact support (see Section \ref{sect:IntroSupport} for additional discussions), which we show provides an integral structure in the $\bT'$-equivariant A-model and is isomorphic to $H_1(\tC_q,  \Re (\hx) \gg 0; \bZ)$. In the non-equivariant limit, our correspondence of equivariant cycles and sheaves induces the the canonical isomorphism \eqref{eqn:HosonoCycleIso} under which Hosono's conjecture \eqref{eqn:MSGammaCY} holds.




\subsubsection{All-genus descendant mirror symmetry via topological recursion}
For a toric orbifold $\cX$ with the action of a torus $T$, we denote $\kappa_z^{T}(\cE) = \hGa_{\cX}^z \tch_{z,T}(\cE)$ and rewrite the quantum cohomology central charge of $\cE$ as the correlator
$$
    \double{\frac{\kappa_z^{T}(\cE)}{z(z + \hpsi)}}^{\cX,T}_{0,1}
$$
of genus-zero, 1-pointed, $T$-equivariant descendant Gromov-Witten invariants, where $\hpsi$ is the descendant class. The mirror symmetric Gamma conjecture is then a descendant mirror symmetry statement with integral structures. It is natural to generalize the statement to the genus-$g$, $n$-pointed case for all $(g,n)$ beyond $(0,1)$.

When the mirror of $\cX$ is given by a curve, the higher-genus topological B-model may be provided by the \emph{Chekhov-Eynard-Orantin topological recursion} \cite{EO07} invariants $\omega_{g,n}$ on the curve. The procedure may be viewed as a B-model analog of the Givental-Teleman quantization \cite{Givental01,Teleman12} in higher-genus Gromov-Witten theory \cite{DOSS14}. In the example $\cX = \bP^1$ with the action of $T = (\bC^*)^2$ scaling the two homogeneous coordinates, Fang-Liu-Zong \cite{FLZ17} proved that the $T$-equivariant Gromov-Witten correlator
$$
    \double{\frac{\kappa_{z_1}^T(\cE_1)}{z_1(z_1 + \hpsi)},\dots,\frac{\kappa_{z_n}^T(\cE_n)}{z_n(z_n + \hpsi)} }^{\bP^1, T}_{g,n}
$$
can be recovered by the oscillatory integral of the topological recursion invariant $\omega_{g,n}$ on the $T$-equivariant Landau-Ginzburg mirror $(\bC^*, W^T)$ along the product of the relative cycles $\gamma_1, \dots, \gamma_n$ corresponding to $\cE_1, \dots, \cE_n$ respectively. This result generalizes (the $T$-equivariant version of) Equation \eqref{eqn:MSGammaFano} and in fact uses it as an essential ingredient in the proof. The result is generalized by Tang \cite{Tang} to the weighted projective line and by Lan \cite{Lan25} to footballs.


In this paper, we establish the all-genus descendant mirror symmetry for the $\bT'$-equivariant theory on a toric Calabi-Yau 3-orbifold $\cX$ that recovers the correlator
$$
    \double{\frac{\kappa_{z_1}^{\bT'}(\cE_1)}{z_1(z_1 + \hpsi)},\dots,\frac{\kappa_{z_n}^{\bT'}(\cE_n)}{z_n(z_n + \hpsi)} }^{\cX, \bT'}_{g,n}
$$
from the oscillatory integral of topological recursion invariant $\omega_{g,n}$ on the mirror curve along corresponding cycles. The use of topological recursion on the mirror curve to reformulate the higher-genus topological B-model on the Hori-Vafa mirror was proposed by Bouchard-Klemm-Mari\~{n}o-Pasquetti (BKMP) \cite{BKMP09, BKMP10}, based on the work of Eynard-Orantin \cite{EO07} and Mari\~{n}o \cite{Marino08}. In particular, they conjectured a precise correspondence, known as the  {\em Remodeling Conjecture}, 
between  $\omega_{g,n}$ and a generating function $F_{g,n}^{\cX, \cL}$ of open Gromov-Witten invariants counting
holomorphic maps from bordered Riemann surfaces with $g$ handles and $n$ holes to $\cX$ with boundaries in an Aganagic-Vafa 
Lagrangian brane $\cL$. This extends the genus-zero open mirror symmetry on disk invariants \cite{AV00, AKV02, FL13, FLT22} which corresponds to the case
$(g,n)=(0,1)$. In the closed string sector, the Remodeling Conjecture relates the B-model genus-$g$ free energy $\check{F}_g= \omega_{g,0}$ to a generating function $F_g^{\cX}$ of genus-$g$ primary Gromov-Witten invariants of $\cX$. 
Simply put, the BKMP Remodeling Conjecture can be viewed as a version of all-genus primary and open mirror symmetry. Eynard-Orantin provided a proof of the Remodeling Conjecture for general smooth  semi-projective toric Calabi-Yau 3-folds \cite{EO15}. Fang-Liu-Zong proved the Remodeling Conjecture for general semi-projective  toric Calabi-Yau 3-orbifolds \cite{flz2020remodeling}. Our result on the all-genus descendant mirror symmetry can then be viewed as the Remodeling Conjecture \emph{with descendants}. It is a higher-loop version of the $(0,1)$-case in \eqref{eqn:MSGammaEquiv} and builds upon the explicit identification of integral structures there.

\subsection{Cohomologies and $K$-theories of toric orbifolds with support conditions}\label{sect:IntroSupport}
As discussed in Section \ref{sect:MSGammaCY}, the A-model integral structure for a semi-projective toric Calabi-Yau 3-orbifold is provided by its $K$-theory with certain support conditions. In this paper, we introduce the $K$-theory with bounded below/above support for a general semi-projective toric orbifold. It provides the required integral structure in \eqref{eqn:MSGammaEquiv} in 3-dimensional Calabi-Yau case and may be of independent interest.


A toric variety $X$ (over $\bC$) is semi-projective if it contains a torus fixed point and is projective over its affinization $X_0:=\Spec\left(H^0(X,\cO_X)\right)$, while a toric orbifold is semi-projective if its coarse moduli space is semi-projective. Let $X$ be a semi-projective simplicial toric variety of dimension $r$ and $X_0$ be its affinization. Let $\cX$ be the minimal orbifold having $X$ as the coarse moduli space. Then $\cX$ is a toric orbifold and contains
the algebraic torus $\bT=(\bC^*)^r$ as the unique $r$-dimensional torus orbit. The rational cohomology ring $H^*(X;\bQ)$ of $X$ admits a 
Stanley-Reisner presentation \cite[Section 2]{HS02}, generalizing the classical results in the projective case. The Chen-Ruan orbifold cohomology ring $H^*_{\CR}(\cX;\bQ)$ of $\cX$ also admits a Stanley-Reisner presentation \cite{JT08}, generalizing the results of Borisov-Chen-Smith \cite{BCS05} in the projective case. 
The cohomology group $H_c^*(X;\bC)$ with compact support (and with complex coefficients) is a module over $H^*(X;\bC)$ and admits a Stanley-Reisner type presentation \cite[Section 2]{BH15}, and similarly for the Chen-Ruan orbifold cohomology $H^*_{\CR,c}(\cX;\bC)$ with compact support.

The core  $\cX^c$ of $\cX$ is the union of all compact
orbit closures in $\cX$; it is also the preimage of the origin (the unique torus fixed point in $X_0$) under the map 
$\cX \to X \to X_0$.  Let $D^c(\cX)$ denote the full subcategory of $D(\cX):= D^b(\Coh(\cX))$ consisting of complexes of sheaves whose cohomology sheaves are supported on 
$\cX^c$.  Borisov-Horja \cite{BH06,BH15} provided combinatorial descriptions of the Grothendieck groups $K(\cX):= K_0(D(\cX))$ and $K^c(\cX):= K_0(D^c(\cX))$, and define combinatorial Chern character maps
$$
\tch: K(\cX) \to  H,  \qquad
\tch^c: K^c(\cX)\to  H^c
$$
where $H= H^*_{\CR}(\cX;\bC)$ and $H^c = H^*_{\CR,c}(\cX;\bC)$.
The Chern character maps induce  $\bC$-linear isomorphisms
$K(\cX)\otimes \bC\cong H$ and $K^c(\cX)\otimes \bC \cong H^c$. There is a non-degenerate Euler characteristic pairing $\chi: K(\cX)\times K^c(\cX)\to \bZ$. 


Now we discuss the cohomology and $K$-theory with bounded below/above support
. Let $\bT_\bR \cong U(1)^r$ be the maximal compact subgroup of $\bT \cong (\bC^*)^r$. 
Then $\cX$ can be equipped with a K\"{a}hler form  $\omega$ such that the holomorphic action of $\bT$ on the complex orbifold $\cX$ restricts to a Hamiltonian action of $\bT_\bR$ on the symplectic orbifold  $(\cX,\omega)$ with a moment map $\mu_{\bT_\bR}: \cX\to M_\bR =M\otimes_{\bZ}\bR\cong \bR^r$, where $M:=\Hom(\bT, \bC^*)\cong \bZ^r $ is the character lattice of $\bT$. 
Let $\lambda_{\sv}:\bC^*\to \bT$ be a one-parameter subgroup of $\bT$ associated to a primitive vector  $\sv$ in the cocharacter lattice
$N:=\Hom(\bC^*, \bT)$ of $\bT$, which restricts to a Hamiltonian circle action with moment map $\mu_{\sv} =\langle \mu_{\bT_\bR},  \sv \rangle:\cX\to \bR$; we choose $\sv$ generically such that  the fixed  substack of this one-parameter subgroup is equal to $\cX^{\bT}$. 

Let $\cX^+$ (resp. $\cX^-$) be the union of $\bT$ orbit closures in $\cX$ on which the $\bR$-valued function 
$\mu_{\sv}$ is bounded below (resp. above). We introduce the bounded below/above $K$-group
$K^\pm(\cX):= K_0(D^\pm(\cX))$, which is the Grothendieck group of the full subcategory $D^\pm(\cX)$ of $D(\cX)$ consisting 
of complexes of sheaves whose cohomology sheaves are supported on $\cX^\pm$, and the bounded below/above Chen-Ruan cohomology group $H^\pm$. 
We provide combinatorial descriptions of $K^\pm(\cX)$ and define combinatorial Chern character maps
$$
\tch^\pm: K^\pm(\cX)\to H^\pm 
$$
which induce $\bC$-linear isomorphisms $K^\pm(\cX)\otimes \bC \cong H^\pm$. There is a non-degenerate Euler characteristic pairing 
$\chi: K^+(\cX)\times K^-(\cX)\to \bZ$. The definitions may be generalized to the equivariant setting with respect to $\bT$ or an algebraic subtorus.

If $\sv$ is in the support $|\Sigma|$ of the simplicial fan $\Sigma$ defining $X$, then
$$
\cX^+ =\cX, \quad \cX^- = \cX^c,\quad K^+(\cX) = K(\cX), \quad K^-(\cX)= K^c(\cX).
$$
So bounded below/above $K$-theories $K^\pm(\cX)$ include $K(\cX)$ and $K^c(\cX)$ as special cases, and similarly
bounded below/above cohomology groups $H^\pm$ include $H$ and $H^c$ as special cases. 

\subsection{Main results} \label{sec:main-results}
Let $\cX$ be a semi-projective toric Calabi-Yau 3-orbifold and let $\bT' \cong (\bC^*)^2$ be the Calabi-Yau $2$-subtorus of $\bT$ which acts trivially on the canonical line bundle of $\cX$.
We consider the generating functions of genus-$g$, $\bT'$-equivariant descendant Gromov-Witten invariants of $\cX$ of the form
$$
\left\llangle \frac{\gamma_1}{z_1-\hat{\psi}},\dots, \frac{\gamma_n}{z_n-\hat{\psi}} \right\rrangle^{\cX,\bT'}_{g,n}
=\sum_{a_1,\dots,a_n\in \bZ_{\ge 0}}
\left\llangle \gamma_1 \hat{\psi}^{a_1}, \dots, \gamma_n \hat{\psi}^{a_n} \right\rrangle^{\cX,\bT'}_{g,n}\prod_{i=1}^n z_i^{-a_i-1}.
$$
where $z_1,\cdots,z_n$ are formal variables, $\gamma_1,\ldots,\gamma_n \in H^*_{\CR,\bT'}(\cX;\bC)$ are $\bT'$-equivariant Chen-Ruan cohomology classes of $\cX$, and $\hat{\psi}$ is the descendant class. (See Section \ref{sect:genfun} for more details.) 

\subsubsection{Mirror symmetric Gamma conjecture}
Our first main result is the $\bT'$-equivariant mirror symmetric Gamma conjecture for $\cX$, as discussed in Section \ref{sect:MSGammaCY}. 
We consider the $K$-theory of $\cX$ with bounded below/above support defined by a cocharacter $\sv$ of the Calabi-Yau 2-torus $\bT'$, i.e. $\sv \in N' :=\Hom(\bC^*,\bT')$. Let $M' :=\Hom(\bT',\bC^*) \cong \bZ^2$ be the character lattice. Then $M'$ acts on  $D_{\bT'}(\cX)$, $D^\pm_{\bT'}(\cX)$ (which can be viewed as categories of  B-branes on $\cX$)   and their Grothendieck groups $K_{\bT'}(\cX) = K_0(D_{\bT'}(\cX))$,  $K_{\bT'}^\pm(\cX) = K_0(D_{\bT'}^\pm(\cX))$, by tensoring with characters of $\bT'$.  For $\cE \in K_{\bT'}(\cX)$,  we define $\bT'$-equivariant $K$-theoretic framing
$$
\kappa^{\bT'}_z(\cE) := \hGa_{\cX}^z \tch_{z, \bT'}(\cE)
$$ 
where $\hGa_{\cX}^z$ is the $\bT'$-equivariant Gamma class of $\cX$ and $\tch_{z, \bT'}(\cE)= (\frac{-2\pi\sqrt{-1}}{z})^{\deg/2}\tch_{\bT'}(\cE)$ is the $\bT'$-equivariant twisted Chern character. Define the $\bT'$-equivariant quantum cohomology central charge of $\cE$ by 
$$
    Z_{\bT'}(\cE):= \double{\frac{\kappa_z^{\bT'}(\cE)}{z(z + \hpsi)}}^{\cX, \bT'}_{0,1}.
$$
In particular, the definition is made for $\cE \in K_{\bT'}^+(\cX)$.

The mirror curve  $C_q$ is a hypersurface in the  2-torus $(\bT')^\vee = M'\otimes_{\bZ} \bC^*$ dual to $\bT'$. The $\bT'$-equivariant mirror curve $\tC_q$ of $\cX$ is the preimage of $C_q$ under the universal covering
map $M'_{\bC}:= M'\otimes_{\bZ} \bC  \cong \bC^2 \to  (\bT')^\vee \cong (\bC^*)^2$:
$$
\tC_q  = C_q \times_{(\bT')^\vee} M'_\bC = C_q\times_{(\bC^*)^2}\bC^2. 
$$
In particular, $p:\tC_q\to C_q$ is a regular covering with fiber $M'\cong \bZ^2$. Let $x =-\log X$ and $y= -\log Y$ be complex coordinates on 
$M'_\bC =\bC^2$, which restrict to holomorphic functions on $\tC_q$.  We consider
$$
\hx= u_1 x + u_2 y,\quad \hy=\frac{y}{u_1}.
$$
where $u_1,u_2\in \bC$, $u_1\neq 0$.  The $M'$-action on $\tC_q$ by deck transformation induces an $M'$-action on the group of integral relative 1-cycles
$$
    H_1(\tC_q,  \Re (\hx) \gg 0; \bZ). 
$$

\begin{theorem}[={\bf Theorem \ref{thm:EquivMir}}) (Genus-zero descendant mirror theorem]  \label{thm:equivMir}
There is a unique homomorphism
$$
    \mir^+_{\bT'}: K^+_{\bT'}(\cX) \to H_1(\tC_q,  \Re (\hx) \gg 0; \bZ)
$$
such that for any $\cE \in K^+_{\bT'}(\cX)$, we have
$$
Z_{\bT'}(\cE) = 2\pi\sqrt{-1}\int_{\mir^+_{\bT'}(\cE)} e^{- \hx/z} \hy d\hx 
$$
under the mirror map $\btau=\btau(q)$.  Moreover, $\mir^+_{\bT'}$ is $M'$-equivariant and is an isomorphism.
\end{theorem}


The construction of $\mir^+_{\bT'}$ is given in Sections \ref{sec:divisor-cycle}, \ref{sec:curve-cycle}. On a high level, we construct the mirror cycles of generators of $K^+_{\bT'}(\cX)$ in the form of a $\bT'$-equivariant line bundle on a bounded-below toric divisor or curve in $\cX$. Representing the line bundle by a twisted polytope \cite{fltz14}, i.e. the data of its induced $\bT'$-characters at the torus fixed points, we find the local cycles that are mirror to the characters and patch them into a global cycle. The equality of central charges is verified locally. Moreover, underlying the isomorphism $\mir^+_{\bT'}$ is an identification of the defining parameters $\sv$ and $(u_1, u_2)$ in a wall-and-chamber structure on $N'_{\bR}:= N'\otimes_{\bZ} \bR$. See Section \ref{sect:EquivMir} for additional details and an illustration.

\subsubsection{Non-equivariant situation: Hosono's conjecture}
In the non-equivariant setting, we use Theorem \ref{thm:equivMir} to prove Hosono's conjecture \cite[Conjecture 6.3]{Hosono06} that relates non-equivariant genus-zero Gromov-Witten invariants 
of $\cX$ to period integrals of the holomorphic 3-form $\Omega_q$ of the Hori-Vafa mirror $\chcX_q$. As discussed in Section \ref{sect:MSGammaCY}, the A-model integral structure comes from the (non-equivariant) $K$-theory of $\cX$ with compact support. We define the central charge of $\cE \in K^c(\cX)$ by 
$$
Z^c(\cE) :=  \double{\frac{\kappa_z(\cE)}{z(z + \hpsi)}  }^{\cX}_{0,1}  = \double{ \one, \frac{\kappa_z(\cE)}{z + \hpsi}  }^{\cX}_{0,2}  
$$
where $\kappa_z(\cE) \in z^2 H^2_{\CR,c}(\cX;\bC) \oplus z H^4_{\CR,c}(\cX;\bC) \oplus H^6_{\CR,c}(\cX;\bC)$ is the non-equivariant $K$-theoretic framing. It turns out
that $Z^c(\cE)$ does not depend on $z$ and (by the dilaton equation) is a generating function of genus-zero non-equivariant primary Gromov-Witten invariants. 

\begin{theorem}[={\bf Theorem \ref{thm:Kc3Mir}}) (Hosono's conjecture]\label{thm:kc3Mir}
There is a unique homomorphism
$$
    \mir^c_3: K^c(\cX) \to H_3(\chcX_q; \bZ)
$$
such that for any $\cE \in K^c(\cX)$, we have
$$    
Z^c(\cE)=     -\int_{\mir^c_3(\cE)} \Omega_q 
$$
under the mirror map $\btau=\btau(q)$.  Moreover, $\mir^c_3$ is an isomorphism.
\end{theorem}

The construction of $\mir^c_3$ is induced by $\mir^+_{\bT'}$ from Theorem \ref{thm:equivMir} in the non-equivariant limit, in combination with the dimensional reduction from the Hori-Vafa mirror to the mirror curve. We expect that the construction agrees with the correspondence proposed by \cite{GM18}.

\begin{remark}\label{rem:Han}\rm{
Han \cite{Han} studied a dual version of the Hosono's conjecture  for semi-projective toric Calabi-Yau orbifolds of any dimension $r$. When $r=3$, the 
B-brane central charges $\big(Z^{B, \cE}_c \big)_{ c\in C^\circ}$ in \cite[Definition 1.4]{Han},
where $\cE$ is any class in $K(\cX)$, are related to 
generating functions  $\Big\{  \double{ z^{-\deg/2} a, \frac{\kappa_z(\cE)}{z + \hpsi} }^{\cX}_{0,2}(\btau) : a \in H^c \Big\}$ of genus-zero Gromov-Witten invariants of $\cX$ 
under the mirror map $\btau=\btau(q)$, where $\kappa_z(\cE) \in z^3 H^0_{\CR}(\cX;\bC)  \oplus z^2 H^2_{\CR}(\cX;\bC) \oplus z H^4_{\CR}(\cX;\bC)$.
}\end{remark}

\subsubsection{All-genus descendant mirror symmetry, a.k.a. Remodeling Conjecture with descendants}
Finally, we promote Theorem \ref{thm:equivMir} to higher-genus. Specifically, we obtain all-genus $\bT'$-equivariant descendant Gromov-Witten invariants of $\cX$ from oscillatory integrals of the topological recursion invariants $\omega_{g,n}$ defined from the meromorphic functions $\hx, \hy$ and the fundamental bidifferential of the second kind on the mirror curve.

\begin{itemize}
\item  Given $\cE_1, \dots, \cE_n \in K^+_{\bT'}(\cX)$, we define
$$
    Z^+_{g,n}(\cE_1, \dots, \cE_n) := \double{\frac{\kappa_{z_1}^{\bT'}(\cE_1)}{z_1(z_1 + \hpsi)},\dots,\frac{\kappa_{z_n}^{\bT'}(\cE_n)}{z_n(z_n + \hpsi)} }^{\cX, \bT'}_{g,n}.
$$  
In particular, $Z^+_{0,1}(\cE) = Z_{\bT'}(\cE)$. 

\item Given $\gamma_1, \dots, \gamma_n \in H_1(\tC_q,  \Re (\hx) \gg 0; \bZ)$, we define
$$
    \chZ^+_{g,n}(\gamma_1, \dots, \gamma_n) := (2\pi\sqrt{-1})^n \int_{p_1\in\gamma_1}\cdots\int_{p_n\in\gamma_n } e^{ - (\hx(p_1)/z_1+\cdots+\hx(p_n)/z_n)} \omega_{g,n}(p_1,\dots,p_n).
$$
In particular, $\chZ^+_{0,1}(\gamma) = 2\pi\sqrt{-1} \displaystyle{ \int_\gamma  e^{- \hx/z} \hy d\hx}$ (since $\omega_{0,1}= \hy d\hx$). 
\end{itemize}

\begin{theorem}[={\bf Theorem \ref{thm:All-genus-mirror}}) (All-genus descendant mirror theorem]  \label{thm:all-genus-mirror}
For any $g \in \bZ_{\ge 0}$, $n \in \bZ_{\ge 1}$, and $\cE_1, \dots, \cE_n \in K^+ _{\bT'}(\cX)$, we have
$$
Z^+_{g,n}(\cE_1, \dots, \cE_n) = (-1)^{g-1+n}  \chZ^+_{g,n}\left(\mir^+_{\bT'}(\cE_1), \dots, \mir^+_{\bT'}(\cE_n) \right) 
$$
under the mirror map $\btau = \btau(q)$.
\end{theorem}
Theorem \ref{thm:all-genus-mirror} can be viewed as an analog of \cite[Theorem B]{FLZ17} for the projective line $\bP^1$, \cite[Theorem 2]{Tang} for weighted projective lines, and \cite[Theorem 2]{Lan25} for footballs. Its proof relies on the graph sum formulas on both A-model (Theorem \ref{thm:Zong}) and B-model (Theorem \ref{thm:DOSS}). In \cite[Theorem 7.2]{flz2020remodeling}, the graph sum formulas on A-model and B-model are identified modulo a certain leaf factor, which is the genus-zero data. Therefore, the proof of Theorem \ref{thm:all-genus-mirror} is reduced to the case when $(g,n)=(0,1)$, which is 
Theorem \ref{thm:equivMir}.

\subsection{Future work}
We remark on the application of the results in this paper to the study of the Crepant Transformation Conjecture (CTC) for semi-projective toric Calabi-Yau $3$-orbifolds. The Crepant Transformation Conjecture, originally proposed by Y. Ruan \cite{Ruan02,Ruan06}, relates the Gromov-Witten theory of a pair of $K$-equivalent manifolds/orbifolds.
It was later refined and then extended to the higher-genus situation \cite{BryanGraber-crepant,CoatesIritaniTseng-crepant, Iritani09, CoatesRuan13-crepant}. 
Various versions of all-genus CTC of toric Calabi-Yau 3-orbifolds have been proved for the minimal resolution $\cA_n\times \bC\to [\bC^2/\bmu_{n+1}]\times \bC$  \cite{Zhou08, Brini-Cavalieri-Ross}, toric Calabi-Yau 3-folds with transverse $A$-singularities \cite{Ross15},  
and the crepant resolution $K_{\bP^2}\to [\bC^3/\bmu_3]$ \cite{Lho-Pandharipande, CI21}.

Let $\cX_\pm$ be a pair of $K$-equivalent toric Calabi-Yau $3$-orbifolds. They are different phases over a \emph{global} stringy K\"ahler moduli, and correspond to two top-dimensional cones (GIT chambers) of their common secondary fan. In the B-model, we have a family of mirror curves $\mathcal C$ over the toric variety $\mathcal M_B$ defined by this secondary fan, and the two top-dimensional cones define two points $\fs_\pm$ in $\mathcal M_B$. The mirror curves $C_{q_\pm}$ of $\cX_\pm$ are the fibers of $\mathcal C$ near the limit points $\fs_\pm$ respectively. An analytic continuation from $C_{q_-}$ to $C_{q_+}$ identifies their integral homology by the Gauss-Manin connection, depending on the path. As the A-model counterpart, the integral structures are identified by the Fourier-Mukai transform on the $K$-groups, where note that the support condition is preserved by $K$-equivalence (cf. \cite{BH15} for the case of compact support). We have the following diagram where all arrows are isomorphisms of abelian groups:
\[
    \xymatrix@C+3pc{
        K^+_{\bT'}(\cX_-) \ar[d]^-{\mir^+_{\bT'}} \ar[r]^-{\text{Fourier-Mukai}} & K^+_{\bT'} (\cX_+) \ar[d]^{\mir_{\bT'}^+} \\
		H_1(\tC_{q_-}, \Re (\hx) \gg 0; \bZ) \ar[r]^-{\text{Gauss-Manin}} & H_1(\tC_{q_+}, \Re (\hx) \gg 0; \bZ).
    }
\]
This implies the transformation of genus-zero and higher-genus Gromov-Witten invariants under the Crepant Transformation Conjecture as the following.
\begin{itemize}
\item 
    By Theorem \ref{thm:equivMir}, the quantum cohomology central charge $Z_{\bT'}(\cE)$ is identified with $Z_{\bT'}(\FM(\cE))$ after analytic continuation, for $\cE \in K^+_{\bT'}(\cX_-)$. In particular, the symplectic transformation on the state space (cf. \cite{CIJ18}) in CTC takes $\kappa_z^{\bT'}(\cE)$ to $\kappa_z^{\bT'}(\FM(\cE))$.

\item The modular transform in the higher-genus CTC naturally arises in the analytic continuation of the recursion kernel $\omega_{0,2}$ (see Section \ref{sec:eynard-orantin}) from $C_{q_-}$ to $C_{q_+}$. Thus $Z^+_{g,n}(\cE_1,\dots,\cE_n)$ is related to $Z^+_{g,n}(\FM(\cE_1),\dots,\FM(\cE_n))$ under the analytic continuation and a modular transform, for $\cE_i \in K^+_{\bT'}(\cX_-)$. 
\end{itemize}
We will investigate the all-genus primary, descendant, and open Crepant Transformation Conjecture in the forthcoming work \cite{FLYZ-crepant}.

\subsection{Outline of the paper}
In Section \ref{sect:Prelim}, we study the geometry of semi-projective toric Calabi-Yau 3-orbifolds and their equivariant Chen-Ruan cohomology. In Section \ref{sect:GW}, we study the equivariant Gromov-Witten theory of toric Calabi-Yau 3-orbifolds and give a graph sum formula (Theorem \ref{thm:Zong}). 
In Section \ref{sect:Sheaves}, we study the $K$-theory with different support conditions and compute quantum cohomology central charges.
In Section \ref{sect:Integrals}, we compute oscillatory integrals on the mirror curve and construct the mirror cycle homomorphism $\mir_{\bT'}^+$. In Section \ref{sect:GenusZero}, we prove Theorems \ref{thm:equivMir}=\ref{thm:EquivMir} and \ref{thm:kc3Mir}=\ref{thm:Kc3Mir}.  In Section \ref{sec:allgenus}, we prove Theorem \ref{thm:all-genus-mirror}=\ref{thm:All-genus-mirror}.

\subsection{Acknowledgments}

The authors would like to thank Simons Center for Geometry and Physics (2023 Simons Math Summer Workshop and Recent Developments in Higher Genus Curve Counting in 2025), and Les Houches School of Physics (Quantum Geometry in August 2024) for their support and hospitality, where important progress of this work was made during these programs.

The work of the first author is partially supported by National Key R\&D Program of China 2023YFA1009803, NSFC 12125101, NSFC 11890661 and NSFC 11831017. The work of the fourth author is partially supported by the Natural Science Foundation of Beijing, China grant No. 1252008 and NSFC grant No. 11701315.

\section{Geometry of toric Calabi-Yau 3-orbifolds}\label{sect:Prelim}
In this section, we review the basic notions of semi-projective toric Calabi-Yau 3-orbifolds and set up notation. We work over $\bC$.

\subsection{Simplicial fan and associated toric variety}\label{sect:Fan}
Let $N \cong \bZ^3$ and $\Sigma$ be a finite simplicial fan in $N_{\bR} := N \otimes \bR$. For $d = 0, 1, 2, 3$, let $\Sigma(d)$ denote the set of $d$-dimensional cones in $\Sigma$. We further label the set of 1-cones as
$$
    \Sigma(1) = \{\rho_1, \dots, \rho_{3+\fp'}\}
$$
where $\fp' := |\Sigma(1)|-3$, and let $b_i \in N$ be the primitive integral vector on the ray $\rho_i$, i.e. $\rho_i \cap N = \bZ_{\ge 0} b_i$. We assume that $\Sigma$ is 3-dimensional and every cone in $\Sigma$ is contained in some 3-cone.

Let $X$ be the 3-dimensional simplicial toric variety associated to $\Sigma$, which admits the action of the algebraic torus $\bT := N \otimes \bC^* \cong (\bC^*)^3$. The lattice $N$ is canonically identified with the cocharacter lattice $\Hom(\bC^*, \bT)$ of $\bT$, and the dual $M := \Hom(N, \bZ)$ is canonically identified with the character lattice $\Hom(\bT, \bC^*)$. In this paper, we assume the following:
\begin{itemize}
    \item $X$ is Calabi-Yau: the canonical bundle $K_X$ is trivial;
    
    \item $X$ is semi-projective: it is projective over its affinization $X_0 = \Spec(H^0(X, \cO_X))$.
\end{itemize}

The Calabi-Yau condition is equivalent to the existence of $e_3^\vee \in M$ such that $\inner{e_3^\vee, b_i} = 1$ for all $i = 1, \dots, 3+\fp'$. Let $N' := \ker (e_3^\vee: N \to \bZ) \cong \bZ^2$ and $\bT' := N' \otimes \bC^* \cong (\bC^*)^2$ be the Calabi-Yau 2-subtorus of $\bT$. The dual lattice $M' := \Hom(N', \bZ)$ is canonically identified with the character lattice $\Hom(\bT', \bC^*)$ of $\bT'$.

We complete $e_3^\vee$ into a $\bZ$-basis $\{e_1^\vee, e_2^\vee, e_3^\vee\}$ of $M$ and let $\{e_1, e_2, e_3\}$ be the dual $\bZ$-basis of $N$. Then the coordinate of $b_i$ is $(m_i, n_i, 1)$ for some $m_i, n_i \in \bZ$. Let $P \subset N'_{\bR} := N' \otimes \bR$ be the convex hull of $\{(m_i, n_i) : i = 1, \dots, 3+ \fp'\}$. Let $C \subset N_{\bR}$ be the cone over $P \times \{1\}$, which is a 3-dimensional finite strongly convex polyhedral cone. The semi-projectivity condition implies that $C$ is the support of $\Sigma$. In particular, $H^0(X, \cO_X) = \bC[C^\vee \cap M]$, where $C^\vee \subset M_{\bR} := M \otimes \bR$ is the dual cone of $C$, and thus $X_0$ is defined by $C$.

\subsection{Stacky fan and orbifold structure}
Let $\bSi^\can = (N, \Sigma, (b_1, \dots, b_{3+\fp'}))$ be the \emph{canonical stacky fan} associated to the simplicial fan $\Sigma$ and $\cX$ be the associated smooth toric Deligne-Mumford stack \cite{BCS05,FMN10}. In our case, $\cX$ is a toric orbifold, i.e. has trivial generic stabilizers. Moreover, $\cX$ is Calabi-Yau and has semi-projective coarse moduli space $X$.

We will use the alternative description of $\cX$ by the \emph{extended stacky fan} \cite{Jiang08} $\bSi^\ext = (N, \Sigma, (b_1, \dots, b_{3+\fp}))$. Here, $\fp := |P \cap N'| -3$ and we take the additional vectors $b_i = (m_i, n_i, 1) \in N$, $i = 4+\fp', \dots, 3+\fp$ such that
$$
    P \cap N' = \{(m_i, n_i) : i = 1, \dots, 3 + \fp\}.
$$
The extended list of vectors defines a surjective homomorphism
$$
    \phi: \tN := \bigoplus_{i=1}^{3+\fp} \bZ \tb_i \to N, \qquad \tb_i \mapsto b_i.
$$
Let $\bL := \Ker(\phi) \cong \bZ^{\fp}$, which fits into the short exact sequence
\begin{equation}\label{eqn:NExactSeq}
    \xymatrix{
        0 \ar[r] & \bL \ar[r]^\psi & \tN \ar[r]^\phi & N \ar[r] & 0.
    }
\end{equation}
Applying $-\otimes \bC^*$ to the above, we obtain a short exact sequence
\begin{equation}\label{eqn:TExactSeq}
    \xymatrix{
        1 \ar[r] & G \ar[r]^\psi & \tbT \ar[r]^\phi & \bT \ar[r] & 1.
    }
\end{equation}
Now consider $\bC^{3+\fp} = \Spec(\bC[Z_1, \dots, Z_{3+\fp}])$ with the standard action of $\tbT \cong (\bC^*)^{3+\fp}$ and the action of $G \cong (\bC^*)^{\fp}$ via the inclusion above. In the coordinate subspace $\bC^{3+\fp'} = \Spec(\bC[Z_1, \dots, Z_{3+\fp'}])$, let $Z(\Sigma)$ be closed subvariety defined by the ideal generated by monomials of form $\prod_{\rho_i \not \subseteq \sigma} Z_i$ for $\sigma \in \Sigma$. The subset
$$
    U := (\bC^{3+\fp'} \setminus Z(\Sigma)) \times (\bC^*)^{\fp - \fp'}
$$
is dense, open in $\bC^{3+\fp}$, and $\tbT$-invariant. Then
$$
    X = U/G, \qquad \cX = [U/G].
$$
The semi-projectivity assumption implies that the above are GIT quotients. Moreover, $\cX$ may be presented as a symplectic quotient under the Hamiltonian action of the maximal compact torus of $G$ on $\bC^{3+\fp}$. This endows $\cX$ with the structure of a K\"ahler orbifold. We refer to e.g. \cite[Section 2.9]{flz2020remodeling} for details.

\subsection{Cones, toric closed substacks, and stabilizers}\label{sect:Stabilizers}
Let $\sigma \in \Sigma(d)$. We denote
$$
    I'_\sigma := \{ i \in \{1, \dots, 3+\fp'\} : \rho_i \subseteq \sigma\}, \qquad I_\sigma := \{1, \dots, 3+\fp\} \setminus I'_\sigma.
$$
Then $|I'_\sigma| = d$. Let $\cV(\sigma)$ denote the $\bT$-invariant closed substack of $\cX$, which has codimension $d$. We note that the action of the Calabi-Yau subtorus $\bT'$ has the same $0$- and $1$-dimensional orbits as that of $\bT$. Depending on $d$, we will adopt the following notation:
\begin{itemize}
    \item $\sigma \in \Sigma(3)$: $\cV(\sigma)$ is a $\bT$-fixed point and is denoted by $\fp_\sigma$;
    
    \item $\sigma \in \Sigma(2)$: $\cV(\sigma)$ is a $\bT$-invariant line and is denoted by $\fl_\sigma$;
    
    \item $\sigma = \rho_i \in \Sigma(1)$: $\cV(\rho_i)$ is a $\bT$-divisor and is denoted by $\cD_i$.

\end{itemize}
We use $\iota_\sigma$ to denote the inclusion $\cV(\sigma) \to \cX$, which is $\bT$-equivariant.

Now suppose $d = 3$ or $2$. Let $N_\sigma$ be the rank-$d$ sublattice of $N$ spanned by $\{b_i : i \in I'_\sigma\}$, and let $N'_\sigma := N_\sigma \cap N'$, which has rank $d-1$. Consider the dual lattices $M_\sigma := \Hom(N_\sigma, \bZ)$, $M'_\sigma := \Hom(N'_\sigma, \bZ)$ which have ranks $d$, $d-1$ respectively. When $d = 3$, $M_\sigma$ is an overlattice of $M$ in $M_{\bQ} := M \otimes \bQ$ and $M'_\sigma$ is an overlattice of $M'$ in $M'_{\bQ} := M' \otimes \bQ$. For an inclusion of cones $\tau \subseteq \sigma$, we have an inclusion $N'_\tau \subseteq N'_\sigma$ and a projection $M'_\sigma \to M'_\tau$.

For $\sigma \in \Sigma$, let $G_\sigma$ denote the group of generic stabilizers of $\cV(\sigma)$, which is a finite subgroup of $G$, and let $G^*_\sigma := \Hom(G_\sigma, \bC^*)$ denote the dual group. For an inclusion of cones $\tau \subseteq \sigma$, we have an inclusion $G_\tau \subseteq G_\sigma$ and a projection $G^*_\sigma \to G^*_\tau$. When $\sigma \in \Sigma(3)$ is maximal, we have identifications
$$
    G_\sigma \cong N/N_\sigma \cong N'/N'_\sigma , \qquad G_\sigma^* \cong M_\sigma / M \cong M'_\sigma / M'.
$$
In particular, $G_\sigma$ is in bijection with the set of representatives
$$
    \Box(\sigma) := \left\{v \in N : v = \sum_{i \in I'_\sigma} c_i(v) b_i \text{ for some } c_i(v) \in [0,1) \cap \bQ \right\}
$$
of $N/N_\sigma$. The above equation defines the quantities $c_i(v)$; note that $c_i(v) = 0$ for all $i \in I_\sigma$. The definition of $\Box(\sigma)$ is valid when $\sigma$ is not necessarily maximal. If $\tau\subseteq \sigma$ then $\Box(\tau)\subset\Box(\sigma)$.

\subsection{Flags}\label{sect:Flags}
Let $F(\Sigma)$ be the set of \emph{flags} in $\Sigma$. A flag is a pair $(\tau, \sigma)$ where $\tau \in \Si(2)$, $\si \in \Si(3)$, and $\tau \subset \si$. It corresponds to the inclusion of the fixed point $\fp_\sigma$ in the line $\fl_\tau$. For stabilizers, there is a short exact sequence of finite abelian groups
$$
    \xymatrix{
           1 \ar[r] & G_\tau \ar[r] & G_\sigma \ar[r] & \bmu_{\fr_{(\tau,\si)}} \ar[r] & 1
    }
$$
where $G_\tau \cong \bmu_{\fm_\tau}$ is a cyclic subgroup of order $\fm_\tau := |G_\tau|$, and the quotient is also a cyclic group of order $\fr_{(\tau,\si)} := \frac{|G_\sigma|}{|G_\tau|}$.

A flag $\bff = (\tau,\si)$ determines an ordered triple
$(i_1 = i_1^\bff, i_2 = i_2^\bff, i_3 = i_3^\bff)$, where $i_1,i_2, i_3\in \{1,\ldots, 3+\fp'\}$, characterized by the following three conditions: (i) $I_\si' = \{ i_1, i_2, i_3\}$, (ii) $I_\tau'=\{ i_2, i_3\}$, and (iii)
 $(m_{i_1}, n_{i_1})$, $(m_{i_2}, n_{i_2})$, $(m_{i_3}, n_{i_3})$ are vertices of a triangle $P_\si \subseteq P$ in the counterclockwise order.
 Note that the 3-cone $\si$ is the cone over the triangle $P_\si\times \{1\}$.
 There exists an ordered $\bZ$-basis $\{e_1^{\bff}, e_2^{\bff}, e_3^{\bff}\}$ of $N$ such that
$$
b_{i_1^{\bff}} = \fr_{\bff} e_1^{\bff}   -\fs_{\bff}e_2^{\bff}  + e_3^{\bff} ,\quad
b_{i_2^{\bff}} = \fm_\tau e_2^{\bff}   + e_3^{\bff},\quad
b_{i_3^{\bff}} = e_3^{\bff}
$$
where $\fs_{\bff} \in \{0,1,\ldots, \fr_{\bff}-1\}$. Then
$\{ e_1^{\bff}, e_2^{\bff} \}$ is a $\bZ$-basis of $\bZ e_1\oplus \bZ e_2$ and $e_1^{\bff}\wedge
e_2^{\bff} = e_1\wedge e_2$. Let $\{e_1^{\bff\vee}, e_2^{\bff\vee}, e_3^{\vee}\}$ be the dual $\bZ$-basis of $M$.

In this paper, we fix a reference flag $\bff_0 = (\tau_0,\si_0)$. Assume without loss of generality that $i_1^{\bff_0} = 1$, $i_2^{\bff_0} = 2$, $i_3^{\bff_0} = 3$. Let
$$
\fr = \fr_{\bff_0}, \qquad \fs = \fs_{\bff_0}, \qquad \fm=\fm_{\tau_0}.
$$
We assume that the $\bZ$-basis $\{e_1^\vee, e_2^\vee, e_3^\vee\}$ of $M$ in Section \ref{sect:Fan} is chosen such that $e_1^\vee = e_1^{\bff_0\vee}$, $e_2^\vee = e_2^{\bff_0\vee}$. We have
$$
    b_1 = \fr e_1 - \fs e_2 + e_3 ,\quad
b_2 = \fm e_2   + e_3, \quad
b_3 = e_3.
$$
We use the notation $\su_1, \su_2, \su_3$ for $e_1^\vee, e_2^\vee, e_3^\vee$ respectively when viewed as characters of $\bT$. We also use $\su_1, \su_2$ to denote the induced characters of $\bT'$, viewed as elements in $M'$. Then
$$
    M' = \bZ\su_1 \oplus \bZ\su_2, \qquad H^*_{\bT'}(\pt) = \bZ[\su_1, \su_2].
$$

\subsection{Character lattices and equivariant line bundles}\label{sect:LineBundles}
Applying $\Hom(-, \bZ)$ to \eqref{eqn:NExactSeq} gives the short exact sequence of dual lattices
$$    
    \xymatrix{
        0 \ar[r] & M \ar[r]^{\phi^\vee} & \tM \ar[r]^{\psi^\vee} & \bL^\vee \ar[r] & 0
    }
$$
which are the character lattices of the tori in \eqref{eqn:TExactSeq}. We have the identifications
$$
    \tM \cong H^2_{\tbT}(\bC^{3+\fp}; \bZ), \qquad \bL^\vee \cong H^2_G(\bC^{3+\fp}; \bZ).
$$
For $i = 1, \dots, 3+ \fp$, let $D_i^{\bT} \in \tM$ (resp. $D_i \in \bL^\vee$) be the $\tbT$-equivariant (resp. $G$-equivariant) Poincar\'e dual of the coordinate hyperplane $\{Z_i = 0\}$ in $\bC^{3+\fp}$. The $\bZ$-basis $\{D_1^{\bT}, \dots, D_{3+\fp}^{\bT}\}$ of $\tM$ is dual to the $\bZ$-basis $\{\tb_1, \dots, \tb_{3+\fp}\}$ of $\tN$. We have $\psi^\vee(D_i^{\bT}) = D_i$ for all $i$. The inclusion of the open subset $U \subseteq \bC^{3+\fp}$ induces surjective homomorphisms
$$
    \kappa_{\bT}: \tM \cong H^2_{\tbT}(\bC^{3+\fp}; \bZ) \to H^2_{\tbT}(U; \bZ) \cong H^2_{\bT}(\cX; \bZ), \quad
    \kappa: \bL^\vee \cong H^2_G(\bC^{3+\fp}; \bZ) \to H^2_G(U; \bZ) \cong H^2(\cX; \bZ)
$$
whose kernels are $\bigoplus_{i = 4+\fp'}^{3+ \fp} \bZ D_i^{\bT}$, $\bigoplus_{i = 4+\fp'}^{3+ \fp} \bZ D_i$ respectively. We have the following commutative diagram where the rows and columns are exact:
$$
    \xymatrix{
        & & 0 \ar[d] & 0 \ar[d] & \\
        & & \displaystyle \bigoplus_{i = 4+\fp'}^{3+ \fp} \bZ D_i^{\bT} \ar[r]^{\cong} \ar[d] & \displaystyle \bigoplus_{i = 4+\fp'}^{3+ \fp} \bZ D_i \ar[d] & \\
        0 \ar[r] & M \ar[r]^{\phi^\vee} \ar[d]^{=} & \tM \ar[r]^{\psi^\vee} \ar[d]^{\kappa_{\bT}} & \bL^\vee \ar[r] \ar[d]^{\kappa} & 0\\
        0 \ar[r] & M \ar[r] & H^2_{\bT}(\cX; \bZ) \ar[r] \ar[d] & H^2(\cX; \bZ) \ar[r] \ar[d] & 0\\
        & & 0 & 0 &
    }
$$

To pass to the $\bT'$-equivariant setting, consider the following commutative diagram where the rows and columns are exact:
\begin{equation}\label{eqn:TPrimeDiagram}
    \xymatrix{
        & 0 \ar[d] & 0 \ar[d] & & \\
        & \bZ e_3^\vee \ar[r]^{\cong \phantom{opera}} \ar[d] & \bZ \left( \sum_{i=1}^{3+\fp'} \Dbar_i^{\bT}\right) \ar[d] & & \\
        0 \ar[r] & M \ar[r] \ar[d] & H^2_{\bT}(\cX; \bZ) \ar[r] \ar[d] & H^2(\cX; \bZ) \ar[r] \ar[d]^{=} & 0\\
        0 \ar[r] & M' \ar[d] \ar[r] & H^2_{\bT'}(\cX; \bZ) \ar[r] \ar[d] & H^2(\cX; \bZ) \ar[r] & 0\\
        & 0 & 0 & &
    }
\end{equation}
where $e_3^\vee \in M$ is specified by the Calabi-Yau condition and $\Dbar_i^{\bT} := \kappa_{\bT}(D_i^{\bT})$. Let $\Dbar_i^{\bT'} \in H^2_{\bT'}(\cX; \bZ)$, $\Dbar_i \in H^2(\cX; \bZ)$ denote the images of $\Dbar_i^{\bT}$. The first Chern class homomorphisms
$$
    c_1^{\bT'}: \Pic_{\bT'}(\cX) \to H^2_{\bT'}(\cX; \bZ), \qquad c_1: \Pic(\cX) \to H^2(\cX; \bZ)
$$
are isomorphisms, where $c_1^{\bT'}$ identifies $\Dbar_i^{\bT'}$ with $\cO_{\cX}(\cD_i)$ viewed as a $\bT'$-equivariant line bundle. Any $\bT'$-equivariant line bundle on $\cX$ can be written as
$$
    \cO_{\cX}\left(\sum_{i = 1}^{3+\fp'} r_i \cD_i \right), \qquad r_i \in \bZ
$$
and two tuples $(r_i)$, $(r_i') \in \bZ^{3+\fp'}$ represent the same $\bT'$-equivariant line bundle if and only if
$$
    r_1 - r_1' = \cdots = r_{\fp'} - r_{\fp'}'.
$$

For the closed substack $\cV(\sigma)$, $\sigma \in \Sigma$, the restriction gives surjective homomorphisms
$$
    \Pic_{\bT'}(\cX) \to \Pic_{\bT'}(\cV(\sigma)), \qquad \Pic(\cX) \to \Pic(\cV(\sigma)).
$$
In Section \ref{sect:Integrals}, we will also use a description of $\bT'$-equivariant line bundles on $\cV(\sigma)$ in terms of \emph{twisted polytopes} \cite{fltz14}, which is the data of the $\bT'$-weights of the line bundles at the torus fixed points. For $i = 1, \dots, 3+\fp'$ and $\sigma \in \Sigma(3)$, denote
$$
    \sw_{i, \sigma} := \Dbar_i^{\bT'} \big|_{\fp_\sigma} \qquad \in M'_\sigma \subset M'_{\bQ}
$$
which is non-zero if and only if $i \in I_\sigma'$.

On the other hand, for $\sigma \in \Sigma(3)$, consider the associated open substack $\cX_\sigma = [\bC^3/G_\sigma]$ of $\cX$ which is an affine toric Calabi-Yau 3-orbifold. Its fan consists of the single 3-cone $\sigma$ and its coarse moduli space $X_\sigma$ is the affine simplicial toric variety $\Spec(\bC[\sigma^\vee \cap M]) \cong \bC^3/G_\sigma$. We have identifications
$$
    H^2_{\bT'}(\cX_\sigma; \bZ) \cong M'_\sigma, \qquad G^*_\sigma \cong \Pic(\cX_\sigma) \cong H^2(\cX_\sigma; \bZ)
$$
and the third row of \eqref{eqn:TPrimeDiagram} applied to $\cX_\sigma$ reads
\begin{equation}\label{eqn:TPrimeAffine}
    \xymatrix{
        0 \ar[r] &  M' \ar[r] &  M'_\si \ar[r] & G_\si^* \ar[r] & 0.     
    }
\end{equation}

\subsection{Extended nef and Mori cones}
Let $\sigma \in \Sigma(3)$ be a maximal cone. Let
$
    \bK_\sigma^\vee := \bigoplus_{i \in I_\sigma} \bZ D_i
$
which is a sublattice of $\bL^\vee$ of finite index. The \emph{extended $\sigma$-nef cone} is
$$
    \tNef_\sigma := \sum_{i \in I_\sigma} \bR_{\ge 0} D_i
$$
which is a top-dimensional cone in $\bL^\vee_{\bR} := \bL^\vee \otimes \bR$. Dually, denote $\bL_{\bQ} := \bL \otimes \bQ$, $\bL_{\bR} := \bL \otimes \bR$. Let
$
    \bK_\sigma := \{\beta \in \bL_{\bQ} : \inner{D, \beta} \in \bZ \text{ for all } D \in \bK_\sigma^\vee\}
$
be the dual lattice of $\bK_\sigma^\vee$, which is an overlattice of $\bL$ in $\bL_{\bQ}$. The \emph{extended $\sigma$-Mori cone} is
$$
    \tNE_\sigma := \{\beta \in \bL_{\bR} : \inner{D, \beta} \ge 0 \text{ for all } D \in \tNef_\sigma\} 
$$
which is the dual cone of $\tNef_\sigma$. We set
$
    \bK_{\eff, \sigma} := \bK_\sigma \cap \tNE_\sigma.
$

Taking all maximal cones into account, we define the \emph{extended nef cone} of $\cX$ to be
$$
    \tNef(\cX) := \bigcap_{\sigma \in \Sigma(3)} \tNef_\sigma
$$
and the \emph{extended Mori cone} of $\cX$ to be
$$
    \tNE(\cX) := \bigcup_{\sigma \in \Sigma(3)} \tNE_\sigma.
$$
We further set
$$
    \bK := \bigcup_{\sigma \in \Sigma(3)} \bK_\sigma, \qquad \bK_{\eff} := \bK \cap \tNE(\cX) = \bigcup_{\sigma \in \Sigma(3)} \bK_{\eff, \sigma}.
$$

For $\sigma \in \Sigma(3)$, we have an identification $G_\sigma\cong \bK_\sigma / \bL$, and a bijection $\bK_\sigma / \bL \to \Box(\sigma)$ is induced by the map
$$
    v: \bK_\sigma \to \Box(\sigma), \qquad \beta \mapsto \sum_{i \in I'_\sigma}  \{-\inner{D_i, \beta}\} b_i.
$$
Here for $x \in \bR$, $\{x\} = x - \floor{x}$ denotes the fractional part of $x$. In particular, $c_i(v(\beta)) = \{-\inner{D_i, \beta}\}$. This bijection restricts to a bijection $G_\tau \to \Box(\tau)$ for any face $\tau \subseteq \sigma$.

\subsection{The inertia stack}
Let
$$
    \Box(\cX) := \bigcup_{\sigma \in \Sigma} \Box(\sigma) = \bigcup_{\sigma \in \Sigma(3)} \Box(\sigma).
$$
This set indexes the connected components of the inertia stack $\cI\cX$ of $\cX$. Let $\cX_v$ denote the component indexed by $v \in \Box(\cX)$, which is a $\bT$-invariant closed substack of $\cX$. Let
$$
    \inv: \cI \cX \to \cI \cX, \qquad \inv: \Box(\cX) \to \Box(\cX)
$$
denote the natural involutions on $\cI \cX$ and the index set $\Box(\cX)$.

The dimension of the inertia component $\cX_v$ is
$$
    \dim \cX_v = 3 - |\{i \in \{1, \dots, 3+\fp'\} : c_i(v) \neq 0\}|.
$$
The \emph{age} of $\cX_v$ is defined to be
$$
    \age(v) := \sum_{i = 1}^{3+\fp'} c_i(v) \quad \in \{0, 1, 2\}.
$$
For $v = 0$, $\cX_v = \cX$ is referred to as the \emph{untwisted sector} and has age $0$. The other components are refereed to as \emph{twisted sectors} and have positive age.

For any line bundle $\cL \in \Pic(\cX)$ and any $v \in \Box(\cX)$, the age of $\cL$ along $\cX_v$ is the number $\age_v(\cL) \in [0,1) \cap \bQ$ such that if $x$ is any point in $\cX_v$ and $k \in \Aut(x)$ corresponds to $v$, then $k$ acts on the fiber $\cL_x$ with eigenvalue $\exp(2\pi\sqrt{-1}\age_v(\cL))$. More explicitly, if we write $\cL = \cO_{\cX}\left(\sum_{i = 1}^{3+\fp'} r_i \cD_i \right)$, $r_i \in \bZ$, then for $v \in \Box(\cX)$ we have
$$
    \age_v(\cL) = \left\{ \sum_{i = 1}^{3+\fp'} r_ic_i(v) \right\}.
$$
If $\cL$ is pulled back from the coarse moduli space $X$ of $\cX$, then $\age_v(\cL) = 0$ for any $v$. Thus the age only depends on the class of $\cL$ in the \emph{stacky Picard group}
$$
    \Picst(\cX) := \Pic(\cX) / \Pic(X).
$$

\subsection{Chen-Ruan orbifold cohomology}
Let $\bF = \bQ$, $\bR$, or $\bC$. The \emph{Chen-Ruan orbifold cohomology} \cite{CR02} of $\cX$ with coefficients in $\bF$ is
$$
    H^*_{\CR}(\cX; \bF) = \bigoplus_{v \in \Box(\cX)} H^*(\cX_v; \bF)[2\age(v)]
$$
where $[2\age(v)]$ denotes the degree shift by $2\age(v)$. For $v \in \Box(\cX)$, let $\one_v$ denote the unit of $H^*(\cX_v; \bF)$, which has degree $2\age(v)$ in $H^*_{\CR}(\cX; \bF)$.

The dimensions of the homogeneous components of $H^*_{\CR}(\cX; \bF)$ can be retrieved from the combinatorial data of the polytope $P$. Let
$$
    \fg := |\Int(P) \cap N'|, \qquad \fn := | \partial P \cap N'|
$$
denote the number of lattice points in the interior $\Int(P)$ and on the boundary $\partial P$ of $P$ respectively. Note that
$$
    \fg + \fn = |P \cap N'| = 3 + \fp.
$$
Then we have
\begin{align*}
    & \dim_{\bF} H^0_{\CR}(\cX; \bF)  = 1, && \dim_{\bF} H^2_{\CR}(\cX; \bF) = \fp = \fg + \fn - 3,\\
    & \dim_{\bF} H^4_{\CR}(\cX; \bF) = \fg, && \dim_{\bF} H^*_{\CR}(\cX; \bF) = 1 + \fp + \fg = 2\fg - 2 + \fn.
\end{align*}

In the $\bT'$-equivariant setting, let
$$
    R_{\bT'} := H^*_{\bT'}(\pt; \bC) = \bC[\su_1, \su_2], \qquad
    S_{\bT'} = \bC(\su_1, \su_2).
$$
Over $\bF = \bC$, the $\bT'$-equivariant Chen-Ruan cohomology of $\cX$ is
$$
    H^*_{\CR, \bT'}(\cX; \bC) = \bigoplus_{v \in \Box(\cX)} H^*_{\bT'}(\cX_v; \bC)[2\age(v)].
$$
Let $\star_{\cX}$ denote the $\bT'$-equivariant Chen-Ruan orbifold cup product (which note is not the component-wise multiplication). Let
$$
    \left(\phi, \phi'\right)_{\cX, \bT'} := \int_{\cI\cX} \inv^*(\phi) \phi'
$$
denote the $\bT'$-equivariant orbifold Poincar\'e pairing, where the integral is defined via $\bT'$-equivariant localization and will be described in detail in the following subsection. The pairing is non-degenerate over the field $S_{\bT'}$. Then $\star_{\cX}$ and $\left(-, -\right)_{\cX, \bT'}$ together make $H^*_{\CR, \bT'}(\cX; \bC) \otimes_{R_{\bT'}} S_{\bT'}$ into a Frobenius algebra over $S_{\bT'}$.

\subsection{Semi-simplicity}
For $\sigma \in \Sigma(3)$, consider the open immersion
$$
    \iota_{\sigma, o}: \cX_\sigma \to \cX.
$$
Recall the bijection $G_\sigma \to \Box(\sigma) = \Box(\cX_\sigma)$. For $h \in G_\sigma$ we use $\age(h)$ to denote the age of the corresponding $\Box$ element. The inertia stack of $\cX_\sigma$ is
$$
    \cI\cX_\sigma = \bigcup_{h \in G_\sigma} \cX_h, \qquad \cX_h = [(\bC^3)^h/G_\sigma].
$$
Over $\bC$, the $\bT'$-equivariant Chen-Ruan cohomology of $\cX_\sigma$ is
$$
    H^*_{\CR, \bT'}(\cX_\sigma; \bC) = \bigoplus_{h \in G_\sigma} \bC \one_h
$$
where $\one_h$ has degree $2\age(h)$. The $\bT'$-equivariant orbifold cup product is given by
$$
    \one_h \star_{\cX_\sigma} \one_{h'} = \one_{hh'} \prod_{i \in I'_\sigma}^3 \sw_{i, \sigma}^{c_i(h) + c_i(h') - c_i(hh')}.
$$
The $\bT'$-equivariant orbifold Poincar\'e pairing is given by
$$
    \left(\one_h, \one_{h'}\right)_{\cX_\sigma, \bT'} = \frac{\delta_{hh', 1}}{|G_\sigma| e_{\bT'}(T_{\fp_\sigma}\cX_h)}
$$
where
$$
    e_{\bT'}(T_{\fp_\sigma}\cX_h) = \prod_{i \in I'_\sigma} \sw_{i, \sigma}^{\delta_{c_i(h), 0}}.
$$

Now, let $\bSt$ be the minimal field extension of $S_{\bT'}$ that contains
$$
    \left\{\sw_{i, \sigma}^{c_{i}(h)} : \sigma \in \Sigma(3), i \in I'_\sigma, h \in G_\sigma \right\} \cup \left\{\sqrt{e_{\bT'}(T_{\fp_\sigma}\cX)} = \prod_{i \in I'_\sigma} \sw_{i, \sigma}^{\frac{1}{2}} : \sigma \in \Sigma(3) \right\}.
$$
For $\sigma \in \Sigma(3)$ and $h \in G_\sigma$, we define
$$
    \bar{\one}_h := \frac{\one_h}{\prod_{i \in I'_\sigma}\sw_{i, \sigma}^{c_{i}(h)}} \qquad \in H^*_{\CR, \bT'}(\cX_\sigma; \bC) \otimes_{R_{\bT'}} \bSt.
$$
Moreover, for $\gamma \in G_\sigma^*$, writing $\chi_\gamma: G_\sigma \to \bC^*$ for the corresponding character of $G_\sigma$, we define
$$
    \bar{\phi}_\gamma := \frac{1}{|G_\sigma|} \sum_{h \in G_\sigma} \chi_\gamma(h^{-1}) \bar{\one}_h.
$$
Then we have
$$
    \bar{\phi}_\gamma \star_{\cX_\sigma} \bar{\phi}_{\gamma'} = \delta_{\gamma, \gamma'} \bar{\phi}_\gamma, \qquad \left(\bar{\phi}_\gamma, \bar{\phi}_{\gamma'}\right)_{\cX_\sigma, \bT'} = \frac{\delta_{\gamma, \gamma'}}{|G_\sigma|^2 \prod_{i \in I'_\sigma} \sw_{i, \sigma}}.
$$
Therefore, $\{\bar{\phi}_\gamma\}_{\gamma \in G_\sigma^*}$ is a canonical basis of the semi-simple Frobenius algebra
$$
    \left(H^*_{\CR, \bT'}(\cX_\sigma; \bC) \otimes_{R_{\bT'}} \bSt, \star_{\cX_\sigma}, \left(-, -\right)_{\cX_\sigma, \bT'} \right)
$$
over $\bSt$.

The pullbacks $\iota_{\sigma, o}^*$ to $\bT$-invariant affine charts $\cX_\sigma$ give an isomorphism of Frobenius algebras
$$
    \bigoplus_{\sigma \in \Sigma(3)} \iota_{\sigma, o}^*: H^*_{\CR, \bT'}(\cX; \bC) \otimes_{R_{\bT'}} \bSt \to \bigoplus_{\sigma \in \Sigma(3)} H^*_{\CR, \bT'}(\cX_\sigma; \bC) \otimes_{R_{\bT'}} \bSt.
$$
over $\bSt$. Introduce the index set
$$
    I_\Sigma := \{ \bsi = (\sigma, \gamma) : \sigma \in \Sigma(3), \gamma \in G_\sigma^* \}.
$$
For $\bsi = (\sigma, \gamma) \in I_\Sigma$, let $\phi_{\bsi}$ be the unique element in $H^*_{\CR, \bT'}(\cX; \bC) \otimes_{R_{\bT'}} \bSt$ such that $\phi_{\bsi} \big|_{\cX_\sigma} = \bar{\phi}_\gamma$ and $\phi_{\bsi} \big|_{\fp_{\sigma'}} = 0$ for all $\sigma' \in \Sigma(3) \setminus \{\sigma\}$. Then we have
$$
    \phi_{\bsi} \star_{\cX} \phi_{\bsi'} = \delta_{\bsi, \bsi'} \phi_{\bsi}, \qquad \left(\phi_{\bsi}, \phi_{\bsi'}\right)_{\cX, \bT'} = \frac{\delta_{\bsi, \bsi'}}{|G_\sigma|^2 \prod_{i \in I'_\sigma} \sw_{i, \sigma}}.
$$
Therefore, $\{\phi_{\bsi}\}_{\bsi \in I_\Sigma}$ is a canonical basis of the semi-simple Frobenius algebra
$$
    \left(H^*_{\CR, \bT'}(\cX; \bC) \otimes_{R_{\bT'}} \bSt, \star_{\cX_\sigma}, \left(-, -\right)_{\cX, \bT'} \right)
$$
over $\bSt$.

For $\bsi = (\sigma, \gamma) \in I_\Sigma$, we write $\Delta^\bsi := |G_\sigma|^2 \prod_{i \in I'_\sigma} \sw_{i, \sigma}$ and define
$$
    \hat{\phi}_\bsi := \sqrt{\Delta^{\bsi}} \phi_{\bsi}.
$$
Then we have
$$
    \hat{\phi}_{\bsi} \star_{\cX} \hat{\phi}_{\bsi'} = \delta_{\bsi, \bsi'} \sqrt{\Delta^{\bsi}} \hat{\phi}_{\bsi}, \qquad \left(\hat{\phi}_{\bsi}, \hat{\phi}_{\bsi'}\right)_{\cX, \bT'} = \delta_{\bsi, \bsi'}.
$$
We call $\{\hat{\phi}_{\bsi}\}_{\bsi \in I_\Sigma}$ the \emph{classical normalized canonical basis} of $H^*_{\CR, \bT'}(\cX; \bC) \otimes_{R_{\bT'}} \bSt$.

\section{Gromov-Witten theory}\label{sect:GW} 
In this section, we review the Gromov-Witten theory of the toric Calabi-Yau 3-orbifold $\cX$. We state the A-model graph sum formula for the all-genus equivariant descendant potentials. Moreover, we define the equivariant GKZ-type system associated to $\cX$ and characterize its solutions.

\subsection{Equivariant descendant Gromov-Witten invariants}\label{sect:genfun}
Recall that the coarse moduli space of $\cX$ is the simplicial toric variety $X$. Let $\NE(\cX) \subset H_2(\cX;\bR)=H_2(X;\bR)$ be the Mori cone generated by effective curve classes in $X$, and $E(\cX)$ denote the semigroup $\NE(\cX)\cap H_2(X;\bZ)$. For $g, n \in \bZ_{\ge 0}$ and $\beta \in E(\cX)$, let $\Mbar_{g,n}(\cX, \beta)$ be the moduli space of genus-$g$, $n$-pointed, degree-$\beta$ twisted stable maps to $\cX$. Let $\ev_i:\Mbar_{g,n}(\cX,\beta)\to \cI\cX$ be the evaluation map
at the $i$-th marked point. The $\bT'$-action on $\cX$ induces
$\bT'$-actions on the moduli space $\Mbar_{g,n}(\cX,\beta)$ and on the inertia stack $\cI\cX$, and
the evaluation map $\ev_i$ is $\bT'$-equivariant. Similarly, we can define the moduli space $\Mbar_{g,n}(X,\beta)$.

For $i=1,\dots,n$, let $\bL_i$ be the $i$-th tautological line bundle over $\Mbar_{g,n}(X, \beta)$ formed
by the cotangent line at the $i$-th marked point. Define the $i$-th descendant class $\psi_i$ as
$$
\psi_i := c_1(\bL_i)\in H^2(\Mbar_{g,n}(X, \beta);\bQ).
$$
The $\bT'$-action on $X$ induces a $\bT'$-action on
$\Mbar_{g,n}(X, \beta)$, and we choose a $\bT'$-equivariant
lift $\psi_i^{\bT'}\in H^2_{\bT'}(\Mbar_{g.n}(X, \beta);\bQ)$
of $\psi_i$. Consider the map $p:\Mbar_{g,n}(\cX, \beta)\to \Mbar_{g,n}(X, \beta)$ induced by $\cX\to X$, which is $\bT'$-equivariant. For $i=1,\dots,n$, let
$$
\hat{\psi}_i:=p^*\psi_i \in H^2(\Mbar_{g,n}(\cX, \beta);\bQ), \quad 
\hat{\psi}_i^{\bT'}:=p^*\psi_i^{\bT'} \in H^2_{\bT'}(\Mbar_{g,n}(\cX, \beta);\bQ).
$$

Given $\gamma_1,\dots, \gamma_n\in H_{\bT'}^*(\cX,\bC)$ and $a_1,\dots,a_n\in \bZ_{\ge 0}$, we define the genus-$g$, degree-$\beta$, $\bT'$-equivariant \emph{descendant Gromov-Witten invariant}
\begin{align*}
\langle \tau_{a_1}(\gamma_1), \dots, \tau_{a_n}(\gamma_n)\rangle^{\cX,\bT'}_{g,n,\beta} & = 
\langle \gamma_1\hat{\psi}^{a_1}, \dots,\gamma_n\hat{\psi}^{a_n}\rangle^{\cX,\bT'}_{g,n,\beta}\\
& := \int_{[\Mbar_{g,n}(\cX, \beta)^{\bT'}]^{w,\bT'}} \frac{\iota^*\big(\prod_{i=1}^n \ev_i^*(\gamma_i)(\hat{\psi}_i^{\bT'})^{a_i}\big)}{e_{\bT'}(N^\vir)}
\quad \in \ST=\bQ(\su_1,\su_2)
\end{align*}
where $\Mbar_{g,n}(\cX, \beta)^{\bT'}$ is the $\bT'$-fixed locus, $e_{\bT'}(N^\vir)$ is the $\bT'$-equivariant Euler class of the virtual normal bundle of $\Mbar_{g,n}(\cX,\beta)^{\bT'}$ in $\Mbar_{g,n}(\cX,\beta)$, $[\Mbar_{g,n}(\cX, \beta)^{\bT'}]^{w,\bT'}$ is the weighted virtual fundamental class, and $\iota: \Mbar_{g,n}(\cX, \beta)^{\bT'}\hookrightarrow \Mbar_{g,n}(\cX, \beta)$ is the inclusion map. The invariant
$$
\langle \tau_0(\gamma_1), \dots, \tau_0(\gamma_n)\rangle^{\cX,\bT'}_{g,n,\beta} = \langle \gamma_1,\dots, \gamma_n \rangle_{g,n,\beta}^{\cX,\bT'} 
$$
is the genus-$g$, degree-$\beta$ \emph{primary Gromov-Witten invariant}.

Define the Novikov ring
\[
\nov:=\widehat{\bC[E(\cX)]}= \left\{ \sum_{\beta \in E(\cX)} c_\beta Q^\beta: c_\beta\in \bC \right\}
\]
where $Q$ is the Novikov variable.
Given $\gamma_1,\dots,\gamma_n\in H^*_{\CR,\bT'}(\cX)\otimes_\RT \bST$ and $a_1,\dots, a_n\in \bZ_{\ge 0}$, define the following generating function
$$
	\left\llangle  \tau_{a_1}(\gamma_1), \dots, \tau_{a_n}(\gamma_n) \right\rrangle^{\cX,\bT'}_{g,n} = 
	\left\llangle  \gamma_1\hat{\psi}^{a_1},  \dots, \gamma_n\hat{\psi}^{a_n} \right\rrangle^{\cX,\bT'}_{g,n}
	:=\sum_{m \ge 0} \sum_{\beta \in E(\cX)}\frac{Q^\beta}{m!} \left\langle
	\gamma_1\hat{\psi}^{a_1}, \dots, \gamma_n\hat{\psi}^{a_n}, t^m \right\rangle^{\cX,\bT'}_{g,n+m,\beta}
$$
where $t\in H^*_{\CR,\bT'}(\cX)\otimes_{R_\bT}\bST$. We choose an $\ST$-basis $\{ T_i: i=0,1,\dots,E-1\}$ of $H^*_{\CR,\bT'}(\cX;\ST)$ such that
$$
T_0= \text{$\one$},\qquad T_a=\bar{D}_{3+a}^{\bT'} \quad \text{for  $a=1,\dots,\fp'$},\qquad
T_a= \text{$\one_{b_{3+a}}$} \quad \textup{for $a=\fp'+1,\dots, \fp$}
$$
and that for $i=\fp+1,\dots,E-1$, $T_i$ is of the form $T_aT_b$ for some $a,b\in \{1,\dots, \fp\}$.
Write $t=\sum_{i=0}^{E-1}\tau_i T_i$, and let $\tau'=(\tau_1,\dots, \tau_{\fp'})$,
$\tau''=(\tau_0, \tau_{\fp'+1},\dots, \tau_{E-1})$.
By the divisor equation,
$$
	\left\llangle  \gamma_1\hat{\psi}^{a_1},  \dots, \gamma_n\hat{\psi}^{a_n} \right\rrangle^{\cX,\bT'}_{g,n}
	=\sum_{m=0}^\infty \sum_{\beta \in E(\cX)}\frac{\tQ^\beta}{m!} \left\langle
	\gamma_1\hat{\psi}^{a_1}, \dots, \gamma_n\hat{\psi}^{a_n}, (t'')^m \right\rangle^{\cX,\bT'}_{g,n+m,d} \quad \in \bSTQ
$$
where $\tQ^\beta = Q^\beta \exp(\sum_{a=1}^{\fp'}\tau_a \langle T_a, \beta \rangle)$ and $t'' = \tau_0 T_0 + \sum_{i=\fp'+1}^{E-1}\tau_i T_i$. In particular, the restriction to $Q=1$ is well-defined.


For $i=1,\dots,n$, introduce formal variables
$$
\bu_i =\bu_i(z)= \sum_{a\ge 0}(u_i)_a z^a
$$
where $(u_i)_a \in H^*_{\CR,\bT'}(\cX;\bC)\otimes_\RT\bST$.
Define
$$
	\left\llangle \bu_1,\dots, \bu_n  \right\rrangle_{g,n}^{\cX,\bT'} =
	\left\llangle \bu_1(\hat{\psi}),\dots, \bu_n(\hat{\psi})  \right\rrangle_{g,n}^{\cX,\bT'}
	=\sum_{a_1,\dots,a_n\ge 0}
	\left\llangle (u_1)_{a_1}\hat{\psi}^{a_1}, \dots, (u_n)_{a_n}\hat{\psi}^{a_n} \right\rrangle_{g,n}^{\cX,\bT'}.
$$
Let $z_1,\dots,z_n$ be formal variables and $\gamma_1,\dots,\gamma_n\in H^*_{\CR,\bT'}(\cX)\otimes_\RT \bST$. Define
$$
\left\llangle \frac{\gamma_1}{z_1-\hat{\psi}},\dots, \frac{\gamma_n}{z_n-\hat{\psi}} \right\rrangle^{\cX,\bT'}_{g,n}
=\sum_{a_1,\dots,a_n\in \bZ_{\ge 0}}
\left\llangle \gamma_1 \hat{\psi}^{a_1}, \dots, \gamma_n \hat{\psi}^{a_n} \right\rrangle^{\cX,\bT'}_{g,n}\prod_{i=1}^n z_i^{-a_i-1}.
$$

\begin{convention} \rm{
In this paper, to simplify the expressions, we adopt the following conventions for the unstable terms in $\llangle \rrangle^{\cX,\bT'}_{0,n}$ corresponding to $\beta = 0$ and $m = 0$:
$$
	\left\langle \frac{b}{z-\hat{\psi}} \right\rangle^{\cX,\bT'}_{0,1,0} = z(\one,b)_{\cX,\bT'}, \quad \left\langle a,\frac{b}{z-\hat{\psi}} \right\rangle^{\cX,\bT'}_{0,2,0} = (a,b)_{\cX,\bT'}, \quad
	\left\langle \frac{a}{z_1-\hat{\psi}},\frac{b}{z_2-\hat{\psi}} \right\rangle^{\cX,\bT'}_{0,2,0} = \frac{(a,b)_{\cX,\bT'}}{z_1 + z_2}.
$$
}
\end{convention}

\subsection{Quantum cohomology and Frobenius manifolds}\label{sec:QH}
The $\bT'$-equivariant quantum cohomology $QH^*_{\bT'}(\cX)$ of $\cX$ is defined by its genus-zero primary Gromov-Witten invariants. More concretely, for any  $a,b,c\in H_{\CR,\bT'}^*(\cX;\bST)$, define the quantum product $\star_t$ by
$$	
	(a\star_t b,c)_{\cX,\bT'} = \left\llangle a,b,c \right\rrangle_{0,3}^{\cX,\bT'} \qquad \in \bSTQ.
$$

Moreover, we can define a formal Frobenius manifold as follows. Let
$$
\novT:= \bST\otimes_{\bC}\nov =\bST \formal{E(\cX)}.
$$
Then $H^*_{\CR,\bT'}(\cX;\novT)$ is a free $\novT$-module of rank $E$. Let
$$
H=\mathrm{Spec}(\novT [ t^{\bsi}:\bsi\in I_\Si])
$$
and $\hH$ be the formal completion of $H$ along the origin:
$$
\hH :=\mathrm{Spec}(\novT \formal{ t^{\bsi}:\bsi\in I_\Si }).
$$
Let $\cO_{\hH}$ be the structure sheaf on $\hH$ and $\cT_{\hH}$ be the tangent sheaf on
$\hH$.
Then $\cT_{\hH}$ is a sheaf of free $\cO_{\hH}$-modules of rank $E$.
Given an open set $U$ in $\hH$, we have
$$
\cT_{\hH}(U)  \cong \bigoplus_{\bsi\in I_\Si}\cO_{\hat{H}}(U) \frac{\partial}{\partial t^{\bsi}}.
$$
The quantum product and the $\bT'$-equivariant Poincar\'{e} pairing define the structure of a formal Frobenius manifold on $\hH$ over $\Lambda_{\cX}^{\bT'}$:
$$
\frac{\partial}{\partial t^{\bsi}} \star_t \frac{\partial}{\partial t^{\bsi'}}
=\sum_{\brho\in I_\Si} \left\llangle \hat{\phi}_{\bsi},\hat{\phi}_{\bsi'},\hat{\phi}_{\brho} \right\rrangle_{0,3}^{\cX,\bT'}
\frac{\partial}{\partial t^{\brho}}
\in \Gamma(\hat{H}, \cT_{\hat{H}}), \qquad
\left( \frac{\partial}{\partial t^{\bsi}},\frac{\partial}{\partial t^{\bsi'}}\right)_{\cX,\bT'} =\delta_{\bsi,\bsi'}.
$$

\subsection{Semi-simplicity and canonical coordinates}\label{sec:A-canonical}
The semi-simplicity of the classical cohomology $H^*_{\CR, \bT'}(\cX;\bC)\otimes_\RT \bST$ implies the semi-simplicity of the quantum cohomology $QH^*_{\bT'}(\cX)$. In fact, there exists a canonical basis $\{\phi_{\bsi}(t)\}_{\bsi\in I_\Si}$ of $QH^*_{\bT'}(\cX)$ characterized by the property that for all $\bsi\in I_\Si$,
$$
	\phi_{\bsi}(t)\to  \phi_{\bsi},\quad\mathrm{when }\quad t,\tQ\to 0.
$$
We define $\{\phi^{\bsi}(t)\}_{\bsi\in I_\Si}$ to be the basis dual to $\{\phi_{\bsi}(t)\}_{\bsi\in I_\Si}$ with respect to the metric $(-,-)_{\cX,\bT'}$.

The canonical coordinates $\{ u^{\bsi}=u^{\bsi}(t):\bsi\in I_\Si\}$ on the formal Frobenius
manifold $\hat{H}$ are characterized by
$$
	\frac{\partial}{\partial u^{\bsi}} = \phi_{\bsi}(t)
$$
up to additive constants in $\novT$. We choose canonical coordinates
such that they vanish when $\tQ=0,\hat{t}^\bsi=0,\bsi\in I_\Si$.

We define $\Delta^{\bsi}(t)$ by
$$
(\phi_{\bsi}(t), \phi_{\bsi'}(t))_{\cX,\bT'} =\frac{\delta_{\bsi,\bsi'}}{\Delta^{\bsi}(t)}
$$
and define the normalized canonical basis of $(\hat{H},\star_t)$ as
$$
\{ \hat{\phi}_{\bsi}(t):= \sqrt{\Delta^{\bsi}(t)}\phi_{\bsi}(t): \bsi\in I_\Si\}.
$$
They satisfy
$$
\hat{\phi}_{\bsi}(t)\star_t \hat{\phi}_{\bsi'}(t) =\delta_{\bsi, \bsi'}\sqrt{\Delta^{\bsi}(t)}\hat{\phi}_{\bsi}(t),\qquad
(\hat{\phi}_{\bsi}(t), \hat{\phi}_{\bsi'}(t))_{\cX,\bT'}=\delta_{\bsi,\bsi'}.
$$
Let $\Psi=(\Psi_{\bsi'}^{\spa  \bsi})$ be the transition matrix defined as 
$$
	\hat{\phi}_{\bsi'}=\sum_{\bsi\in I_\Si} \Psi_{\bsi'}^{\spa \bsi} \hat{\phi}_\bsi(t).
$$

\subsection{Quantum differential equation and fundamental solution}
We consider the Dubrovin connection $\nabla^z$, which is a family
of connections parameterized by $z\in \bC\cup \{\infty\}$, on the tangent bundle
$T_{\hat{H}}$ of the formal Frobenius manifold $\hat{H}$:
$$
\nabla^z_{\bsi}=\frac{\partial}{\partial t^{\bsi}} -\frac{1}{z} \hat{\phi}_{\bsi}\star_t.
$$
The commutativity (resp. associativity)
of $*_t$ implies that $\nabla^z$ is a torsion
free (resp. flat) connection on $T_{\hat{H}}$ for all $z$. The equation
\begin{equation}\label{eqn:qde}
	\nabla^z \mu=0
\end{equation}
for a section $\mu\in \Gamma(\hat{H},\cT_{\hat{H}})$ is called the $\bT'$-equivariant
{\em big quantum differential equation} (big QDE). Let
$$
\cT_{\hat{H}}^{f,z}\subset \cT_{\hat{H}}
$$
be the subsheaf of flat sections with respect to the connection $\nabla^z$.
For each $z$, $\cT_{\hat{H}}^{f,z}$ is a sheaf of
$\novT$-modules of rank $E$.

An element $L(z)\in \End(\cT_{\hat H})$ is called a {\em fundamental solution} to the $\bT'$-equivariant big QDE if
the $\cO_{\hat{H}}(\hat{H})$-linear map
$$
L(z): \Gamma(\hat{H},\cT_{\hat{H}}) \to \Gamma(\hat{H},\cT_{\hat{H}})
$$
restricts to a $\novT$-linear isomorphism
$$
L(z): \Gamma(\hat{H},\cT_H^{f,\infty})=\bigoplus_{\bsi\in I_{\Si}} \novT \frac{\partial}{\partial t^{\bsi}}
\to \Gamma(\hat{H},\cT_H^{f,z})
$$
between rank-$E$ free $\novT$-modules.

\subsubsection{The $R-$matrix}
Let $U$ denote the diagonal matrix whose diagonal entries are the canonical coordinates.
The results in \cite{Givental97, Givental98} and \cite{Zong15} imply the following statement.
\begin{theorem}\label{R-matrix}
	There exists a unique matrix power series $R(z)= \one + R_1z+R_2 z^2+\cdots$
	satisfying the following properties:
	\begin{itemize}
		\item The entries of each $R_k$ lie in $\bSTQ$.
		\item $\tS=\Psi R(z) e^{U/z}$  is a fundamental solution to the $\bT'$-equivariant
		big QDE \eqref{eqn:qde}.
		\item $R$ satisfies the unitary condition $R^T(-z)R(z)=\one$.
		\item For $(\sigma, \gamma), (\rho, \delta) \in I_\Sigma$, we have
		$$			
				\lim_{\tQ,\tau''\to 0} R_{\rho,\delta}^{\spa\si,\gamma}(z)
				= \frac{\delta_{\rho,\si}}{|G_\si|}\sum_{h\in G_\si}\chi_\delta(h) \chi_\gamma(h^{-1})
				\prod_{i=1}^3 \exp\left( \sum_{m=1}^\infty \frac{(-1)^m}{m(m+1)}B_{m+1}(c_i(h))
				\left(\frac{z}{\w_i(\si)}\right)^m \right)
		$$
		where $B_m(x)$ is the $m$-th \emph{Bernoulli polynomial} defined by the identity
		$
		\frac{te^{tx}}{e^t-1}=\sum_{m\ge 0}\frac{B_m(x)t^m}{m!}
		$.
	\end{itemize}
\end{theorem}
The matrix $R(z)$ in Theorem \ref{R-matrix} is called the {\em A-model $R$-matrix}.

\subsubsection{The $\cS$-operator}
For any $a,b\in H_{\CR,\bT'}^*(\cX;\bST)$, define the \emph{$\cS$-operator} by
$$
	(a,\cS(b))_{\cX,\bT'} = \left\llangle a,\frac{b}{z-\hat{\psi}}\right\rrangle^{\cX,\bT'}_{0,2}.
$$
The $\cS$-operator, viewed as an element in $\End(\cT_{\hat{H}})$, is a fundamental solution to the $\bT'$-equivariant
big QDE \eqref{eqn:qde} \cite[Section 10.2]{CK99} \cite{Iritani09}. For later use we record some of its well-known properties; see e.g. \cite{Givental97, GT13}. First, it also satisfies the unitary condition 
$$
	\cS^T(-z)\cS(z) = \one.
$$
Moreover, for any basis $\{\phi_\alpha\}$ of $H^*_{\CR, \bT'}(\cX;\bC)$ and dual basis $\{\phi^\alpha\}$ under the $\bT'$-equivariant Poincar\'e pairing, we have the identity
\begin{equation}\label{eqn:SIdentity}
	\frac{1}{z_1 + z_2}\sum_{\alpha} \left\llangle \phi_\alpha,\frac{a}{z_1-\hat{\psi}}\right\rrangle^{\cX,\bT'}_{0,2} \left\llangle \phi^\alpha,\frac{b}{z_2-\hat{\psi}}\right\rrangle^{\cX,\bT'}_{0,2} = \left\llangle \frac{a}{z_1-\hat{\psi}}, \frac{b}{z_2-\hat{\psi}}\right\rrangle^{\cX,\bT'}_{0,2}.
\end{equation}

We introduce some additional notation. For $\bsi, \bsi'\in I_\Si$, define
$$
S^{\bsi'}_{\spa \bsi}(z) := (\phi^{\bsi'}, \cS(\phi_{\bsi}))_{\cX,\bT'},\qquad S_{\bsi'}^{\spa \widehat{\bsi} }(z) := (\phi_{\bsi'}, \cS(\hat{\phi}^{\bsi}))_{\cX,\bT'}.
$$
We define
$$
S^{\widehat{\underline{\bsi}}}_{\spa \widehat{\underline{\bsi'}} }(z)
:= (\hat{\phi}_{\bsi}(t), \cS(\hat{\phi}_{\bsi'}(t)))_{\cX,\bT'}.
$$
Then $(S^{ \widehat{\underline{\bsi}}  }_{\spa \widehat{\underline{\bsi'}} }(z))$ is the matrix of the $\cS$-operator with
respect to the normalized canonical basis  $\{ \hat{\phi}_{\bsi}(t): \bsi\in I_\Si\} $:
$$
	\cS(\hat{\phi}_{\bsi'}(t))=\sum_{\bsi\in I_\Si} \hat{\phi}_{\bsi}(t)
	S^{\widehat{\underline{\bsi}} }_{\spa \widehat{\underline{\bsi'}} }(z).
$$
Moreover, we define
$$
S^{\widehat{\underline{\bsi}}}_{\spa \bsi'}(z)
:= (\hat{\phi}_{\bsi}(t), \cS(\phi_{\bsi'}))_{\cX,\bT'}.
$$
Then $(S^{ \widehat{\underline{\bsi}}  }_{\spa \bsi'}(z))$ is the matrix of the $\cS$-operator with
respect to the  basis $\{\phi_{\bsi}:\bsi\in I_\Si\}$ and
$\{ \hat{\phi}_{\bsi}(t): \bsi\in I_\Si\} $:
$$
	\cS(\phi_{\bsi'})=\sum_{\bsi\in I_\Si} \hat{\phi}_{\bsi}(t)
	S^{\widehat{\underline{\bsi}} }_{\spa \bsi'}(z).
$$

\subsection{The A-model graph sum} \label{sec:Agraph}
In this subsection, we introduce the graph sum formula for the generating functions of descendant Gromov-Witten invariants.
Given a connected graph $\Ga$, we introduce the following notation:
\begin{itemize}
	\item $V(\Ga)$ is the set of vertices.
	\item $E(\Ga)$ is the set of edges.
	\item $H(\Ga)$ is the set of half edges.
	\item $L^o(\Ga)$ is the set of ordinary leaves. The ordinary
	leaves are ordered: $L^o(\Ga)=\{l_1,\dots,l_n\}$ where
	$n$ is the number of ordinary leaves.
	\item $L^1(\Ga)$ is the set of dilaton leaves. The dilaton leaves are unordered.
\end{itemize}
With the above notation, we introduce the following labels:
\begin{itemize}
	\item (genus) $g: V(\Ga)\to \bZ_{\ge 0}$.
	\item (marking) $\bsi: V(\Ga) \to I_\Si$. This induces
	$\bsi :L(\Ga)=L^o(\Ga)\cup L^1(\Ga)\to I_\Si$, as follows:
	if $l\in L(\Ga)$ is a leaf attached to a vertex $v\in V(\Ga)$,
	define $\bsi(l)=\bsi(v)$.
	\item (height) $k: H(\Ga)\to \bZ_{\ge 0}$.
\end{itemize}

Given an edge $e$, let $h_1(e),h_2(e)$ be the two half edges associated to $e$. The order of the two half edges does not affect the graph sum formula in this paper. Given a vertex $v\in V(\Ga)$, let $H(v)$ denote the set of half edges
emanating from $v$. The valency of the vertex $v$ is equal to the cardinality of the set $H(v)$: $\val(v)=|H(v)|$.
A labeled graph $\vGa=(\Ga,g,\bsi,k)$ is {\em stable} if
$$
2g(v)-2 + \val(v) >0
$$
for all $v\in V(\Ga)$. The \emph{genus} of a stable labeled graph
$\vGa$ is defined to be
$$
g(\vGa):= \sum_{v\in V(\Ga)}g(v)  + |E(\Ga)|- |V(\Ga)|  +1
=\sum_{v\in V(\Ga)} (g(v)-1) + \left(\sum_{e\in E(\Gamma)} 1\right) +1.
$$

Let $\bGa(\cX)$ denote the set of all stable labeled graphs
$\vGa=(\Gamma,g,\bsi,k)$. For $g, n \in \bZ_{\ge 0}$ define
$$
\bGa_{g,n}(\cX)=\{ \vGa=(\Gamma,g,\bsi,k)\in \bGa(\cX): g(\vGa)=g, |L^o(\Ga)|=n\}.
$$

We assign weights to leaves, edges, and vertices of a labeled graph $\vGa\in \bGa(\cX)$ as follows:
\begin{itemize}
	\item {\em Ordinary leaves.}
	To each ordinary leaf $l_i \in L^o(\Ga)$ with  $\bsi(l_i)= \bsi\in I_\Si$
	and  $k(l)= k \ge 0$, we assign the following descendant  weight:
	$$
		(\cL^{\bu})^{\bsi}_k(l_i) := [z^k] \left(\sum_{\bsi',\brho\in I_\Si}
		\left(\frac{\bu_i^{\bsi'}(z)}{\sqrt{\Delta^{\bsi'}(t)} }
		S^{\widehat{\underline{\brho}} }_{\spa
			\widehat{\underline{\bsi'}}}(z)\right)_+ R(-z)_{\brho}^{\spa \bsi} \right),
	$$
	where $(\cdot)_+$ means taking the nonnegative powers of $z$.
	
	\item {\em Dilaton leaves.} To each dilaton leaf $l \in L^1(\Ga)$ with $\bsi(l)=\bsi
	\in I_\Si$
	and $k(l)=k \ge 2$, we assign
	$$
	(\cL^1)^{\bsi}_k := [z^{k-1}]\left(-\sum_{\bsi'\in I_\Si}
	\frac{1}{\sqrt{\Delta^{\bsi'}(t)}}
	R_{\bsi'}^{\spa \bsi}(-z) \right).
	$$
	
	\item {\em Edges.} To an edge connecting a vertex marked by $\bsi\in I_\Si$ and a vertex
	marked by $\bsi'\in I_\Si$, and with heights $k$ and $l$ at the corresponding half-edges, we assign
	$$
	\cE^{\bsi,\bsi'}_{k,l} := [z^k w^l]
	\left(\frac{1}{z+w} \left(\delta_{\bsi\bsi'}-\sum_{\brho\in I_\Si}
	R_{\brho}^{\spa \bsi}(-z) R_{\brho}^{\spa \bsi'}(-w)\right) \right).
	$$
	\item {\em Vertices.} To a vertex $v$ with genus $g(v)=g\in \bZ_{\ge 0}$ and with
	marking $\bsi(v)=\bsi$, with $n_1$ ordinary
	leaves and half-edges attached to it with heights $k_1, ..., k_{n_1} \in \bZ_{\ge 0}$ and $n_2$
	dilaton leaves with heights $k_{n_1+1}, \dots, k_{n_1+n_2}\in \bZ_{\ge 0}$, we assign
	$$
	\left(\sqrt{\Delta^{\bsi}(t)}\right)^{2g(v)-2+\val(v)}\langle  \tau_{k_1}\cdots\tau_{k_{n_1+n_2}}\rangle_g
	$$
	where $\langle  \tau_{k_1}\cdots\tau_{k_{n_1+n_2}}\rangle_{g}:= \displaystyle \int_{\Mbar_{g,n_1+n_2}}\psi_1^{k_1} \cdots \psi_{n_1+n_2}^{k_{n_1+n_2}}$.
\end{itemize}

We define the \emph{A-model weight} of a labeled graph $\vGa\in \bGa(\cX)$ to be
\begin{align*}
	w_A^{\bu}(\vGa) := & \prod_{v\in V(\Ga)} \left(\sqrt{\Delta^{\bsi(v)}(t)}\right)^{2g(v)-2+\val(v)} \left\langle \prod_{h\in H(v)} \tau_{k(h)} \right\rangle_{g(v)}
	\prod_{e\in E(\Ga)} \cE^{\bsi(v_1(e)),\bsi(v_2(e))}_{k(h_1(e)),k(h_2(e))}\\
	& \cdot \prod_{l\in L^1(\Ga)}(\cL^1)^{\bsi(l)}_{k(l)}\prod_{j=1}^n(\cL^{\bu})^{\bsi(l_i)}_{k(l_i)}(l_i).
\end{align*}
With the above definition, we have
the following theorem which expresses the $\bT'$-equivariant descendant
Gromov-Witten potential of $\cX$ in terms of a graph sum.

\begin{theorem}[{Zong \cite{Zong15}}]\label{thm:Zong}
For $2g-2+n>0$, we have
$$
	\llangle \bu_1,\dots, \bu_n\rrangle_{g,n}^{\cX,\bT'}=\sum_{\vGa\in \bGa_{g,n}(\cX)}\frac{w_A^{\bu}(\vGa)}{|\Aut(\vGa)|}.
$$
\end{theorem}

\subsection{Equivariant $J$-function}\label{sect:JFunction}
The $\bT'$-equivariant \emph{big $J$-function} $J_{\bT'}^{\mathrm{big}}$ of $\cX$ is a certain generating function of 1-pointed descendant invariants of $\cX$. It is characterized by
$$
	(J_{\bT'}^{\mathrm{big}}(t,z),b)_{\cX,\bT'} = (\one,\cS(b)\big|_{Q = 1})_{\cX,\bT'}
$$
for any $b\in H_{\CR,\bT'}^*(\cX;\bST)$. Equivalently, we have
$$
	J_{\bT'}^{\mathrm{big}}(t,z) = \sum_{\alpha} \double{\one, \frac{\phi_\alpha}{z - \hpsi}}^{\cX, \bT'}_{0,2} \bigg|_{Q = 1} \phi^\alpha = \sum_{\alpha} \double{\frac{\phi_\alpha}{z(z - \hpsi)}}^{\cX, \bT'}_{0,1} \bigg|_{Q = 1} \phi^\alpha.
$$
Here, $\{\phi_\alpha\}$ is any basis of $H^*_{\CR, \bT'}(\cX;\bC)$ and $\{\phi^\alpha\}$ is the dual basis under the $\bT'$-equivariant Poincar\'e pairing. 

We will further consider the restriction to the small phase space $H^2_{\CR, \bT'}(\cX; \bC)$. For parameter $\btau \in H^2_{\CR, \bT'}(\cX; \bC)$, the $\bT'$-equivariant \emph{small $J$-function} of $\cX$ is defined to be
$$
	J_{\bT'}(\btau, z) := J_{\bT'}^{\mathrm{big}}(t = \btau,z) = \sum_{\alpha} \double{\frac{\phi_\alpha}{z(z - \hpsi)}}^{\cX, \bT'}_{0,1} \bigg|_{t = \btau, Q = 1} \phi^\alpha.
$$

\begin{convention}
\rm{
From this point onwards, we make the specialization
$$
	t = \btau, \qquad Q = 1
$$
in the generating functions $\llangle \rrangle^{\cX,\bT'}_{g,n}$ for all $g, n$.
}
\end{convention}

\subsection{Equivariant small $I$-function and genus-zero mirror theorem}\label{sect:IFunction}
We now state the genus-zero mirror theorem for toric Deligne-Mumford stacks, which follows from \cite{Givental98,CCK15,CCIT15}, in the case of toric Calabi-Yau 3-orbifolds over the small phase space. We choose elements $H_1, \dots, H_{\fp} \in \bL^\vee \cap \tNef(\cX)$ that satisfy the following conditions. We denote $\bar{H}_a := \kappa(H_a)$ where $\kappa: \bL^\vee \to H^2(\cX; \bZ)$ is defined in Section \ref{sect:LineBundles}.
\begin{itemize}
	\item $H_a = D_{3+a}$ for $a = \fp'+1, \dots, \fp$. This implies that $\bar{H}_a = 0$ for such $a$.

	\item $\{H_1, \dots, H_{\fp}\}$ is a $\bQ$-basis of $\bL^\vee_{\bQ} := \bL^\vee \otimes \bQ$. This implies that $\{\bar{H}_1, \dots, \bar{H}_{\fp'}\}$ is a $\bQ$-basis of $H^2(\cX; \bQ)$.
	
	\item Given any 3-cone $\sigma \in \Sigma(3)$, we may uniquely write
	\begin{equation}\label{eqn:sai}
		H_a = \sum_{i \in I_\sigma} s_{ai}^\sigma D_i
	\end{equation}
	for $s_{ai}^\sigma \in \bQ_{\ge 0}$. Then we require that $s_{ai}^\sigma \in \bZ_{\ge 0}$ for all $\sigma, a, i$.
	
\end{itemize}
We introduce formal variables $t_0, q = (q_1, \dots, q_{\fp})$ and write $q_K = (q_1, \dots, q_{\fp'})$, $q_{\orb} = (q_{\fp'+1}, \dots, q_{\fp})$. We write $|q|$ for the maximum modulus of the $q_a$'s, and $|q_K|$, $|q_{\orb}|$ similarly. The limit $q \to 0$ may be viewed as a B-model large complex structure/orbifold mixed type limit. In particular, $q_K \to 0$ is the large complex structure limit which, under mirror symmetry, corresponds to taking the large radius limit while preserving the twisted classes. For $\beta \in \bK$, we define
$$
	q^\beta := \prod_{a = 1}^{\fp} q_a^{\inner{H_a, \beta}}.
$$
It is a \emph{monomial} in the $q_a$'s when $\beta \in \bK_{\eff}$, by the conditions above.

For $a = 1, \dots, \fp'$, let $\bar{H}_a^{\bT'} \in H^2_{\bT'}(\cX; \bZ)$ be the unique $\bT'$-equivariant lift of $\bar{H}_a$ whose restriction to $\fp_{\si_0}$ is zero, where $\si_0$ is the preferred 3-cone fixed in Section \ref{sect:Flags}. For $\sigma \in \Sigma(3)$, denote
$$
    \sw_{a, \sigma} := \bar{H}_a^{\bT'} \big|_{\fp_\sigma} \qquad \in M'_\sigma \subset M'_{\bQ}.
$$

Following \cite{Iritani09,CCIT15}, the $\bT'$-equivariant \emph{small $I$-function} of $\cX$ is defined to be
\begin{align*}
	I_{\bT'}(t_0, q,z) := & e^{(t_0 + \sum_{a = 1}^{\fp'} \bar{H}_a^{\bT'}\log q_a)/z}\\
	& \sum_{\beta \in \bK_{\eff}} q^\beta \prod_{i = 1}^{3+\fp'} \frac{\prod_{m = \ceil{\inner{D_i, \beta}}}^\infty (\bar{D}_i^{\bT'} + (\inner{D_i, \beta}-m)z)}{\prod_{m = 0}^\infty (\bar{D}_i^{\bT'} + (\inner{D_i, \beta}-m)z)} \prod_{i = 4+\fp'}^{3+\fp} \frac{\prod_{m = \ceil{\inner{D_i, \beta}}}^\infty ((\inner{D_i, \beta}-m)z)}{\prod_{m = 0}^\infty ((\inner{D_i, \beta}-m)z)} \one_{v(\beta)}.
\end{align*}
The following theorem follows from \cite{Givental98,CCK15,CCIT15}.

\begin{theorem}\label{thm:ToricMirror}
We have
$$
J_{\bT'}(\btau, z) = I_{\bT'}(t_0, q,z)
$$
under the $\bT'$-equivariant mirror map $\btau = \btau(t_0, q)$ given by the first-order term in the asymptotics expansion
$$
	I_{\bT'}(t_0, q,z) = \one + z^{-1}\btau(t_0, q) + o(z^{-1}).
$$
\end{theorem}

In this paper, we take the specialization $t_0 = 0$. We write $\btau(q) := \btau(0, q)$ and
\begin{align*}
	& I_{\bT'}(q,z) := I_{\bT'}(t_0 = 0, q,z) \\
	& = e^{(\sum_{a = 1}^{\fp'} \bar{H}_a^{\bT'}\log q_a)/z} \sum_{\beta \in \bK_{\eff}} q^\beta \prod_{i = 1}^{3 + \fp'} \frac{\Gamma\left(1 + \frac{\bar{D}_i^{\bT'}}{z} - \{-\inner{D_i, \beta}\} \right)}{\Gamma\left(1 + \frac{\bar{D}_i^{\bT'}}{z} + \inner{D_i, \beta} \right)} \prod_{i = 4+\fp'}^{3+\fp} \frac{\Gamma\left(1 - \{-\inner{D_i, \beta}\} \right)}{\Gamma\left(1 + \inner{D_i, \beta} \right)} \frac{\one_{v(\beta)}}{z^{\age(v(\beta))}}.
\end{align*}
For any $\sigma \in \Sigma(3)$, we have
$$
	I_{\bT'}(q,-z) \big|_{\fp_\sigma} = e^{-(\sum_{a = 1}^{\fp'} \sw_{a, \sigma}\log q_a)/z} \sum_{\beta \in \bK_{\eff, \sigma}} q^\beta \prod_{i \in I_\sigma} \frac{1}{\inner{D_i, \beta}!} \prod_{i \in I'_\sigma} \frac{\Gamma\left(1 - \frac{\sw_{i, \sigma}}{z} - c_i(v(\beta)) \right)}{\Gamma\left(1 - \frac{\sw_{i, \sigma}}{z} + \inner{D_i, \beta} \right)} \frac{\one_{v(\beta)}}{(-z)^{\age(v(\beta))}}.
$$
In obtaining the above we used the following observations. If $\beta \not \in \bK_{\eff, \sigma}$, then either $\beta \not \in \bK_\sigma$, in which case $v(\beta) \not \in \Box(\sigma)$, or $\beta \in \bK_\sigma \setminus \tNE_\sigma$, in which case there exists $i \in I_\sigma$ such that $\inner{D_i, \beta} \in \bZ_{<0}$, which further implies that $\frac{1}{\Gamma\left(1 + \inner{D_i, \beta} \right)} = 0$. On the other hand, for $\beta \in \bK_{\eff, \sigma}$ and $i \in I_\sigma$, we have $\inner{D_i, \beta} \in \bZ_{\ge 0}$ and thus $\Gamma\left(1 + \inner{D_i, \beta} \right) = \inner{D_i, \beta}!$.

We further separate the above restriction of the $I$-function to inertia components of $\fp_\sigma$. Here, we take a specialization of equivariant parameters $\su_1 = u_1$, $\su_2 = u_2$ for $u_1, u_2 \in \bC$. Let 
$$
	w_{i, \sigma} \qquad \text{(resp. $w_{a, \sigma}$)}
$$
denote the resulting specialization of $\sw_{i, \sigma}$ for $i = 1, \dots, 3 + \fp'$ (resp. $\sw_{a, \sigma}$ for $a = 1, \dots, \fp'$). For $v \in \Box(\sigma)$, the restriction is $(-z)^{-\age(v)}I_{\sigma, v}$ where
$$
	I_{\sigma, v} := e^{-(\sum_{a = 1}^{\fp'} w_{a, \sigma}\log q_a)/z} \sum_{\substack{\beta \in \bK_{\eff, \sigma} \\ v(\beta) = v}} q^\beta \prod_{i \in I_\sigma} \frac{1}{\inner{D_i, \beta}!} \prod_{i \in I'_\sigma} \frac{\Gamma\left(1 - \frac{w_{i, \sigma}}{z} - c_i(v) \right)}{\Gamma\left(1 - \frac{w_{i, \sigma}}{z} + \inner{D_i, \beta} \right)}.
$$
There is a unique $\beta \in \bK_{\eff, \sigma}$ with $v(\beta) = v$ such that
\begin{itemize}
	\item $q^\beta$ is a monomial in $(q_a)_{3+a \in I_\sigma, b_{3+a} \in \sigma}$ of degree $\age(v)$;
	\item the coefficient of the $q^\beta$ in the above summation is $1$;
	\item for any $\beta' \in \bK_{\eff, \sigma}$ with $\beta' \neq \beta$ and $v(\beta') = v$, $q^{\beta'}$ is divisible by $q^\beta$ and is a monomial in $q$ in which either there is a variable among $q_K$ or the total power of variables in $q_{\orb}$ is at least $3$.
\end{itemize}
Letting $q^v := q^\beta$, we have the estimate
\begin{equation}\label{eqn:IsvLeading}
	I_{\sigma, v} = e^{-(\sum_{a = 1}^{\fp'} w_{a, \sigma}\log q_a)/z} \left(q^v + O(|q_{\orb}|^3) + O(|q_K|)\right).
\end{equation}

By $\bT'$-equivariant localization, we have
\begin{equation}\label{eqn:ILocalize}
	I_{\bT'}(q,-z) \big|_{\su_1 = u_1, \su_2 = u_2} = \sum_{\substack{\sigma \in \Sigma(3) \\ v \in \Box(\sigma)}} (-z)^{-\age(v)} I_{\sigma, v} \phi_{\sigma, v} \big|_{\su_1 = u_1, \su_2 = u_2}
\end{equation}
where for $\sigma \in \Sigma(3)$, $v \in \Box(\sigma)$ we set 
$$
	\phi_{\sigma, v} := \frac{\iota_{\sigma, *} \one_v}{e_{\bT'}(T_{\fp_\sigma}\cX_v)}.
$$

\subsection{Equivariant GKZ system}\label{sect:GKZ}
The above components of the equivariant $I$-functions are solutions to the $\bT'$-equivariant Gelfand-Kapranov-Zelevinsky (GKZ) type system of differential equations \cite{GKZ89}. We now write down the system using the labeling conventions fixed by the preferred flag $(\tau_0, \si_0)$ as in Section \ref{sect:Flags}. First, let $s_{ai} = s_{ai}^{\sigma_0} \in \bZ_{\ge 0}$ for $a = 1, \dots, \fp$, $i = 4, \dots, 3+\fp$, where recall from \eqref{eqn:sai} that $H_a = \sum_{i = 4}^{3 + \fp} s_{ai} D_i$. For convenience we set $s_{ai} = 0$ for $i = 1, 2, 3$. Conversely, we write
$$
D_i = \sum_{a = 1}^{\fp} m_i^{(a)} H_a, \qquad i = 1, \dots, 3 + \fp
$$
for $m_i^{(a)} \in \bQ$. We denote
$$
\partial_a := q_a\frac{\partial}{\partial q_a}, \quad  a = 1, \dots, \fp, \qquad \partial_i := \sum_{a = 1}^{\fp} m_i^{(a)} \partial_a, \quad i = 1, \dots, 3+\fp.
$$
For $\beta \in \bL$, define the differential operator
$$
\bD_\beta^{\bT'} := q^\beta \prod_{i: \inner{D_i, \beta}<0} \prod_{m=0}^{-\inner{D_i, \beta} -1} \left(\partial_i - m - \frac{w^i}{z} \right) - \prod_{i: \inner{D_i, \beta}>0} \prod_{m=0}^{\inner{D_i, \beta} -1} \left(\partial_i - m - \frac{w^i}{z} \right),
$$
where
$$
(w^1, \dots, w^{3+\fp}) := \left( \frac{1}{\fr}u_1, \frac{\fs}{\fr \fm} u_1 + \frac{1}{\fm} u_2 , -\frac{\fs + \fm}{\fr \fm} u_1 - \frac{1}{\fm} u_2, 0, \dots, 0 \right).
$$
Moreover, define the differential operator
$$
	\bE := z\frac{\partial}{\partial z} + u_1\frac{\partial}{\partial u_1} + u_2\frac{\partial}{\partial u_2}.
$$


\begin{definition}\rm{
Let
$
	\bfS_{\bT'}
$
denote the space of functions $f$, depending on $q, u_1, u_2, z$, that satisfy the GKZ-type differential equations
\begin{equation}\label{eqn:GKZ}
	\bD_\beta^{\bT'} f = 0, \quad \text{for all }\beta \in \bL, \qquad \bE f = 0.
\end{equation}
}
\end{definition}

\begin{lemma}\label{lem:ICoeffSoln}
The components $\{I_{\sigma, v}\}_{\sigma \in \Sigma(3), v \in \Box(\sigma)}$ form a basis of the solutions space $\bfS_{\bT'}$ for small $q_a$'s and generic $u_1, u_2$. In particular, any solution in $\bfS_{\bT'}$ can be written uniquely as
$$
	\sum_{\substack{\sigma \in \Sigma(3) \\ v \in \Box(\sigma)}} a_{\sigma, v}\left(\tfrac{u_1}{z}, \tfrac{u_2}{z}\right) I_{\sigma, v}
$$
where each coefficient function $a_{\sigma, v}\left(\tfrac{u_1}{z}, \tfrac{u_2}{z}\right)$ is meromorphic in the arguments.
\end{lemma}

The lemma follows from \cite[Lemma 4.19]{Iritani09} and its proof, which applies to weak Fano toric orbifolds that are not necessarily compact. Indeed, our Equation \eqref{eqn:GKZ} is the analog of Equations (77), (78) there. Note that in the Calabi-Yau case, the equation $\bE f = 0$ imposes the homogeneity condition that solutions are homogeneous of degree 0 in the variables whose degrees are prescribed by $\deg(z) = \deg(u_1) = \deg(u_2) = 2$ and $\deg(q^\beta) = 0$ for all $\beta$.

\subsection{Non-equivariant GKZ system}\label{sect:NEGKZ}
In the non-equivariant setting, for $\beta \in \bL$, define the differential operator
$$
\bD_\beta := q^\beta \prod_{i: \inner{D_i, \beta}<0} \prod_{m=0}^{-\inner{D_i, \beta} -1} \left(\partial_i - m \right) - \prod_{i: \inner{D_i, \beta}>0} \prod_{m=0}^{\inner{D_i, \beta} -1} \left(\partial_i - m \right).
$$

\begin{definition}\rm{
Let
$
	\bfS
$
denote the space of functions $f$, depending on $q$, that satisfy the GKZ-type differential equations
$$
	\bD_\beta f = 0, \quad \text{for all }\beta \in \bL.
$$
}
\end{definition}

Here note that in the non-equivariant case, the equation $\bE f = 0$ simply means that $f$ is independent of $z$. The system is also referred to as the non-equivariant \emph{Picard-Fuchs system}. It is known that $\bfS$ is a $\bC$-vector space of rank $\dim_{\bC} H^*_{\CR}(\cX; \bC)$ \cite{GKZ90,Iritani09} and that a basis of solutions is given by components of the non-equivariant small $I$-function of $\cX$ \cite{Givental98,Iritani09}
$$	
	I(q,z) = e^{(\sum_{a = 1}^{\fp'} \bar{H}_a\log q_a)/z} \sum_{\beta \in \bK_{\eff}} q^\beta \prod_{i = 1}^{3 + \fp'} \frac{\Gamma\left(1 + \frac{\bar{D}_i}{z} - \{-\inner{D_i, \beta}\} \right)}{\Gamma\left(1 + \frac{\bar{D}_i}{z} + \inner{D_i, \beta} \right)} \prod_{i = 4+\fp'}^{3+\fp} \frac{\Gamma\left(1 - \{-\inner{D_i, \beta}\} \right)}{\Gamma\left(1 + \inner{D_i, \beta} \right)} \frac{\one_{v(\beta)}}{z^{\age(v(\beta))}}.
$$
Specifically, consider $H^*_{\CR, c}(\cX; \bC)$, the Chen-Ruan cohomology of $\cX$ \emph{with compact support}, and the perfect pairing
$$
	(-,-)_{\cX}: H^*_{\CR, c}(\cX; \bC) \times H^*_{\CR}(\cX; \bC) \to \bC.
$$
Then a $\bC$-basis of $\bfS$ can be given by
$$
	[z^{-\frac{6-\deg(a)}{2}}] \left( a,I(q, -z) \right)_{\cX}
$$
as $a$ ranges through any homogeneous $\bC$-basis of $H^*_{\CR, c}(\cX; \bC)$, where the notation $[z^{-k}]$ stands for taking the $z^{-k}$-coefficient in the expansion in $z^{-1}$.



\section{Central charges of coherent sheaves}\label{sect:Sheaves}
In this section, we define quantum cohomology central charges of $\bT'$-equivariant coherent sheaves on $\cX$. We will focus on the central charges of sheaves whose support in $\cX$ is bounded below/above with respect to a non-zero cocharacter of $\bT'$. We characterize such sheaves by introducing the notion of bounded below/above $K$-groups. This notion extends the study of Borisov-Horja \cite{BH06,BH15} on the $K$-groups and $K$-groups with compact support, and can be generalized to toric orbifolds of higher dimensions. Additional details are provided in Appendix \ref{appdx:Bounded} to supplement this section.

\subsection{$K$-groups and equivariant central charges}\label{sect:KCentral}
Let
$$
    D(\cX) := D^b(\Coh(\cX)), \qquad D_{\bT'}(\cX) := D^b(\Coh_{\bT'}(\cX))
$$
denote the bounded derived category of ($\bT'$-equivariant) coherent sheaves on $\cX$ and
$$
    K(\cX) := K_0(D(\cX)), \qquad K_{\bT'}(\cX) := K_0(D_{\bT'}(\cX))
$$
denote their Grothendieck groups. There is a natural action of the lattice $M'$ on $K_{\bT'}(\cX)$ given by tensoring with characters of $\bT'$.

The $\bT'$-equivariant \emph{twisted Chern character}
$$
    \tch_{z, \bT'}: K_{\bT'}(\cX) \to H^*_{\CR, \bT'}(\cX; \bC)\formal{\su_1, \su_2, z^{-1}}
$$
is the map uniquely characterized by the following properties:
\begin{itemize}
    \item $\tch_{z, \bT'}$ is a homomorphism of additive groups, i.e.
    $$
        \tch_{z, \bT'}(\cE_1 \oplus \cE_2) = \tch_{z, \bT'}(\cE_1) + \tch_{z, \bT'}(\cE_2).
    $$

    \item For any $\bT'$-equivariant line bundle $\cL$ on $X$,
    $$
        \tch_{z, \bT'}(\cL) = \bigoplus_{v \in \Box(\cX)} \one_v \exp 2\pi\sqrt{-1}\left(-\frac{(c_1)_{\bT'}(\cL)}{z} + \age_v(\cL)\right).
    $$
    In particular, for $\cL = \cO_{\cX}\left(\sum_{i = 1}^{3+\fp'} r_i \cD_i \right)$, we have
    $$
        \tch_{z, \bT'}(\cL) = \bigoplus_{v \in \Box(\cX)} \one_v \exp 2\pi\sqrt{-1}\left(-z^{-1}\sum_{i = 1}^{3+\fp'} r_i\Dbar_i^{\bT'} + \sum_{i = 1}^{3+\fp'} r_ic_i(v) \right).
    $$
\end{itemize}

Moreover, let
$$
    \hGa_{\cX}^z := \bigoplus_{v \in \Box(\cX)} z^{3-\age(v)} \one_v \prod_{i = 1}^{3+\fp'} \Gamma\left(1 + \frac{\Dbar_i^{\bT'}}{z} - c_i(v) \right)
$$
where the Gamma function is expanded at $1 - c_i(v) > 0$. The characteristic class $\hGa^z$ may be defined on all of $K_{\bT'}(\cX)$; the above is the class of $T\cX$.


\begin{definition}\label{def:KFraming} \rm{
For $\cE \in K_{\bT'}(\cX)$, the $\bT'$-equivariant \emph{$K$-theoretic framing} of $\cE$ is defined as
$$
    \kappa_z^{\bT'}(\cE) := \hGa_{\cX}^z \tch_{z, \bT'}(\cE).
$$
}
\end{definition}

We note that $\tch_{z, \bT'}$, $\hGa_{\cX}^z$, and thus $\kappa_z^{\bT'}$ all admit a non-equivariant limit valued in $H^*_{\CR}(\cX; \bC)[z^{-1}]$, which we denote by
$$
    \tch_z, \qquad  \hGa_{\cX}^z,  \qquad \kappa_z
$$
respectively. Borisov-Horja \cite{BH06,BH15} studied the non-equivariant $K$-group $K(\cX)$ and twisted Chern character extensively and provided explicit combinatorial descriptions; see Sections \ref{apx-sect:KGroup}, \ref{apx-sect:CombChern} for a review.

Motivated by \cite{Hosono06,KKP08,Iritani09, FLZ17,Fang20}, we make the following definition. 

\begin{definition}\label{def:SheafCentral}\rm{
For $\cE \in K_{\bT'}(\cX)$, the $\bT'$-equivariant quantum cohomology \emph{central charge} of $\cE$ is defined by
$$
    Z_{\bT'}(\cE) := \double{\frac{\kappa_z^{\bT'}(\cE)}{z(z + \hpsi)}}^{\cX, \bT'}_{0,1}.
$$    
}
\end{definition}

For $\cE \in K_{\bT'}(\cX)$, by Theorem \ref{thm:ToricMirror}, we may write
$$
    Z_{\bT'}(\cE) = \left(\kappa_z^{\bT'}(\cE), J_{\bT'}(\btau, -z)\right)_{\cX, \bT'} =  \left(\kappa_z^{\bT'}(\cE), I_{\bT'}(q, -z)\right)_{\cX, \bT'}
$$
under the mirror map $\btau = \btau(q)$. Here and in what follows, we implicitly take the specialization $\su_1 = u_1$, $\su_2 = u_2$. It follows from Lemma \ref{lem:ICoeffSoln} that $Z_{\bT'}(\cE) \in \bfS_{\bT'}$.

\begin{lemma}\label{lem:CentralInjective}
The $\bT'$-equivariant central charge homomorphism
$$
    K_{\bT'}(\cX) \to \bfS_{\bT'}, \qquad \cE \mapsto Z_{\bT'}(\cE)
$$
is injective.
\end{lemma}

\begin{proof}
By the non-degeneracy of the pairing $\left(-, -\right)_{\cX, \bT'}$, it suffices to show that the $K$-theoretic framing $\kappa_z^{\bT'}$ is injective, which follows from the injectivity of the twisted Chern character $\tch_{z, \bT'}$.
\end{proof}

\subsection{Sheaves with compact support and non-equivariant central charges}\label{sect:KCompact}
Let $\Sigma^c$ denote the subset of cones in $\Sigma$ whose interior is contained in the interior of the support of $\Sigma$. Set
$$
    \cX^c = \bigcup_{\sigma \in \Sigma^c} \cV(\sigma).
$$
Let $\pi:\cX\to X_0$ be the composition of the map $\cX\to X$ from $\cX$ to its coarse moduli space $X$
and the map $X\to X_0$ from the simplicial toric variety $X$ to its affinization $X_0 = \Spec \left(H^0(X, \cO_X)\right)$. 
Then $\cX^c =\pi^{-1}(0)$, where 0 is the unique $\bT$-fixed point in the affine toric variety $X_0$.

When $X$ is affine (and equal to $X_0$), $\Sigma^c = \Sigma(3)$ consists of a unique 3-cone $\sigma = \sigma_0$ and $\cX^c = \fp_{\sigma_0}$. Otherwise, any minimal cone in $\Sigma^c$ is either 1- or 2-dimensional, and any irreducible component of $\cX^c$ is a compact $\bT$-invariant divisor or curve.

Let $D^c(\cX), D^c_{\bT'}(\cX)$ denote the full subcategories of $D(\cX), D_{\bT'}(\cX)$ respectively consisting of complexes of sheaves whose cohomology sheaves are supported on $\cX^c$, and
$$
    K^c(\cX) := K_0(D^c(\cX)), \qquad K^c_{\bT'}(\cX) := K_0(D^c_{\bT'}(\cX))
$$
denote their Grothendieck groups. The group $K^c(\cX)$ (resp. $K^c_{\bT'}(\cX)$) is a module over the ring $K(\cX)$ (resp. $K_{\bT'}(\cX)$). The non-equivariant version $K^c(\cX)$ is studied in detail in \cite{BH15} and interpreted as a $K$-group with compact support; see Sections \ref{apx-sect:KGroup}, \ref{apx-sect:CombChern} for a review. 
The inclusion of categories induces natural maps
$$
    K^c(\cX) \to K(\cX), \qquad K^c_{\bT'}(\cX) \to K_{\bT'}(\cX)
$$
which are $K(\cX)$-, $K_{\bT'}(\cX)$-linear respectively. Thus the definition of the $\bT'$-equivariant $K$-theoretic framing and central charge can be extended to $K^c_{\bT'}(\cX)$.

Now let $\cE \in K^c(\cX)$. Choose a $\bT'$-equivariant lift $\cE_{\bT'} \in K_{\bT'}(\cX)$. Then the $\bT'$-equivariant central charge $Z_{\bT'}(\cE_{\bT'})$ admits a well-defined non-equivariant limit
$$
    Z^c(\cE) := Z_{\bT'}(\cE_{\bT'}) \bigg|_{u_1 = u_2 = 0}
$$
which does not depend on the choice of lift $\cE_{\bT'}$. We may alternatively describe $Z^c(\cE)$ as follows. Similar to the non-equivariant twisted Chern character $\tch_z$, consider the twisted Chern character with compact support
$$
    \tch_z^c: K^c(\cX) \to H^*_{\CR, c}(\cX; \bC)[z^{-1}]
$$
which fits into a commutative diagram
$$
    \xymatrix{
        K^c(\cX) \ar[r] \ar[d]_{\tch_z^c} & K(\cX) \ar[d]^{\tch_z} \\
        H^*_{\CR, c}(\cX; \bC)[z^{-1}] \ar[r] & H^*_{\CR}(\cX; \bC)[z^{-1}]
    }
$$
(cf. Section \ref{apx-sect:CombChern}). Then we have
$$
    Z^c(\cE) = \left(\hGa_X^z \tch_z^c(\cE), I(q, -z)\right)_{\cX}.
$$
By Section \ref{sect:NEGKZ}, we have $Z^c(\cE) \in \bfS$. The same reasoning for Lemma \ref{lem:CentralInjective} gives the following statement which is already observed in \cite[Section 5]{Hosono06} when $\cX$ is smooth.


\begin{lemma}\label{lem:NECentralInjective}
The non-equivariant compact-support central charge homomorphism
$$
    K^c(\cX) \to \bfS, \qquad \cE \mapsto Z^c(\cE)
$$
is injective and is an isomorphism after tensoring the domain with $\bC$.
\end{lemma}

\subsection{Sheaves with bounded below/above support}
Now we introduce $K$-groups of $\cX$ with bounded below/above support. 
Let $f \in \bQ$, written as $f = \frac{\sfb}{\sfa}$ for coprime integers $\sfa, \sfb \in \bZ$ such that $\sfa$ is positive. The non-zero primitive cocharacter $\sv = \sfa e_1 + \sfb e_2 \in N' \subset N$ of $\bT'$ (and of $\bT$) determines a 1-subtorus $\bT_f := \ker(\sfa\su_2 - \sfb\su_1) \cong \bC^*$ of $\bT'$. Let $\bT_{f, \bR} \cong U(1)$ be the maximal compact subtorus and
$$
    \mu_{\bT_{f, \bR}}: \cX \to \bR
$$
denote the moment map where we use the identification of the dual Lie algebra of $\bT_{f, \bR}$ with $\bR$ provided by $\sv$. We choose $f$ generically such that for any flag $(\tau, \sigma) \in F(\Sigma)$, the weight of the $\bT_f$-action on the tangent line at $\fp_\sigma$ along $\fl_\tau$ is non-zero; equivalently, the image of any $\fl_\tau$ under $\mu_{\bT_{f, \bR}}$ is non-constant.

Let $\Sigma^+$ (resp. $\Sigma^-$) denote the subset of cones $\sigma \in \Sigma$ such that $\mu_{\bT_{f, \bR}}(\cV(\sigma))$ is a bounded below (resp. above) subset of $\bR$. Let
$$
    \cX^\pm := \bigcup_{\sigma \in \Sigma^\pm} \cV(\sigma).
$$
We have $\cX^+ \cap \cX^- = \cX^c$. Note that any minimal cone in $\Sigma^\pm$ is either 1- or 2-dimensional, and any irreducible component of $\cX^c$ is a $\bT$-invariant divisor or curve.

Now let $D^\pm(\cX), D^\pm_{\bT'}(\cX)$ denote the full subcategories of $D(\cX), D_{\bT'}(\cX)$ respectively consisting of complexes of sheaves whose cohomology sheaves are supported on $\cX^\pm$, and
$$
    K^\pm(\cX) := K_0(D^\pm(\cX)), \qquad K^\pm_{\bT'}(\cX) := K_0(D^\pm_{\bT'}(\cX))
$$
denote their Grothendieck groups. The group $K^\pm(\cX)$ (resp. $K^\pm_{\bT'}(\cX)$) is also a module over the ring $K(\cX)$ (resp. $K_{\bT'}(\cX)$).


In Appendix \ref{appdx:Bounded}, we give the construction of $K^\pm(\cX)$ for a general semi-projective toric orbifold $\cX$ with respect to a non-zero primitive cocharacter of the algebraic torus. We also give Stanley-Reisner presentations following \cite{BH06, BH15} and study additional properties. Here, we collect a useful result which follows from Lemmas \ref{apx-lem:KGroupGenerator}, \ref{apx-lem:EquivKGroupGenerator}.


\begin{lemma}\label{lem:KGroupGenerator}
For $\star \in \{\emptyset, c, +, -\}$, the group $K^\star(\cX)$ is additively generated by sheaves of the form
\begin{equation}\label{eqn:KGroupGenerator}
    \iota_{\sigma, *}\cL
\end{equation}
where $\sigma \in \Sigma^\star$ is minimal and $\cL$ is a line bundle on $\cV(\sigma)$. Moreover, $K^\star_{\bT'}(\cX)$ is generated by sheaves of the form \eqref{eqn:KGroupGenerator} where $\cL$ is a $\bT'$-equivariant line bundle on $\cV(\sigma)$.
\end{lemma}

The map $\iota_\sigma^*:\Pic(\cX)\to \Pic(\cV(\sigma))$ is surjective, so any line bundle $\cL$ on $\cV(\sigma)$
is of the form $\cL=\iota_\sigma^*\tilde{\cL}$ where $\tilde{\cL}$ is a line bundle on $\cX$, and 
\eqref{eqn:KGroupGenerator} can be rewritten as
$$
\iota_{\sigma, *} \cL= \iota_{\sigma,*} \left(\cO_{\cV(\sigma)}\otimes \cL\right) = 
\iota_{\sigma,*} \left(\cO_{\cV(\sigma)}\otimes \iota_\sigma^* \tilde{\cL}\right) 
= \iota_{\sigma,*}\cO_{\cV(\sigma)} \otimes \tilde{\cL}.
$$
Similarly, given any $\bT'$-equivariant line bundle $\cL$ on $\cV(\sigma)$, there exists a $\bT'$-equivariant line bundle $\tilde{\cL}$ on $\cX$ such that $\iota_{\sigma,*} \cL = \iota_{\sigma,*}\cO_{\cV(\sigma)}  \otimes \tilde{\cL}$.
From now on, we will denote $\iota_{\sigma, *}\cO_{\cV(\sigma)}$ by $\cO_{\cV(\sigma)}$. Lemma \ref{lem:KGroupGenerator} can be reformulated as follows. 
\begin{lemma}\label{lem:KGroupGenerator2}
For $\star \in \{\emptyset, c, +, -\}$, the group $K^\star(\cX)$ is additively generated by sheaves of the form
\begin{equation}\label{eqn:KGroupGenerator2}
    \cO_{\cV(\sigma)} \otimes \cL
\end{equation}
where $\sigma \in \Sigma^\star$ is minimal and $\cL$ is a line bundle on $\cX$. Moreover, $K^\star_{\bT'}(\cX)$ is generated by sheaves of the form \eqref{eqn:KGroupGenerator2} where $\cL$ is a $\bT'$-equivariant line bundle on $\cX$.
\end{lemma}

In the present situation, for $\star = \pm$, $\cV(\sigma)$ in \eqref{eqn:KGroupGenerator} is either a divisor or a curve. For $\star = c$, $\cV(\sigma)$ is a \emph{compact} divisor or curve except when $\cX$ is affine, in which case $\sigma = \sigma_0$ is the unique 3-cone and $\cV(\sigma) = \cX^c$ is the point $\fp_{\sigma}$. 

The inclusions of categories induce natural maps
\begin{equation}\label{eqn:KNaturalMaps}
    K^c(\cX) \to K^\pm(\cX) \to K(\cX), \qquad K^c_{\bT'}(\cX) \to K^\pm_{\bT'}(\cX) \to K_{\bT'}(\cX)
\end{equation}
which are $K(\cX)$-, $K_{\bT'}(\cX)$-linear respectively. Thus the definition of the $\bT'$-equivariant $K$-theoretic framing and central charge can be extended to $K^\pm_{\bT'}(\cX)$. Moreover, we have by Lemma \ref{apx-lem:EquivKInjective} that the map $K^\pm_{\bT'}(\cX) \to K_{\bT'}(\cX)$ is injective. This combined with Lemma \ref{lem:CentralInjective} gives the following statement.

\begin{proposition}\label{prop:PMCentralInjective}
The $\bT'$-equivariant central charge homomorphism
$
    K^\pm_{\bT'}(\cX) \to \bfS_{\bT'}
$
is injective. 
\end{proposition}

On the categorical level, there is a natural pairing between $D^+(\cX)$ (resp. $D^+_{\bT'}(\cX)$) and $D^-(\cX)$ (resp. $D^-_{\bT'}(\cX)$) given by
$$    
    \inner{\cE^+, \cE^-} := \sum_k (-1)^k \dim_{\bC} \Ext^k(\cE^+, \cE^-).
$$
The pairings descend to the Grothendieck groups
\begin{equation}\label{eqn:EulerPairing}
    \chi: K^+(\cX) \times K^-(\cX) \to \bZ, \qquad \chi_{\bT'}: K^+_{\bT'}(\cX) \times K^-_{\bT'}(\cX) \to \bZ
\end{equation}
which we refer to as the \emph{Euler characteristic pairings}. Additional details and properties of the non-equivariant pairing $\chi$ are supplied in Section \ref{apx-sect:PMEulerPairing}.

\subsection{Computation of central charges}
We now compute the quantum cohomology central charges of the structure sheaves of toric divisors, curves, and fixed points in $\cX$ with twists by $\bT'$-equivariant line bundles. To begin with, given $\cE \in K_{\bT'}(\cX)$ and writing
$$
    Z_{\bT'}(\cE) =  \left(\kappa_z^{\bT'}(\cE), I_{\bT'}(q, -z)\right)_{\cX, \bT'} = \sum_{\substack{\sigma \in \Sigma(3) \\ v \in \Box(\sigma)}}  a_{\sigma, v}(\cE)I_{\sigma, v}
$$
as in Lemma \ref{lem:ICoeffSoln}, we may use the localization expression \eqref{eqn:ILocalize} to determine the coefficient functions as
$$
    a_{\sigma, v}(\cE) = \frac{(-z)^{-\age(v)} \inv^*(\kappa_z^{\bT'}(\cE)) \big|_{\fp_{\sigma, v}}}{|G_\sigma| e_{\bT'}(T_{\fp_\sigma}\cX_v)}.
$$
For $\sigma \in \Sigma(3)$, $v \in \Box(\sigma)$, we have
\begin{align*}
    z^{-\age(v)} \inv^*\left(\hGa_{\cX}^z\right) \big|_{\fp_{\sigma, v}} &= z^{-\age(v)} z^{3-\age(\inv(v))} \prod_{i \in I'_\sigma} \Gamma\left(1 + \frac{w_{i, \sigma}}{z} - c_i(\inv(v)) \right)\\
    & = z^{\dim \cX_v} \prod_{i \in I'_\sigma} \Gamma\left(\frac{w_{i, \sigma}}{z} + c_i(v) \right) \left(\frac{w_{i, \sigma}}{z} \right)^{\delta_{c_i(v), 0}}\\
    & = e_{\bT'}(T_{\fp_\sigma}\cX_v) \prod_{i \in I'_\sigma} \Gamma\left(\frac{w_{i, \sigma}}{z} + c_i(v) \right)
\end{align*}
which implies that
$$
    a_{\sigma, v}(\cE) = \frac{\inv^*(\tch_{z, \bT'}(\cE)) \big|_{\fp_{\sigma, v}}}{(-1)^{\age(v)} |G_\sigma|} \prod_{i \in I'_\sigma} \Gamma\left(\frac{w_{i, \sigma}}{z} + c_i(v) \right).
$$

\subsubsection{Structure sheaves of toric divisors}
Consider the toric divisor $\cD_j$ corresponding to the 1-cone $\rho_j$. We have
\begin{align*}
    \tch_{z, \bT'}(\cO_{\cD_j}) & = \tch_{z, \bT'}(\cO_{\cX}) - \tch_{z, \bT'}(\cO_{\cX}(-\cD_j)) \\
    & = \bigoplus_{v \in \Box(\cX)} \left(1 - \exp 2\pi\sqrt{-1} \left(\frac{\Dbar_j^{\bT'}}{z} - c_j(v)\right) \right) \one_v \\
    & = \bigoplus_{v \in \Box(\cX)} -2\sqrt{-1} \exp \pi\sqrt{-1} \left(\frac{\Dbar_j^{\bT'}}{z} - c_j(v)\right) \sin \pi \left(\frac{\Dbar_j^{\bT'}}{z} - c_j(v)\right) \one_v.
\end{align*}
For $\sigma \in \Sigma(3)$, $v \in \Box(\sigma)$, we have
\begin{align*}
    \inv^*\left(\tch_{z, \bT'}(\cO_{\cD_j})\right) \big|_{\fp_{\sigma,v}} & = -2\sqrt{-1} \exp \pi \sqrt{-1} \left(\frac{w_{j, \sigma}}{z} - c_j(\inv(v))\right) \sin\pi\left(\frac{w_{j, \sigma}}{z} - c_j(\inv(v))\right) \\
    & = -2\sqrt{-1} \exp \pi \sqrt{-1} \left(\frac{w_{j, \sigma}}{z} + c_j(v)\right) \sin\pi\left(\frac{w_{j, \sigma}}{z} + c_j(v)\right)
\end{align*}
which is non-zero if and only if $j \in I'_\sigma$. Moreover, we have
$$
    \sin\pi\left(\frac{w_{j, \sigma}}{z} + c_j(v)\right) \Gamma\left(\frac{w_{j, \sigma}}{z} + c_j(v) \right) = \frac{\pi}{\Gamma\left(1 - \frac{w_{j, \sigma}}{z} - c_j(v) \right)}.
$$
Therefore we have
\begin{equation}\label{eqn:aside-result-divisor}
    a_{\sigma, v}(\cO_{\cD_j}) = \begin{cases}
        \displaystyle \frac{(-2\pi \sqrt{-1})e^{\pi \sqrt{-1}(w_{j, \sigma}/z + c_j(v))} \Gamma\left(\frac{w_{k_1, \sigma}}{z} + c_{k_1}(v)\right)\Gamma\left(\frac{w_{k_2, \sigma}}{z} + c_{k_2}(v) \right)}{(-1)^{\age(v)} |G_\sigma| \Gamma\left(1 - \frac{w_{j, \sigma}}{z} - c_j(v) \right)} & \text{if } I'_\sigma = \{j, k_1, k_2\},\\
        0 & \text{if } j \not \in I'_\sigma.
    \end{cases}
\end{equation}
Adding a twist by a $\bT'$-equivariant line bundle $\cL = \cO_{\cX}\left(\sum_{s = 1}^{3+\fp'} r_s \cD_s \right)$ introduces an overall factor of $\inv^*(\tch_{z, \bT'}(\cL))$. We have
\begin{equation}\label{eqn:aside-result-divisor-twist}
    a_{\sigma, v}(\cO_{\cD_j} \otimes \cL ) = e^{-2\pi \sqrt{-1}\sum_{s = 1}^{3+\fp'} (r_s w_{s, \sigma}/z + r_s c_s(v))} a_{\sigma, v}(\cO_{\cD_j}).
\end{equation}


\subsubsection{Structure sheaves of toric curves}
Consider the toric curve $\fl_\tau$ corresponding to a 2-cone $\tau$, with $I'_\tau = \{j_1, j_2\}$. Then $\fl_{\tau} = \cD_{j_1} \cap \cD_{j_2}$. We have
\begin{align*}
    \tch_{z, \bT'}(\cO_{\fl_\tau}) &= \left(\tch_{z, \bT'}(\cO_{\cX}) - \tch_{z, \bT'}(\cO_{\cX}(-\cD_{j_1}))\right)\left(\tch_{z, \bT'}(\cO_{\cX}) - \tch_{z, \bT'}(\cO_{\cX}(-\cD_{j_2}))\right) \\
    &= \bigoplus_{v \in \Box(\cX)} (-2\sqrt{-1})^2 \exp \pi\sqrt{-1} \left(\frac{\Dbar_{j_1}^{\bT'} + \Dbar_{j_2}^{\bT'}}{z} - c_{j_1}(v)- c_{j_2}(v)\right) \\
    & \qquad \cdot \sin \pi \left(\frac{\Dbar_{j_1}^{\bT'}}{z} - c_{j_1}(v)\right) \sin \pi \left(\frac{\Dbar_{j_2}^{\bT'}}{z} - c_{j_2}(v)\right) \one_v.
\end{align*}
For $\sigma \in \Sigma(3)$, $v \in \Box(\sigma)$, we have
\begin{align*}
    \inv^*\left(\tch_{z, \bT'}(\cO_{\fl_\tau})\right) \big|_{\fp_{\sigma,v}} = &  (-2\sqrt{-1})^2  \exp \pi \sqrt{-1} \left(\frac{w_{j_1, \sigma} + w_{j_2, \sigma}}{z} + c_{j_1}(v) + c_{j_2}(v)\right) \\
    & \cdot \sin\pi\left(\frac{w_{j_1, \sigma}}{z} + c_{j_1}(v)\right) \sin\pi\left(\frac{w_{j_2, \sigma}}{z} + c_{j_2}(v)\right)
\end{align*}
which is non-zero if and only if $j_1, j_2 \in I'_\sigma$. In fact there are at most two (depending on whether $\tau \in \Sigma(2)_c$) such choices of $\sigma$. We have
\begin{equation}\label{eqn:aside-result-curve}
    a_{\sigma, v}(\cO_{\fl_\tau}) = \begin{cases}
        \displaystyle \frac{(-2\pi\sqrt{-1})^2 e^{\pi \sqrt{-1}((w_{j_1, \sigma} + w_{j_2, \sigma})/z + c_{j_1}(v) + c_{j_2}(v))} \Gamma\left(\frac{w_{k, \sigma}}{z} + c_k(v)\right)}{(-1)^{\age(v)} |G_\sigma| \Gamma\left(1 - \frac{w_{j_1, \sigma}}{z} - c_{j_1}(v) \right) \Gamma\left(1 - \frac{w_{j_2, \sigma}}{z} - c_{j_2}(v) \right)} & \text{if } I'_\sigma = \{j_1, j_2, k\},\\
        0 & \text{if } \{j_1, j_2\} \not \subset I'_\sigma.
    \end{cases}
\end{equation}
Adding a twist by $\cL = \cO_{\cX}\left(\sum_{s = 1}^{3+\fp'} r_s \cD_s \right)$, we have
\begin{equation}\label{eqn:aside-result-curve-twist}
    a_{\sigma, v}(\cO_{\fl_\tau} \otimes \cL ) = e^{-2\pi \sqrt{-1}\sum_{s = 1}^{3+\fp'} (r_s w_{s, \sigma}/z + r_s c_s(v))} a_{\sigma, v}(\cO_{\fl_\tau}).
\end{equation}


\subsubsection{Structure sheaves of torus fixed points}
Finally, consider the torus fixed point $\fp_\sigma$ corresponding to a 3-cone $\sigma$, with $I'_\sigma = \{j_1, j_2, j_3\}$. Similar to the above, for $v \in \Box(\sigma)$ we have
$$    
    a_{\sigma, v}(\cO_{\fp_\sigma}) = \frac{(-2\pi\sqrt{-1})^3}{|G_\sigma| \Gamma\left(1 - \frac{w_{j_1, \sigma}}{z} - c_{j_1}(v) \right) \Gamma\left(1 - \frac{w_{j_2, \sigma}}{z} - c_{j_2}(v) \right) \Gamma\left(1 - \frac{w_{j_3, \sigma}}{z} - c_{j_3}(v) \right)}
$$
and $a_{\sigma', v}(\cO_{\fp_\sigma}) = 0$ for any $\sigma' \neq \sigma$ and any $v$. Adding a twist by $\cL = \cO_{\cX}\left(\sum_{s = 1}^{3+\fp'} r_s \cD_s \right)$, we have
$$    
    a_{\sigma, v}(\cO_{\fp_\sigma} \otimes \cL ) = e^{-2\pi \sqrt{-1}\sum_{s = 1}^{3+\fp'} (r_s w_{s, \sigma}/z + r_s c_s(v))} a_{\sigma, v}(\cO_{\fp_\sigma}).
$$

The sheaves $\cO_{\fp_\sigma} \otimes \cL$ represent classes in $K^c_{\bT'}(\cX)$ and their central charges admit non-equivariant limits, as in Section \ref{sect:KCompact}. In the non-equivariant setting, for $\sigma\in \Si(3)$, consider
$$
    \cF_\sigma := \bigoplus_{\gamma \in G_\si^*} 
    \iota_{\sigma,*} \cL_\gamma = \bigoplus_{\gamma\in G_\si^*} \cO_{\fp_\sigma} \otimes \tilde{\cL}_\gamma
$$
where $\cL_\gamma$ is the line bundle on $B\fp_\si \cong BG_\sigma$ associated to $\gamma \in G_\sigma^*$, and
$\tilde{\cL}_\gamma$ is any line bundle on $\cX$ such that $\iota_\sigma^* \tilde{\cL}_\gamma =\cL_\gamma$. Then $\cF_\sigma$ represents a class in $K^c(\cX)$. When $\fp_\sigma$ is a scheme point and $G_\sigma$ is trivial, we have $\cF_\sigma = \cO_{\fp_\sigma}$.

\begin{lemma} \label{lem:StructureSheafPoint} 
For any $\si\in \Si(3)$, we have
$$
    Z^c(\cF_\sigma) = \left(-2\pi\sqrt{-1}\right)^3.
$$
\end{lemma}

Lemmas \ref{lem:NECentralInjective} and \ref{lem:StructureSheafPoint} imply that the class of $\cF_\sigma$ in $K^c(\cX)$ is independent of $\sigma$.

\begin{proof}[Proof of Lemma \ref{lem:StructureSheafPoint}]
We consider the non-equivariant limit of $I_{\sigma, v}$, $a_{\sigma, v}$ above without changing notation. For $v \in \Box(\sigma)$ corresponding to $h \in G_\sigma$ we have
$$
    a_{\sigma, v}(\cO_{\fp_\sigma}) = \frac{(-2\pi\sqrt{-1})^3}{|G_\sigma| \Gamma\left(1 - c_{j_1}(v) \right) \Gamma\left(1 - c_{j_2}(v) \right) \Gamma\left(1 - c_{j_3}(v) \right)},
$$
while for $\gamma \in G_\sigma^*$ the twisting is
$$
    a_{\sigma, v}(\cO_{\fp_\sigma} \otimes \tilde{\cL}_\gamma ) = \chi_\gamma(h^{-1}) a_{\sigma, v}(\cO_{\fp_\sigma}).
$$
The lemma follows from that for any twisted sector $v$, the sum of $\chi_\gamma(h^{-1})$ over $\gamma \in G_\sigma^*$ is zero, while for the untwisted sector $v$, we have the non-equivariant limit $I_{\sigma, v} = 1$.
\end{proof}

\section{Mirror curves and oscillatory integrals}\label{sect:Integrals}
The mirror of the toric Calabi-Yau $3$-orbifold $\cX$ is an affine curve in $(\bC^*)^2$. In this section, we give its definition and study oscillatory integrals along relative 1-cycles. In particular, we construct cycles mirror to generators of the equivariant $K$-group $K^\pm_{\bT'}(\cX)$ and show that the integrals match the quantum cohomology central charges computed in Section \ref{sect:Sheaves}.


\subsection{Mirror curves} \label{sec:mirror-curve}

Given any 3-cone $\si\in \Si(3)$ and $i\in  \{1,\ldots, 3+\fp\}$, we define a monomial $a^{\si}_i(q)$ in $q=(q_1,\ldots,q_\fp)$ as follows:
$$
a_i^\si(q):=
\begin{cases}
1, & i\in I'_\si,\\
\prod_{a=1}^{\fp} q_a^{s^\si_{ai}}, & i\in I_\si
\end{cases}
$$
where $s^\si_{ai}$ are the non-negative integers defined by \eqref{eqn:sai}. Recall that $q_K = (q_1, \dots, q_{\fp'})$, $q_{\orb} = (q_{\fp'+1}, \dots, q_{\fp})$. Observe that:
\begin{itemize}
\item In the large complex structure limit $q_K\to 0$,
$$
\lim_{q_K\to 0} a_i^{\si}(q)=0 \qquad \text{if } i\in I_\si.
$$
\item If $\fp'+1\le a\le \fp$ then $s_{ai}^\si = \delta_{i, 3+a}$. So
$$
\lim_{q_{\mathrm{orb}}\to 0} a_i^{\si}(q)=0 \qquad \text{if }  i> 3+\fp'.
$$
\end{itemize}

In Section \ref{sect:Flags}, we have chosen a preferred flag $\bff_0 = (\tau_0,\si_0)$ to fix coordinates. Define $a_i(q):= a_i^{\si_0}(q)$, and define
$$
H(X,Y,q):=  \sum_{i=1}^{3+\fp} a_i(q) X^{m_i} Y^{n_i} = X^{\fr} Y^{-\fs} + Y^{\fm} + 1 + \sum_{i=4}^{3+\fp} a_i(q) X^{m_i} Y^{n_i}
$$
which is an element in the ring $\bZ[q_1,\ldots,q_{\fp}][ X, X^{-1}, Y, Y^{-1}]$.
The \emph{mirror curve} of $\cX$ is
$$
    C_q:=\{(X,Y)\in (\bC^*)^2: H(X,Y,q)=0\}.
$$
For fixed $q\in \bC^{\fp}$, $C_q$ is an affine curve in
$(\bC^*)^2$.

The polytope $P$ determines a toric surface $\bS_P$ with a polarization $L_P$, and $H(X,Y,q)$ extends to a section $s_q\in H^0(\bS_P,L_P)$. The compactified mirror curve $\Cbar_q\subset \bS_P$ is the zero locus of $s_q$. For generic $q\in \bC^{\fp}$, $\Cbar_q$ is a compact Riemann surface of genus $\fg$ and $\Cbar_q$ intersects
the anti-canonical divisor $\partial \bS_P=\bS_P\setminus (\bC^*)^2$ transversally at $\fn$ points, where recall $\fg$ and $\fn$ are the number of lattice points in the interior and the boundary of $P$, respectively.

Moreover, consider the covering map $\bC^2_{x,y} \to (\bC^*)^2_{X, Y}$ where
$$
    x = -\log X, \qquad y = -\log Y.
$$
The \emph{equivariant mirror curve}
\[
\tC_q= C_q\times_{(\bC^*)^2} \bC^2
\]
is given by the inclusion $C_q\hookrightarrow (\bC^*)^2$ and the cover $\bC^2 \to (\bC^*)^2$ above. The induced projection
$$
    p: \tC_q \to C_q
$$
is a regular covering with fiber $\bZ^2$ which will be identified with $M'$. The mirror curve $C_q$ is connected, and by Lemma 5.10 of \cite{Yu25} 
the inclusion $C_q\hookrightarrow (\bC^*)^2$ induces a surjection $\pi_1(C_q)\to \pi_1( (\bC^*)^2)\cong \bZ^2$,  so 
$\tC_q$ is also connected. There is an action of $M'$ on $\tC_q$ given by 
$$
    x\mapsto  x - 2m_1 \pi \sqrt{-1}, \qquad y\mapsto y - 2 m_2 \pi \sqrt{-1} \qquad \text{for $m_1 \su_1 + m_2 \su_2 \in M'$,}
$$
which is the deck transformation on $\tC_q$.

For $u_1, u_2 \in \bC$, $u_1 \neq 0$, we define
$$
    \hx= u_1 x + u_2 y,\qquad \hy=\frac{y}{u_1}.
$$

\subsection{Reparameterizations}
We may reparameterize the mirror curve by coordinates specified by any flag $\bff = (\tau, \si)$. Continuing the notation in Section \ref{sect:Flags}, we define integers $a_{\bff}, b_{\bff}, c_{\bff}, d_{\bff}$ by
$$
e_1^{\bff} = a_{\bff} e_1 + b_{\bff} e_2,\quad
e_2^{\bff} = c_{\bff}e_1 + d_{\bff} e_2.
$$
Then $a_{\bff} d_{\bff}-b_{\bff}c_{\bff}=1$. For $i=1,\ldots, 3+\fp$, we define coordinates $(m_i^{\bff}, n_i^{\bff})$ by
$
b_i = m_i^{\bff} e_1^{\bff}  + n_i^{\bff} e_2^{\bff}  + e_3^{\bff}.
$
Note that $m_i^{\bff}$ and $n_i^{\bff}$ are determined by $\{ (m_i,n_i): i=1,\ldots,3+\fp\}$ and $\bff$. We define
$$
H_{\bff} (\XX , \YY, q):= \sum_{i=1}^{3+\fp} a_i^\si(q) \XX ^{m_i^{\bff}} \YY^{n_i^{\bff}}
$$
which is an element in the ring  $\bZ[q_1,\ldots,q_\fp] [ \XX, \XX^{-1}, \YY, \YY^{-1} ]$. Note that
$$
    H_{\bff}(\XX, \YY,0) = \XX ^{\fr_\bff} \YY ^{-\fs_\bff} + \YY^{\fm_\tau} +1.
$$

Then
$$
C_q\cong \{ (X_{\bff}, Y_{\bff})\in (\bC^*)^2: H_{\bff}(X_{\bff}, Y_{\bff},q)=0\}.
$$
More explicitly, for fixed $q\in (\bC^*)^{\fp}$ the isomorphism is induced by the following reparametrization of
$(\bC^*)^2$:
$$
    X_{\bff} = X^{a_{\bff}} Y^{b_{\bff}} a_{i_1^{\bff}}(q)^{\frac{1}{\fr_\bff}} a_{i_2^{\bff}}(q)^{\frac{\fs_\bff}{\fr_\bff \fm_\tau}} a_{i_3^{\bff}}(q)^{-\frac{\fm_\tau + \fs_\bff}{\fr_\bff \fm_\tau}},  \quad Y_{\bff} = X^{c_{\bff}} Y^{d_{\bff}} \left(\frac{a_{i_2^{\bff}}(q)}{a_{i_3^{\bff}}(q)}\right)^\frac{1}{\fm_\tau}.
$$
Under the above change of variables, we have
$$
    H(X,Y,q) = a_{i_3^{\bff}}(q) X^{m_{i_3^{\bff}}} Y^{m_{i_3^{\bff}}}H_{\bff}(X_{\bff}, Y_{\bff}, q).
$$

Let $\su_1^{\bff}, \su_2^{\bff}$ be the characters of $\bT'$ corresponding to $e_1^{\bff\vee}$, $e_2^{\bff\vee}$ viewed as elements of $M'$, and write $u_1^{\bff}, u_2^{\bff}$ for their values after the specialization $\su_1 = u_1$, $\su_2 = u_2$. Then
$$
    u_1^{\bff} = d_{\bff} u_1 - c_{\bff} u_2, \qquad u_2^{\bff} = -b_{\bff} u_1 + a_{\bff} u_2.
$$
For $j = 1, 2, 3$, writing $w_j^{\bff} = w_{i_j^\bff, \sigma}$, we have
\begin{equation}\label{eqn:w123}
    w_1^{\bff} = \frac{u_1^\bff}{\fr_\bff}, \qquad w_2^{\bff}=\frac{\fs_{\bff} u_1^{\bff} + \fr_{\bff} u_2^{\bff}}{\fr_{\bff}\fm_{\tau}}, \qquad w_3^{\bff} = - w_1^{\bff} - w_2^{\bff}.
\end{equation}
For the variables
$$
    x_{\bff} = -\log X_{\bff}, \qquad y_{\bff} = - \log Y_{\bff}
$$
on the cover, define
$$
    \hx_{\bff} = u_1^{\bff} x_{\bff} + u_2^{\bff} y_{\bff},\qquad \hy_{\bff} = \frac{y_{\bff}}{u_1^{\bff}}.
$$
We may compute that
$$
    \hx_\bff = \hx - c_{\si},\qquad c_\sigma= \sum_{j = 1}^3 w_j^{\bff} \log a_{i_j^{\bff}}(q) = \sum_{a=1}^{\fp'} w_{\sigma,a} \log q_a.
$$
In particular, $c_{\si}$ only depends on the 3-cone $\sigma$ in $\bff$. We call it the \emph{tropical distance} between $\sigma$ and the preferred 3-cone $\sigma_0$.

\subsection{Superpotential and ramification points}
\label{sec:ramification}
We say \emph{$q$ is sufficiently small} for a statement if there is an $M_0>0$ depending on $\cX$ (or equivalently, the fan $\Si$) only, such that when $|q|<M_0$ such statement holds. 

We consider
$
    \hx= u_1 x + u_2 y
$
as a multi-valued function on $C_q$. As discussed in \cite[Section 4.6(1)]{flz2020remodeling}, when $q$ is sufficiently small and $u_2/u_1$ is generic, the function $\hat x$ is holomorphic Morse and $d\hat x$ is a well-defined meromorphic form on $\Cbar_q$. Define
$$
    \Phi := \hy d\hx = y d \left( x + \frac{u_2}{u_1} y \right)
$$
which is a multi-valued holomorphic 1-form on $C_q$ obtained by restriction from $(\bC^*)^2$. 


The $\bT'$-perturbed equivariant superpotential $W^{\bT'}_q: (\bC^*)^3\to \bC$ is given by 
\[
W^{\bT'}_q(X, Y, Z) := H(X,Y,q)Z - u_1\log X - u_2 \log Y.
\]
Under the mirror map $\btau=\btau(q)$, there is an isomorphism of Frobenius algebras
\begin{equation}\label{eqn:FrobIso}
    QH^*_{\bT'}(\cX)\big|_{\su_1=u_1,\su_2=u_2, Q=1, \btau=\btau(q)}\cong \Jac(W^{\bT'}_q)
\end{equation}
where $\Jac(W^{\bT'}_q)$ is the Jacobian ring (\cite{CCIT20}, cf. \cite[Equation (4.15)]{flz2020remodeling}). By direct computation (\cite[Lemma 4.6]{flz2020remodeling}), the critical points $\{dW^{\bT'}_q=0\}$ are in one-to-one correspondence with ramification points $\{d\hat x=0\}$. Thus the ramification points of $\hx$ on $C_q$ may be labeled by the set $I_\Si$. In particular $\hat x$ has $2\fg - 2 + \fn$ ramification points. For $\bsi\in I_\Si$, we denote the corresponding critical point of $\hx$ by $p_\bsi$ and the critical value by $\check u^\bsi=\hat x(p_\bsi)$.


Now consider the covering $p: \tC_q \to C_q$. On $\tC_q$, $\hx$ is a holomorphic Morse function whose set of critical points is $p^{-1}(\{p_\bsi\}_{\bsi \in I_\Sigma})$. For each $\bsi \in I_\Sigma$, $p^{-1}(p_{\bsi})$ is an $M'$-torsor given by the deck transformation action. We now introduce an index set to label the critical points. For any $\si \in \Si(3)$, recall the short exact sequence \eqref{eqn:TPrimeAffine}
$$
    \xymatrix{
        0 \ar[r] &  M' \ar[r] &  M'_\si \ar[r]^{\psi_\si^\vee} & G_\si^* \ar[r] & 0.     
    }
$$
Under the standard basis $\su_1$ and $\su_2$, we identify $M'$ with $\bZ^2$ and $M'_\si$ with a rank-$2$ overlattice in $\bQ^2$.
Then, we may label the critical points of $\hx$ on $\tC_q$ by
$$
    \tI_\Sigma := \{ \tbsi = (\sigma, \chi) : \sigma \in \Sigma(3), \chi \in M'_\si \}
$$
in a way such that:
\begin{itemize}
    \item For the critical point $\tp_{(\sigma, \chi)} \in \tC_q$ labeled by $(\sigma, \chi)$, we have $p(\tp_{(\sigma, \chi)}) = p_{(\sigma, \psi_\si^\vee(\chi))} \in C_q$.
    
    \item The deck transformation by $m' \in M'$ takes $\tp_{(\sigma, \chi)}$ to $\tp_{(\sigma, \chi - m')}$. In particular, if $m' = (m_1, m_2)$, we have
    \begin{equation}\label{eqn:hxDeck}
        \hx(\tp_{(\sigma, \chi - m')}) = \hx(\tp_{(\sigma, \chi)}) - 2\pi\sqrt{-1}(u_1m_1 + u_2m_2).
    \end{equation}
\end{itemize}
For a more specific choice of labeling, see Section \ref{sec:mirror-curve-affine} below.


\subsection{Relative cycles and central charges}
On the mirror curve $C_q$, we consider the group of relative 1-cycles
$$
    H_1(C_q, \Re (\hx) \gg 0; \bZ).
$$
A basis of this group is given by the \emph{Lefschetz thimbles} $\{\gamma_{\bsi}\}_{\bsi \in I_\Sigma}$ of $\hx$ associated to the critical points $\{p_\bsi\}_{\bsi \in I_\Sigma}$, where $\gamma_\bsi$ is characterized by $\hx(\gamma_{\bsi}) = \hx(p_{\bsi}) + \bR_{\ge 0}$. Note that with $u_1, u_2$ fixed, different choices of branches of logarithms only affect the multivalued function $\hx$ on $C_q$ by imaginary constants, and thus do not affect the definitions of the subset $\{\Re (\hx) \gg 0\}$ and the Lefschetz thimbles.

On $\tC_q$, we consider the group of relative 1-cycles
$$
    H_1(\tC_q, \Re (\hx) \gg 0; \bZ)
$$
which admits an action of $M'$ by the deck transformation on $\tC_q$. The projection
$$
    p_*: H_1(\tC_q, \Re (\hx) \gg 0; \bZ) \to H_1(C_q, \Re (\hx) \gg 0; \bZ)
$$
along the covering map is invariant under the $M'$-action. A basis of this group is given by the Lefschetz thimbles $\{\gamma_{\tbsi}\}_{\tbsi \in \tI_\Sigma}$ of $\hx$ associated to the critical points $\{\tp_\tbsi\}_{\tbsi \in \tI_\Sigma}$, similar to above.

As $C_q$ and $\tC_q$ vary in family over sufficiently small $q$, the groups of relative 1-cycles form local systems. We now define central charges of flat relative cycles via oscillatory integrals of $\hy d\hx$. From now on we take $z \in \bR_{>0}$ with small radius.

\begin{definition}\rm{
The \emph{central charge} of a flat relative cycle $\gamma$ in $H_1(\tC_q, \Re (\hx) \gg 0; \bZ)$ or $H_1(C_q, \Re (\hx) \gg 0; \bZ)$ is defined by
$$
    I_\gamma := \int_\gamma e^{- \hx/z} \hy d\hx.
$$
}
\end{definition}

If $\gamma$ is a cycle on $C_q$, the integral is defined up to a constant depending on a choice of branches of $\hx, \hy$. Note that if we use coordinates specified by a general flag $\bff = (\tau, \si)$, we have
$$
    I_\gamma =  e^{-c_{\si}/z} \int_{\gamma} e^{-\hx_\bff/z} \hy_\bff d \hx_\bff.
$$

The oscillatory integrals are solutions to the $\bT'$-equivariant GKZ system \eqref{eqn:GKZ}.

\begin{proposition}\label{prop:IntSolvesPF}
For any flat relative cycle $\gamma$ in $H_1(\tC_q, \Re (\hx) \gg 0; \bZ)$ or $H_1(C_q, \Re (\hx) \gg 0; \bZ)$, we have $I_\gamma \in \bfS_{\bT'}$.
\end{proposition}

The above statement may be well-known. We include a proof in Appendix \ref{appdx:PF} for completeness. Indeed, as mentioned in Section \ref{sec:ramification}, \cite[Section 4.6]{flz2020remodeling} provided a dimensional reduction from the 3-dimensional Landau-Ginzburg model $((\bC^*)^3, W^{\bT'}_q)$ to the 1-dimensional Landau-Ginzburg model $(C_q, \hx)$, and one can show that the oscillatory integral
$$
    \int e^{-W^{\bT'}_q/z} \frac{dX}{X} \wedge \frac{dY}{Y} \wedge \frac{dZ}{Z}
$$
also satisfies the GKZ system \eqref{eqn:GKZ} e.g. using the argument in \cite[Proposition 5.1]{Iritani07}. 

As a consequence of Proposition \ref{prop:IntSolvesPF}, the central charge defines a homomorphism
$$
    H_1(\tC_q, \Re (\hx) \gg 0; \bZ) \to \bfS_{\bT'}, \qquad \gamma \mapsto I_\gamma.
$$

\begin{proposition}\label{prop:IntegralInjective}
The above central charge homomorphism is injective.
\end{proposition}


\begin{proof}
We consider the central charges of the basis of $H_1(\tC_q, \Re (\hx) \gg 0; \bZ)$ given by the Lefschetz thimbles $\{\gamma_\tbsi\}_{\tbsi \in \tI_\Sigma}$. By the stationary phase expansion, the central charge of $\gamma_\tbsi$ has the asymptotics
$$
    \int_{\gamma_\tbsi} e^{-\hx/z} \hy d\hx \sim z e^{-\hx(\tp_{\tbsi})/z} \left( \frac{2\pi z}{
    \frac{d^2\hx}{d\hy^2}(\tp_{\tbsi}) } \right)^{\frac{1}{2}} \left(1 + O(z) \right).
$$
By \eqref{eqn:hxDeck}, as $\tbsi$ ranges through $\tI_\Sigma$, there are no non-trivial linear relations among the exponential terms $\{e^{-\hx(\tp_{\tbsi})/z}\}_{\tbsi \in \tI_\Sigma}$ in the asymptotics as functions in $\frac{u_1}{z}, \frac{u_2}{z}$. Thus the central charges of the Lefschetz thimbles are linearly independent, which implies the proposition.
\end{proof}

In the rest of this section, we construct relative cycles that are mirror to generators of the equivariant $K$-group $K^+_{\bT'}(\cX)$ and have matching central charges. The cycles will be glued from local constructions with respect to a decomposition of the mirror curve according to affine toric charts of $\cX$. We note that the construction also applies to $K^-_{\bT'}(\cX)$ while the resulting 1-cycles are relative instead to the subset $\{\Re (\hx) \ll 0\}$ of $C_q$ or $\tC_q$.

\subsection{Mirror curves of affine charts}
\label{sec:mirror-curve-affine}
We first consider the mirror curve $C_\sigma$ of an affine chart $\cX_\si = [\bC^3 / G_\si]$ of $\cX$, for $\si \in \Sigma(3)$. We choose a reference flag $\bff=(\tau,\si)$ but will drop ``$\bff$'' in the superscripts or subscripts as we view the coordinates to be on $C_\sigma$. For instance, $I'_\sigma = \{i_1, i_2, i_3\}$ where $i_j = i_j^\bff$. We also have the notation $w_j = w_{i_j, \sigma}$, $j = 1,2, 3$. Recall from \eqref{eqn:w123} that $w_1= \frac{u_1}{\fr}$, $w_2= \frac{\fs u_1 +\fr u_2}{\fr\fm}$, $w_3 = -w_1 - w_2$. 

For this affine chart, the relevant parameters are in the subcollection $q_\sigma = (q_a)$ of $q_{\orb}$ where $3+a$ ranges through the set
$
    \{i \in I_\sigma : b_i \in \sigma \}.
$
We denote the size of this set by $\fp^\sigma$. The mirror curve equation is
$$1+X^\fr Y^{-\fs} +Y^\fm  +\sum_{i\in I_\si,b_i\in \si}q_{i-3} X^{m_i}Y^{n_i}=0.$$
When $\sigma$ is smooth, $\fp^\sigma = 0$ and $C_\sigma$ is a pair of pants. The stacky Picard group $G_\si^*$ of $\cX_\sigma$ acts on $C_\si$ freely. 

When $q_\sigma=0$, we may lift the action of $G_\si^*$ on $C_\si$ to the action of $M_\si'$ on $\tC_\si$, and the action of the sub-lattice $M'$ on $\tC_\sigma$ is the deck transformation. For the identity element $\be \in M_\si'$, we define the integration cycle $\gamma_\be\subset \tC_\si$ to be one of the connecting component of  $\{Y^\fm \in [0,-1]\}$ such that on this cycle 
\begin{equation}
    a_\be:= \Im(x)=-\Im(\log X)= \frac{\fm +\fs}{\fr\fm}\pi,\quad b_\be:=\Im(y)=-\Im(\log Y)= \frac{\pi}{\fm}.
\label{eqn:cycle-divisor}
\end{equation}
For $\chi = (m_1, m_2) \in M_\si'$, let $\gamma_\chi = (-\chi) \cdot \gamma_\be$ be the image of $\gamma_\be$ under the action of $-\chi$, on which we have
$$
    a_\chi=\Im(x)(\gamma_\chi) = a_\be + 2\pi m_1, \qquad b_\chi=\Im(y) (\gamma_\chi) = b_\be + 2\pi m_2.
$$
We also define the integration cycle $\gamma'_\be \subset \tC_\si$ such that on $\gamma'_\be$
\begin{align}
    \nonumber
&X^\fr Y^{-\fs}= -1+\exp(-\sqrt{-1}t),\quad Y^\fm = -\exp(-\sqrt{-1}t),\ t\in (0,2\pi),\\
&a'_\be=\Im(x)\vert_{t=\pi}= \frac{\pi}{\fr},\quad  b'_\be=\Im(y)\vert_{t=\pi}=0.
\label{eqn:cycle-curve}
\end{align}
Similarly, for $\chi = (m_1, m_2) \in M_\si'$, let $\gamma'_\chi= (-\chi) \cdot \gamma'_\be$ and
$$
    a'_\chi = a'_\be + 2\pi m_1, \qquad b'_\chi = b'_\be + 2\pi m_2.
$$
See Figure \ref{fig:affine-cycles} for an illustration. We extend the definition of $\gamma_\chi$ and $\gamma'_\chi$ over sufficiently small $q_\sigma$ by parallel transport.

On $C_\sigma$ (resp. $\tC_\sigma$), the critical points of $\hx$ are in one-to-one correspondence with $G_\sigma^*$ (resp. $M_\si'$). For $\chi \in M_\si'$, the cycle $\gamma_\chi$ passes through a unique critical point of $\hx$, which we label by $\tp_{(\sigma, \chi)}$. This labeling induces a labeling of the critical points on $C_\sigma$ and satisfies the requirements in Section \ref{sec:ramification}. In addition, when $q = 0$, the action of $\chi' \in M_\si'$ takes $\tp_{(\sigma, \chi)}$ to $\tp_{(\sigma, \chi - \chi')}$. When $\Re(w_1)>0$, $\Re(w_2)>0$, $\Re(w_3)<0$, $\gamma_{\chi}$ is indeed the Lefschetz thimble of $\hx$ associated to $\tp_{(\sigma, \chi)}$; cf. \cite[Section 6.5]{flz2020affine}. When $\Re(w_1)>0$, $\Re(w_2)<0$, $\Re(w_3)<0$, $\gamma'_{\chi}$ is homologous to the Lefschetz thimble of $\hx$ associated to $\tp_{(\sigma, \chi)}$.

Now we consider the punctures of $C_\si$ and $\tC_\si$ where the integration cycles end. We will introduce a labeling of the punctures based on the commutative diagram for the flag $(\tau, \sigma)$
$$
    \xymatrix{
        M'_\si \ar[r]^\pi \ar[d]_{\psi_\si^\vee} & M'_\tau \ar[d] \\
        G^*_\si \ar[r] & G^*_\tau.
    }
$$
The set of punctures $Y^\fm=-1$ of $C_\si$ admits a $G_\tau^*$-action by permutation induced by the $G_\si^*$-action on $C_\si$. We label of these punctures by $G_\tau^*$ such that $(X=0, Y=\exp(-\frac{\pi\sqrt{-1}}{\fm}))$ is labeled by the identity $\be \in G_\tau^*$ and the other punctures are labeled by the permutation action. Similarly, the preimage of the all punctured disks near each puncture $Y^\fm=-1$ in $\tC_\si$ is diffeomorphic to $\bZ$-copies of $(0,\infty)\times\bR$. The set of connected components admits an $M'_\tau$-action by permutation induced by the $M'_\si$-action on $\tC_\si$. We label these components by $M'_\tau$ such that the identity $\be \in M'_\tau$ labels the component where $\Im(y)$ is close to $\frac{\pi}{m}$.

It is clear from the definition that the cycle $\gamma_\chi$ has one end point ending on the component labeled by $\pi(\chi)$ where $\pi: M'_\si \to M'_\tau$. The other end point is not among those given by $Y^\fm=-1$. Moreover, the cycle $\gamma'_\chi$ has two end points ending on the components labeled by the image of $\pi(\chi)$ and $\pi(\chi)+1$. Here we use that $M_\tau'= \Hom(\bZ \fm e_2, \bZ)\cong \bZ$ and the ``$+1$'' is given by the orientation obtained from the flag $(\tau, \si)$.


\begin{figure}
\begin{tikzpicture}
    \draw[->] (0.3,1) -- (0.3,3) node[left] {$-\Re (y)$};
    \draw[->] (1,0) -- (3,0) node[above] {$-\Re (x)$};

    \foreach \y in {1.5,1.75,2}  \draw (1,\y) ellipse [x radius=0.05, y radius=0.1];
    \foreach \y in {0.25,0.5}  \draw (1,\y) ellipse [x radius=0.05, y radius=0.1];
    \draw (2.5,2.5) ellipse [x radius=0.05, y radius=0.1];

    \draw (1,1.4) .. controls (1.2,1.2) and (1.2,0.9) ..  (1, 0.6);
    \draw (1,2.1) .. controls (1.5,2.1) and (2,2.3) ..  (2.48, 2.59);
    \draw (1.03,0.17) .. controls (1.5,1.1) and (2,1.8) ..  (2.52, 2.41);

    \draw (1,0.35) .. controls (1.02,0.37) and (1.02,0.38) ..  (1, 0.4);
    \foreach \y in {0,0.25 }\draw (1,1.6+\y) .. controls (1.02,1.62+\y) and (1.02,1.63+\y) ..  (1, 1.65+\y);

    \draw (1,0.4) node[left] {\small $\tau'$};
    \draw (1,1.7) node[left] {\small $\tau$};

    \draw[magenta] (1.05,1.75) .. controls (1.5,1.7) and (1.5,1.5) .. node[right] {\tiny $\gamma_\chi$} (1.05,0.5);
    \draw[cyan] (1.05,1.56) .. controls (1.5,1.7) and (1.5,1.79) .. node[right] {\tiny $\gamma'_{\chi}$} (0.95,1.79);
\end{tikzpicture}
\caption{Mirror curve $C_\si$ and cycles $\gamma_\chi$, $\gamma'_\chi$. The punctures are illustrated by circles, and the curve may have non-zero genus (not illustrated).}
\label{fig:affine-cycles}
\end{figure}
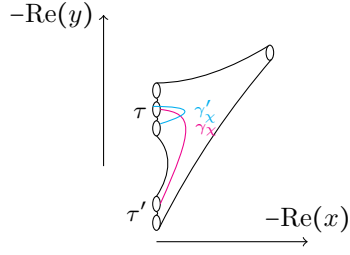

To conclude the local study, we evaluate the oscillatory integrals along the cycles defined above. We set up some notation. For the affine chart $\cX_\sigma$, the extended nef cone is generated by $\{D_i\}_{i \in I_\sigma, b_i \in \sigma }$. Dually, consider the classes
$
    \vec{h} = (h_i)_{i \in I_\sigma, b_i \in \sigma} \in (\bZ_{\ge 0})^{\fp^\sigma}
$
in the extended Mori cone. We write $q^{\vec{h}} = \prod_{i \in I_\sigma, b_i \in \sigma} q_{i-3}^{h_i}$. We define the quantities $c_{i_j}(\vec{h})$, $j = 1, 2, 3$ by
$$
    \sum_{j=1}^3 c_{i_j}(\vec{h}) b_{i_j}=\sum_{i \in I_\sigma, b_i \in \sigma} h_i b_i.
$$
The class $\vec{h}$ represents an element $v = v(\vec{h}) \in \Box(\sigma)$, which is also viewed as an element of $G_\sigma$. For $\chi \in M'_\si$, we write
$
    \chi(\vec{h}) := \psi_\si^\vee(\chi)(v)
$
for the evaluation of $\psi_\si^\vee(\chi) \in G_\si^*$ on $v \in G_\sigma$. Specifically, for the character $\sw_{i_j, \sigma}$, we have
$$
    \sw_{i_j, \sigma}(\vec{h}) = e^{2\pi\sqrt{-1}c_{i_j}(\vec{h})} = e^{2\pi\sqrt{-1}c_{i_j}(v)}.
$$


As in \cite[Theorem 7.6]{flz2020affine}, we compute for $\Re(w_1)>0,\Re(w_2)>0$ that 
\begin{align*}
    & I_{\gamma_\chi}  =\int_{\gamma_\chi} e^{-\hx/z}  \hy d \hx    \\
    & = - \sum_{\vec{h}\in(\bZ_{\ge 0})^{\fp^\si}}\frac{e^{-\sqrt{-1}(u_1 a_\chi + u_2 b_\chi)/z} e^{\pi \sqrt{-1}c_{i_3}(\vec{h})}(-\chi)(\vec{h})\Gamma\left(\frac{w_1}{z}+c_{i_1}(\vec{h})\right)\Gamma\left(\frac{w_2}{z}+c_{i_2}(\vec{h})\right)}{(-1)^{\sum{h_i}}|G_\si|\Gamma\left(1-\frac{w_3}{z}-c_{i_3}(\vec{h})\right)}q^{\vec{h}} \\
    & = - \sum_{v \in \Box(\sigma)}\frac{e^{-\sqrt{-1}(u_1 a_\chi + u_2 b_\chi)/z} e^{\pi \sqrt{-1}c_{i_3}(v)}(-\chi)(v)\Gamma\left(\frac{w_1}{z}+c_{i_1}(v)\right)\Gamma\left(\frac{w_2}{z}+c_{i_2}(v)\right)}{(-1)^{\age(v)}|G_\si|\Gamma\left(1-\frac{w_3}{z}-c_{i_3}(v)\right)}\left(q^{v} + O(|q_\sigma|^3) \right)
\end{align*}
where the last equality follows from the argument as in the derivation of \eqref{eqn:IsvLeading}. Here we orient $\gamma_\chi$ in a way opposite to \cite[Theorem 7.6]{flz2020affine} which gives the overall minus sign. Writing $\chi = (m_1, m_2)$, we have in addition that
$$
    -(u_1 a_\chi + u_2 b_\chi) = - (u_1 a_\be + u_2 b_\be) - 2\pi (u_1m_1 + u_2m_2) = \pi w_3 - 2\pi \chi \big|_{\su_1 = u_1, \su_2 = u_2}.
$$
It follows that
\begin{equation}\label{eqn:pair-divisor}
    I_{\gamma_\chi} = - e^{- \frac{2\pi\sqrt{-1}}{z} \chi \big|_{\su_1 = u_1, \su_2 = u_2}} \sum_{v \in \Box(\sigma)}\frac{(-\chi)(v)e^{\pi\sqrt{-1}(w_3/z + c_{i_3}(v))}\Gamma\left(\frac{w_1}{z}+c_{i_1}(v)\right)\Gamma\left(\frac{w_2}{z}+c_{i_2}(v)\right)}{(-1)^{\age(v)}|G_\si|\Gamma\left(1-\frac{w_3}{z}-c_{i_3}(v)\right)}\left(q^{v} + O(|q_\sigma|^3) \right).
\end{equation}

We also compute for $\Re(w_1)>0$ that
\begin{align*}
    &I_{\gamma'_\chi}  =\int_{\gamma'_\chi} e^{-\hx/z}  \hy d \hx    \\
    & = 2\pi\sqrt{-1}\sum_{\vec{h} \in(\bZ_{\ge 0})^{\fp^\si}}\frac{e^{-\sqrt{-1}(u_1 a'_\chi + u_2 b'_\chi)/z} e^{\pi \sqrt{-1}(c_{i_2}(\vec{h})+c_{i_3}(\vec{h}))}(-\chi)(\vec{h})\Gamma\left(\frac{w_1}{z}+c_{i_1}(\vec{h})\right)}{(-1)^{\sum{h_i}}|G_\si|\Gamma\left(1-\frac{w_2}{z}-c_{i_2}(\vec{h})\right)\Gamma\left(1-\frac{w_3}{z}-c_{i_3}(\vec{h})\right)}q^{\vec{h}} \\
    & = 2\pi\sqrt{-1}\sum_{v \in \Box(\sigma)}\frac{e^{-\sqrt{-1}(u_1 a'_\chi + u_2 b'_\chi)/z} e^{\pi \sqrt{-1}(c_{i_2}(v)+c_{i_3}(v))}(-\chi)(v)\Gamma\left(\frac{w_1}{z}+c_{i_1}(v)\right)}{(-1)^{\age(v)}|G_\si|\Gamma\left(1-\frac{w_2}{z}-c_{i_2}(v)\right)\Gamma\left(1-\frac{w_3}{z}-c_{i_3}(v)\right)} \left(q^{v} + O(|q_\sigma|^3) \right).
\end{align*}
We have in addition that
$$
    -(u_1 a'_\chi + u_2 b'_\chi) = - (u_1 a'_\be + u_2 b'_\be) - 2\pi \chi \big|_{\su_1 = u_1, \su_2 = u_2} = \pi(w_2 + w_3) - 2\pi \chi \big|_{\su_1 = u_1, \su_2 = u_2}.
$$
It follows that
\begin{equation}\label{eqn:pair-curve}
    I_{\gamma'_\chi} = 2\pi\sqrt{-1} e^{-\frac{2\pi\sqrt{-1}}{z} \chi \big|_{\su_1 = u_1, \su_2 = u_2}} \sum_{v \in \Box(\sigma)}\frac{(-\chi)(v) e^{\pi\sqrt{-1}((w_2 + w_3)/z + c_{i_2}(v)+c_{i_3}(v))} \Gamma\left(\frac{w_1}{z}+c_{i_1}(v)\right)}{(-1)^{\age(v)}|G_\si|\Gamma\left(1-\frac{w_2}{z}-c_{i_2}(v)\right)\Gamma\left(1-\frac{w_3}{z}-c_{i_3}(v)\right)} \left(q^{v} + O(|q_\sigma|^3) \right).
\end{equation}


\subsection{Decomposition}\label{sec:decomposition}
We now decompose the mirror curve $C_q$ into local patches as in Section \ref{sec:mirror-curve-affine} and tubes connecting them. For any flag $\bff=(\tau,\si)$ such that $\tau \in \Sigma(2)_c$ is interior, we consider
\[
    |x_\bff|< d_{\bff}:=\min_a\{\log|q_a|\}/K_\bff.
\]
We choose $K_\bff>0$ large enough such that for sufficiently small $q$, $d_\bff<\frac{1}{5}|x_{\bff}+x_{\bff'}|$. Here $\bff'=(\tau,\si')$ where $\si' \in \Sigma(3)$ is the other 3-cone containing $\tau$, and the sum between these two $x$-variables is a constant (rational linear combination of $\log q_a$).
Define for $\sigma \in \Sigma(3)$,
\[
C_\si^\circ=\bigcap_{\tau \in \Sigma(2)_c, \tau \subset \si}\{|x_\bff|\le d_\bff\},
\]
and for $\tau=\si\cap \si' \in \Sigma(2)_c$,
\[
C^\circ_\tau= \{|x_\bff|\ge d_\bff, |x_{\bff'}|\ge d_{\bff'}\}\cap \bigcap_{\tau' \subset \si,\tau'\neq \tau }\{ |x_{(\tau',\si)}|\le d_{(\tau',\si)}\}\cap \bigcap_{\tau' \subset \si',\tau'\neq \tau }\{ |x_{(\tau',\si')}|\le d_{(\tau',\si')}\}.
\]
Therefore, for sufficiently small $q$, we have the decomposition
\begin{equation}
    \label{eqn:curve-decomposition}
    C_q=\bigcup_{\si\in \Si(3)}{C_\sigma^\circ} \cup \bigcup_{\tau\in \Si(2)_c} C_\tau^\circ,
\end{equation}
and with $\tC^\circ_\alpha=C^\circ_\alpha \times_{(\bC^*)^2} \bC^2$ for $\alpha \in \Si(3)\cup \Si(2)_c$, we have
\begin{equation}
    \label{eqn:equiv-curve-decomposition}
    \tC_q=\bigcup_{\si\in \Si(3)}{\tC_\sigma^\circ} \cup \bigcup_{\tau\in \Si(2)_c} \tC_\tau^\circ.
\end{equation}
The ``connecting tube'' $C^\circ_\tau$ consists of $\fm_\tau$-many tubes, each homeomorphic to  $S^1\times [0,1]$, and $G_\tau^*$ permutes these components. The covering space $\tC_\tau^\circ$ consists of $\bZ$-many sheets, each homeomorphic to $\bR\times [0,1]$, and $\pi_0(\tC^\circ_\tau)$ is an $M'_\tau$-torsor. 
The covering space $\tC_\sigma^\circ$ is connected.

For any flag $\bff=(\tau, \si)$, let $(X,Y)$ be the associated local coordinates on the affine mirror curve $C_\si$ as in Section \ref{sec:mirror-curve-affine}, and $(X_\bff,Y_\bff)$ be the associated coordinates on $C_q$. When $q$ is sufficiently small, $Y_\bff$ is a local coordinate on $C_\bff^\circ$ with no ramification points, and if we set $Y=Y_\bff$, $X_\bff$ and $X$ both as multi-valued functions of $Y$ differ by
\begin{equation}\label{eqn:XDiffCond}
    X_\bff = X+O(|q|).
\end{equation}
We define an open embedding
\[
    s_{\bff}: C^\circ_\si \hookrightarrow C_\si, \qquad 
              (X_\bff,Y_\bff) \mapsto (X+O(|q|),Y).
\]

From now on we fix a sufficiently small $q$ such that the decomposition \eqref{eqn:curve-decomposition} holds and condition \eqref{eqn:XDiffCond} is met. We set
$$
    q(t,t')=((q_1)^t,\dots, (q_{\fp'})^t,(q_{\fp' +1})^{t'},\dots,(q_\fp)^{t'})
$$
for $t,t'\ge 1$. In particular, changing $q(1,t')$ to $q(t,t')$ rescales all $c_\si$ to $tc_\si$.

\subsection{Mirror cycle of a toric divisor}
\label{sec:divisor-cycle}
Consider the toric divisor $\cD_j$ corresponding to the 1-cone $\rho_j$, and let $\rho_{k_1},\dots, \rho_{k_l}$ be the adjacent $1$-cones in the counterclockwise order. We require that $\rho_{j},\rho_{k_i},\rho_{k_{i+1}}$ form a $3$-cone for $i = 1, \dots, l-1$, while $\rho_{j}, \rho_{k_l}, \rho_{k_1}$ may or may not form a $3$-cone. Denote these $3$-cones by $\si_1,\dots, \si_m$ where $m=l$ or $l-1$. In the former case we denote $k_{l+1} = k_1$, $\sigma_{l+1} = \sigma_1$ for convenience.

A $\bT'$-equivariant line bundle $\cL$ on $\cD_j$ determines a collection of characters 
$$\uchi=\{\chi_i\in {M'_{\si_i}}\},\qquad \chi_i=c_1^{\bT'}(\cL\vert_{\fp_{\si_i}}).$$
Such a collection of characters satisfies the condition
\[
\langle \chi_{i}-\chi_{i+1}, b_{k_{i+1}}-b_j\rangle=0,\qquad \text{$i=1,\dots,m$}
\]
along the intersections
$
    \tau_{i+1} := \sigma_i \cap \sigma_{i+1}, I'_{\tau_{i+1}} = \{j, k_{i+1}\}
$
and is called a \emph{twisted polytope} \cite{fltz14}. The condition ensures that under
\begin{equation}\label{eqn:DivisorAdjProj}
    \xymatrix{
        M'_{\si_i} \ar[r] & M'_{\tau_{i+1}} & \ar[l] M'_{\si_{i+1}},
    }
\end{equation}
the characters $\chi_i$, $\chi_{i+1}$ have the same image. Conversely, a twisted polytope uniquely determines a $\bT'$-equivariant line bundle $\cL_\uchi$ on $\cD_j$.

Fix the flags $\bff_i=(\si_i,\tau_{i+1})$, so that
$i_1^{\bff_i} = k_i$, $i_2^{\bff_i} = k_{i+1}$, $i_3^{\bff_i} = j$.
Given a twisted polytope $\uchi = \{\chi_i\}$ on $\cD_j$, we let $\gamma_{\bff_i}^\circ = s_{\bff_i}^{-1}(\gamma_{\chi_i})$ be the cycle in $\tC_\si^\circ$ defined in Section \ref{sec:mirror-curve-affine}, and define the cycle
$
    \gamma_\uchi
$
by patching together $\gamma_{\bff_i}^\circ$. We prove the following lemma to validate the patching, which is illustrated in Figure \ref{fig:divisor-cycle}.

\begin{lemma}
    The paths $\gamma^\circ_{\bff_i}$ and $\gamma^\circ_{\bff_{i+1}}$ connect to the same connected component of $\tC^\circ_{\tau_{i+1}}$. 
\end{lemma}

\begin{proof}
It suffices to consider the mirror curve for the fan with only two $3$-cones  $\sigma_i$ and $\sigma_{i+1}$ with twisted sector parameters set to zero. We write the mirror curve equation in $X,Y$ for the flag $\bff_i$
\[
    H_{\bff_i}=1+X^\fr Y^{-\fs} + Y^\fm + q X^{-\bar \fr} Y^{\bar \fs},
\]
while in $X',Y'$ for the flag $\bff_{i+1}$
\[
H_{\bff_{i+1}}=1+X'^{\fr'} Y'^{-\fs'} +Y'^{\fm'} + q' X^{-\bar\fr'} Y'^{\bar \fs'}.
\]
Here $\fm'=\gcd(\bar\fr,\bar\fs)$ and $\fm=\gcd(\fr',\fs')$. The change of coordinates is
\[
Y=X'^{\frac{\fr'}{\fm}} Y'^{-\frac{\fs'}{\fm}},\qquad
Y'=q^\frac{1}{\fm'}X^{-\frac{\bar \fr}{\fm'}} Y^{\frac{\bar \fs}{\fm'}}.
\]

First assume $\gamma_{\bff_i}^\circ= s_{\bff_i}^{-1}(\gamma_{\be})$, $\gamma_{\bff_{i+1}}^\circ= s_{\bff_{i+1}}^{-1}(\gamma_{\be})$ for the identity $\be\in M'_{\si_i}$ or $M'_{\si_{i+1}}$. The component with $Y^\fm=-1$ that $\gamma^\circ_{\bff_i}$ ends in is given by \eqref{eqn:cycle-divisor} as
\[  \Im(\log Y) \sim - \frac{\pi}{\fm},\]
The component with $Y^\fm=-1$ that $\gamma^\circ_{\bff_{i+1}}$ ends in is given by \eqref{eqn:cycle-divisor} as
\[
    \Im(\log X') \sim - \frac{\fm'+\fs'}{\fr'\fm'}\pi,\qquad
    \Im(\log Y') \sim - \frac{\pi}{\fm'}
\]
which also implies that
\[
    \Im(\log Y) \sim  \left( - \frac{\fm'+\fs'}{\fr'\fm'} \frac{\fr'}{\fm} + \frac{\fs'}{\fm\fm'}\right) \pi= -\frac{\pi}{\fm}.
\]
Thus they end on the same component in $\tC^\circ_{\tau_{i+1}}.$

The general case follows from that $M'_{\tau_{i+1}}$ permutes connected components of $\tC_{\tau_{i+1}}^\circ$, and thus condition \eqref{eqn:DivisorAdjProj} ensures that $\gamma^\circ_{\bff_i}$ and $\gamma^\circ_{\bff_{i+1}}$ connect to the same connected component of $\tC^\circ_{\tau_{i+1}}$. 
\end{proof}

\begin{figure}[h]
    \begin{tikzpicture}
        
        \foreach \y in {1.5,1.75,2}  \draw (1,\y) ellipse [x radius=0.05, y radius=0.1];
        \foreach \y in {0.25,0.5}  \draw (1,\y) ellipse [x radius=0.05, y radius=0.1];
        \draw (2.5,2.5) ellipse [x radius=0.05, y radius=0.1];
    
        \draw (2,3) node {$C^\circ_{\si_i}$};
        \draw (-1,3) node {$C^\circ_{\tau_{i+1}}$};
        \draw (-4,3) node {$C^\circ_{\si_{i+1}}$};
    
        \draw (1,1.4) .. controls (1.2,1.2) and (1.2,0.9) ..  (1, 0.6);
        \draw (1,2.1) .. controls (1.5,2.1) and (2,2.3) ..  (2.48, 2.59);
        \draw (1.03,0.17) .. controls (1.5,1.1) and (2,1.8) ..  (2.52, 2.41);
    
        \draw (1,0.35) .. controls (1.02,0.37) and (1.02,0.38) ..  (1, 0.4);
        \foreach \y in {0,0.25 }\draw (1,1.6+\y) .. controls (1.02,1.62+\y) and (1.02,1.63+\y) ..  (1, 1.65+\y);
    
        \draw[magenta] (1.05,1.75) .. controls (1.5,1.7) and (1.5,1.5) .. node[right] {\tiny $\gamma^\circ_{\bff_i}$} (1.05,0.5);

        \begin{scope}[xscale=-1,xshift=2cm]
    
            \foreach \y in {1.5,1.75,2}  \draw (1,\y) ellipse [x radius=0.05, y radius=0.1];
            
            \draw (2.5,2.5) ellipse [x radius=0.05, y radius=0.1];

            \draw (1,1.4) .. controls (1.1,1.3) and (1.2,1.3) ..  (1.98, 0.16);
            \draw (1,2.1) .. controls (1.5,2.1) and (2,2.3) ..  (2.48, 2.59);
            \draw (2.01,0.35) .. controls (1.5,1.1) and (2,1.8) ..  (2.52, 2.41);    
        
            \foreach \y in {0,0.25 }\draw (1,1.6+\y) .. controls (1.02,1.62+\y) and (1.02,1.63+\y) ..  (1, 1.65+\y);

            \draw[magenta] (1.05,1.75) .. controls (1.5,1.6) and (1.5,1.4) .. node[above] {\tiny $\gamma_{\bff_{i+1}}^\circ$} (1.96,0.3);
        
        \end{scope}
        \foreach \y in {0.25}  \draw (-4,\y) ellipse [x radius=0.05, y radius=0.1];

        \foreach \y in {1.5,1.75,2}  \draw (0,\y) ellipse [x radius=0.05, y radius=0.1];
        \foreach \y in {1.5,1.75,2}  \draw (-2,\y) ellipse [x radius=0.05, y radius=0.1];
        \foreach \y in {1.5,1.75,2}  \draw (-2,\y+0.1) -- (0,\y+0.1);
        \foreach \y in {1.5,1.75,2}  \draw (-2,\y-0.1) -- (0,\y-0.1);
        \draw[magenta] (-1.95,1.75) -- (-1.5,1.75);
        \draw[magenta,dotted] (-1.5,1.75) -- (-0.7,1.75);
        \draw[magenta] (-0.7,1.75) -- (0.05,1.75);

    \end{tikzpicture}
    \caption{Patching cycles for the mirror of a toric divisor. The paths $\gamma^\circ_{\bff_i}$ and $\gamma^\circ_{\bff_{i+1}}$ connect to the same component of $\tC^\circ_{\tau_{i+1}}$. This component is homeomorphic to $\bR \times [0,1]$. Passing down to the component of $C^\circ_{\tau_{i+1}}$ (homeomorphic to $S^1\times [0,1]$) this may have a non-trivial winding (dotted magenta part).} 
    \label{fig:divisor-cycle}
\end{figure}

Now we evaluate the oscillatory integral on $\gamma_\uchi$, starting with the two estimation lemmas. Via the specialization $\su_1 = u_1$, $\su_2 = u_2$ of equivariant parameters, we regard $(u_1,u_2)$ as an element in $N'_\bC := N' \otimes \bC$. The tropical distance functions $c_\sigma$, $\sigma \in \Sigma(3)$, as linear functions in $u_1, u_2$ are viewed as elements in $M'_\bC := M' \otimes \bC$. In particular, the functions $\{\Re(c_{\si_i})\}_{i=1}^m$ are the vertices of a convex polygon in $M'_\bR$ if $m=l$, or a convex non-compact spline if $m=l-1$, such that no $\Re(c_\si)$ lies in the interior.

\begin{lemma}\label{lem:PolygonDirection}
    For $i = 1, \dots, m$, there is an open region $U_i\subset N'_\bR$ and $\epsilon>0$ such that when $(u_1,u_2)\in U_i$,
    \[
        \int_{\gamma_\uchi \setminus \gamma^\circ_{\bff_i}} e^{-\hx_{\bff_i}/z} \hy_{\bff_i} d\hx_{\bff_i} = O(\exp(-\epsilon t)).
    \]
\end{lemma}

\begin{proof}
For the $3$-cone $\si_i$, we can choose $u_1,u_2$ such that $\Re({\si_i})<\Re(c_{\si_{i'}})$ for all $i' \in\{1,\dots,m\}\setminus\{i\}$. Then we have $\exp(-\hx_{\bff_i}/z)=O(\exp(-\epsilon t))$ on each $\tC^\circ_{\si_{i'}}$, $i' \neq i$, as well as each $\tC^\circ_{\tau_{i'}}$, $i' = 2, \dots, m+1$. This implies the lemma.
\end{proof}


\begin{lemma}\label{lem:PolygonCorner}
    For $i = 1, \dots, m$, when $(u_1,u_2)\in U_i$,
    \[
        \int_{\gamma^\circ_{\bff_i}} e^{-\hx_{\bff_i}/z} \hy_{\bff_i} d\hx_{\bff_i} = \int_{\gamma_{\chi_i}\subset C_{\si_i}} e^{-\hx/z} \hy d\hx +        O(\exp(-\epsilon t)).
    \]
\end{lemma}
\begin{proof}
    Our local change of coordinates from $C_q$ to $C_{\sigma_i}$ specified at the end of Section \ref{sec:decomposition} satisfies that $X_{\bff_i}=X+O(|q|)$, $Y_{\bff_i} = Y$, which implies that $e^{-\hx_{\bff_i}/z} = e^{-\hx/z}+O(|q|^K)$ for some constant $K>0$. It then follows that the integration differs by $O(\exp(-\epsilon t))$ for some $\epsilon>0$.
\end{proof}


The above lemmas imply that for $i = 1, \dots, m$,  when $(u_1,u_2)\in U_i$, we have the estimate
\begin{equation}\label{eqn:Ichi-estimate}
    I_{\gamma_\uchi}= e^{-c_{\si_i}/z}\left(\int_{\gamma_{\chi_i}\subset C_{\si_i}} e^{-\hx/z} \hy d\hx + O(\exp(-\epsilon t))\right).
\end{equation}

To precisely evaluate the integral, we use Proposition \ref{prop:IntSolvesPF} and Lemma \ref{lem:ICoeffSoln} to write 
$$
    I_{\gamma_\uchi} = \sum_{\substack{\sigma \in \Sigma(3) \\ v \in \Box(\sigma)}} a_{\sigma, v}\left(\tfrac{u_1}{z}, \tfrac{u_2}{z}\right) I_{\sigma,v}
    = \sum_{\substack{\sigma \in \Sigma(3) \\ v \in \Box(\sigma)}} a_{\sigma, v}\left(\tfrac{u_1}{z}, \tfrac{u_2}{z}\right) e^{-c_\si/z} \left(q^v+O(-\exp(3t'))+O(-\exp(t)) \right)
$$
for some coefficient functions $a_{\sigma, v}$, where the second equality follows from \eqref{eqn:IsvLeading}. For any $\si'\notin \{\si_1,\dots,\si_m\}$, since $\{\Re(c_{\si_i})\}_{i=1}^m$ form a convex spline or polygon in $M'_{\bR}$ and $c_{\si'}$ is on its outside, there exists $i \in \{1, \dots, m\}$ such that for $(u_1,u_2)$ in the open subset $U_i$, we have $\Re(c_{\si'})<\Re(c_{\si_{i}}) < \Re(c_{\si_{i'}})$ for $i' \in\{1,\dots,m\}\setminus\{i\}$ as in the proof of Lemma \ref{lem:PolygonDirection}. Then we must have that $a_{\si',v}=0$ for all $v\in \Box(\si')$, since otherwise $a_{\sigma', v}I_{\sigma', v}$ dominates the contributions from the cones $\sigma_1, \dots, \sigma_m$ and invalidates the estimate \eqref{eqn:Ichi-estimate} for $i$. 

Now for $i = 1, \dots, m$ and $v \in \Box(\sigma_i)$, we may use \eqref{eqn:Ichi-estimate} and the explicit evaluation \eqref{eqn:pair-divisor} to determine the coefficient function
$$
a_{\sigma_i,v}\left(\tfrac{u_1}{z}, \tfrac{u_2}{z}\right)= - e^{-\frac{2\pi\sqrt{-1}}{z} \chi_i \big|_{\su_1 = u_1, \su_2 = u_2}}  \frac{(-\chi_i)(v)e^{\pi \sqrt{-1}(w_{j, \sigma_i}/z + c_j(v))} \Gamma\left(\frac{w_{k_i, \sigma_i}}{z} + c_{k_i}(v)\right)\Gamma\left(\frac{w_{k_{i+1}, \sigma_i}}{z} + c_{k_{i+1}}(v) \right)}{(-1)^{\age(v)} |G_{\sigma_i}| \Gamma\left(1 - \frac{w_{j, \sigma_i}}{z} - c_j(v) \right)}.
$$
If $\uchi$ is induced by the $\bT'$-equivariant line bundle $\cL = \cO_{\cX}\left(\sum_{s = 1}^{3+\fp'} r_s \cD_s \right)$, we have $\chi_i = \sum_{s = 1}^{3+\fp'} r_s \sw_{s, \sigma_i}$ and
\begin{equation}\label{eqn:integral-twisting-factor}
    e^{-\frac{2\pi\sqrt{-1}}{z} \chi_i \big|_{\su_1 = u_1, \su_2 = u_2}} = e^{-2\pi\sqrt{-1} \sum_{s = 1}^{3+\fp'} r_s w_{s, \sigma_i}/z}, \qquad (-\chi_i)(v) =  e^{-2\pi\sqrt{-1} \sum_{s = 1}^{3+\fp'} r_s c_s(v)}.
\end{equation}
Comparing with \eqref{eqn:aside-result-divisor}, \eqref{eqn:aside-result-divisor-twist}, we have the following result.

\begin{proposition}\label{prop:DivisorCycle}
We have
$$
    2\pi \sqrt{-1} I_{\gamma_{\uchi}}=  Z_{\bT'}(\cO_{\cD_j} \otimes \cL).
$$
\end{proposition}

\subsection{Mirror cycle of a toric curve}\label{sec:curve-cycle}
Consider the toric curve $\fl_\tau$ corresponding to a 2-cone $\tau$, with $I'_\tau = \{j_1, j_2\}$. If $\tau \in \Sigma(2)_c$, $\fl_\tau$ is compact and we let $\si_1,\si_2$ denote the two $3$-cones containing $\tau$. Otherwise $\fl_\tau$ is non-compact and $\tau$ is contained in a unique $3$-cone $\si_1$. Write $I'_{\si_i} = \{j_1, j_2, k_i\}$, and fix the flag $\bff_i=(\tau, \si_i)$, so that $i_1^{\bff} = k_i$, $\{i_2^{\bff}$, $i_3^{\bff}\} = \{j_1, j_2\}$.

When $\tau \not \in \Sigma(2)_c$ is exterior, a twisted polytope $\uchi$ on $\fl_\tau$ is given by a single character $\chi_1 \in M'_{\si_1}$. We define the cycle
$
    \gamma'_\uchi:= s^{-1}_{\bff_1}(\gamma'_{\chi_1}).
$
In the other case $\tau \in \Sigma(2)_c$ is interior, a twisted polytope $\uchi$ on $\fl_\tau$ is given by two characters $\chi_i \in M'_{\si_i}$ for $i=1,2$ satisfying
$$
    \langle \chi_1-\chi_2,b_{j_1}-b_{j_2}\rangle=0.
$$
The condition ensures that the images of $\chi_1,\chi_2$ agree under 
\begin{equation}\label{eqn:CurveAdjProj}
    \xymatrix{
        M'_{\si_1} \ar[r] & M'_\tau & \ar[l] M'_{\si_2}.
    }
\end{equation}
Let $\gamma'^\circ_{\si_i}:=s^{-1}_{\bff_i}(\gamma'_{\chi_i})$, and define the cycle
$
    \gamma'_\uchi
$
by patching together $\gamma'^\circ_{\si_i}$. We prove the following lemma to validate the patching, which is illustrated in Figure \ref{fig:curve-cycle}.

\begin{lemma}
The paths $\gamma'^\circ_{\si_1}$ and $\gamma'^\circ_{\si_{2}}$ end on the same two connected components of $\tC^\circ_{\tau}$.
\end{lemma}
\begin{proof}
We write the mirror curve equation in $X,Y$ for the flag $\bff_1=(\tau, \si_1)$
\[
    H_{\bff_1}=1+X^\fr Y^{-\fs} + Y^\fm + q X^{-\bar \fr} Y^{\bar \fs},
\]
while in $X',Y'$ for the flag $\bff_2=(\tau, \si_2)$
\[
    H_{\bff_{2}}=Y'^{\fm}+q'X'^{-\bar\fr'} Y'^{\bar \fs'} +1 + X^{\fr'} Y'^{-\fs'}.
\]
We have the change of variables $y'=-y$. When both $\chi_i$ are the identity $\be \in M_{\si_i}$, the end points of the path $\gamma'^\circ_{\si_1}$ are given by \eqref{eqn:cycle-curve} as
\[
    \Im(\log Y) \sim \pm \frac{\pi}{\fm},
\]
while the end points of the path $\gamma'^\circ_{\si_2}$ are given by 
\[
    \Im(\log Y') \sim \pm \frac{\pi}{\fm},
\]
which are in the same components under the transformation $y'=-y$. The general case follows from that $M'_{\tau}$ permutes connected components of $\tC_{\tau}^\circ$.
\end{proof}

\begin{figure}[h]
    \begin{tikzpicture}
        
        \foreach \y in {1.5,1.75,2}  \draw (1,\y) ellipse [x radius=0.05, y radius=0.1];
        \foreach \y in {0.25,0.5}  \draw (1,\y) ellipse [x radius=0.05, y radius=0.1];
        \draw (2.5,2.5) ellipse [x radius=0.05, y radius=0.1];
    
        \draw (2,3) node {$C^\circ_{\si_1}$};
        \draw (-1,3) node {$C^\circ_{\tau}$};
        \draw (-4,3) node {$C^\circ_{\si_{2}}$};
    
        \draw (1,1.4) .. controls (1.2,1.2) and (1.2,0.9) ..  (1, 0.6);
        \draw (1,2.1) .. controls (1.5,2.1) and (2,2.3) ..  (2.48, 2.59);
        \draw (1.03,0.17) .. controls (1.5,1.1) and (2,1.8) ..  (2.52, 2.41);
    
        \draw (1,0.35) .. controls (1.02,0.37) and (1.02,0.38) ..  (1, 0.4);
        \foreach \y in {0,0.25 }\draw (1,1.6+\y) .. controls (1.02,1.62+\y) and (1.02,1.63+\y) ..  (1, 1.65+\y);
    
        \draw[cyan] (1.05,1.55) .. controls (1.5,1.7) and (1.5,1.79) .. node[right] {\tiny $\gamma'^\circ_{\si_1}$} (0.95,1.79);

        \begin{scope}[xscale=-1,xshift=2cm]
    
            \foreach \y in {1.5,1.75,2}  \draw (1,\y) ellipse [x radius=0.05, y radius=0.1];
            
            \draw (2.5,2.5) ellipse [x radius=0.05, y radius=0.1];

            \draw (1,1.4) .. controls (1.1,1.3) and (1.2,1.3) ..  (1.98, 0.16);
            \draw (1,2.1) .. controls (1.5,2.1) and (2,2.3) ..  (2.48, 2.59);
            \draw (2.01,0.35) .. controls (1.5,1.1) and (2,1.8) ..  (2.52, 2.41);    
        
            \foreach \y in {0,0.25 }\draw (1,1.6+\y) .. controls (1.02,1.62+\y) and (1.02,1.63+\y) ..  (1, 1.65+\y);

            \draw[cyan] (1.05,1.55) .. controls (1.5,1.7) and (1.5,1.79) .. node[left] {\tiny $\gamma'^\circ_{\si_2}$} (0.95,1.79);

        
        \end{scope}
        \foreach \y in {0.25}  \draw (-4,\y) ellipse [x radius=0.05, y radius=0.1];

        \foreach \y in {1.5,1.75,2}  \draw (0,\y) ellipse [x radius=0.05, y radius=0.1];
        \foreach \y in {1.5,1.75,2}  \draw (-2,\y) ellipse [x radius=0.05, y radius=0.1];
        \foreach \y in {1.5,1.75,2}  \draw (-2,\y+0.1) -- (0,\y+0.1);
        \foreach \y in {1.5,1.75,2}  \draw (-2,\y-0.1) -- (0,\y-0.1);
        \draw[cyan] (-1.95,1.8) -- (-1.5,1.8);
        \draw[cyan,dotted] (-1.5,1.8) -- (-0.7,1.8);
        \draw[cyan] (-0.7,1.8) -- (0.05,1.8);

        \draw[cyan] (-1.95,1.55) -- (-1.5,1.55);
        \draw[cyan,dotted] (-1.5,1.55) -- (-0.7,1.55);
        \draw[cyan] (-0.7,1.55) -- (0.05,1.55);

    \end{tikzpicture}
    \caption{Patching cycles for the mirror of a compact toric curve. The paths $\gamma'^\circ_{\si_1}$ and $\gamma'^\circ_{\si_{2}}$ connect to the same component of $\tC^\circ_{\tau}$. This component is homeomorphic to $\bR \times [0,1]$. Passing down to the component of $C^\circ_{\tau}$ (homeomorphic to $S^1\times [0,1]$) this may have a non-trivial winding (dotted cyan part).} 
    \label{fig:curve-cycle}.
\end{figure}

Similar to the discussion in Section \ref{sec:divisor-cycle}, for $i = 1$ (or $2$), there is an open region $U_i \subset N'_\bR$ and $\epsilon>0$ such that when $(u_1,u_2)\in U_i$,
\[
    I_{\gamma'_\uchi}= e^{-c_{\si_i}/z}\left(\int_{\gamma_{\chi_i}\subset C_{\si_i}} e^{-\hx/z} \hy d\hx +    O(\exp(-\epsilon t))\right).
\]
Writing 
$$
    I_{\gamma'_\uchi} = \sum_{\substack{\sigma \in \Sigma(3) \\ v \in \Box(\sigma)}} a_{\sigma, v}\left(\tfrac{u_1}{z}, \tfrac{u_2}{z}\right) I_{\sigma,v},
$$
we may use the explicit evaluation \eqref{eqn:pair-curve} to obtain that 
$$
    a_{\sigma_i, v}\left(\tfrac{u_1}{z}, \tfrac{u_2}{z}\right) = 2\pi\sqrt{-1} e^{-\frac{2\pi\sqrt{-1}}{z} \chi_i \big|_{\su_1 = u_1, \su_2 = u_2}} \frac{(-\chi_i)(v) e^{\pi \sqrt{-1}((w_{j_1, \sigma_i} + w_{j_2, \sigma_i})/z + c_{j_1}(v) + c_{j_2}(v))} \Gamma\left(\frac{w_{k_i, \sigma_i}}{z} + c_{k_i}(v)\right)}{(-1)^{\age(v)} |G_{\sigma_i}| \Gamma\left(1 - \frac{w_{j_1, \sigma_i}}{z} - c_{j_1}(v) \right) \Gamma\left(1 - \frac{w_{j_2, \sigma_i}}{z} - c_{j_2}(v) \right)}
$$
for $i = 1$ (or $2$), while $a_{\sigma', v} = 0$ for all other $\sigma' \neq \sigma_1$ (and $\sigma_2$). If $\uchi$ is induced by $\cL = \cO_{\cX}\left(\sum_{s = 1}^{3+\fp'} r_s \cD_s \right)$, we may also use the evaluation \eqref{eqn:integral-twisting-factor}. Comparing with \eqref{eqn:aside-result-curve}, \eqref{eqn:aside-result-curve-twist}, we have the following result.

\begin{proposition}\label{prop:CurveCycle}
We have
$$
    2\pi \sqrt{-1} I_{\gamma'_{\uchi}} = Z_{\bT'}(\cO_{\fl_{\tau}} \otimes \cL).
$$
\end{proposition}

\section{Genus-zero descendant mirror symmetry and integral structures}\label{sect:GenusZero}
In this section, we combine the ingredients from the previous sections into a mirror theorem for genus-zero equivariant descendant invariants of $\cX$. In the context of all-genus descendant mirror symmetry in Section \ref{sec:allgenus}, this will be used to treat the descendant leaf terms in the graph sums. We further discuss implications in the non-equivariant and compact support cases. In particular, we draw conclusions for the Hori-Vafa mirror of $\cX$ and resolve a conjecture of Hosono \cite{Hosono06}.

\subsection{Genus-zero descendant mirror theorem}\label{sect:EquivMir}
In Section \ref{sect:Sheaves}, we considered central charges of $\bT'$-equivariant sheaves with bounded below/above support. Recall that the $K$-group $K^+_{\bT'}(\cX)$ is a $K_{\bT'}(\cX)$-module and inherits the action of $M'$. We further evaluated the central charges on the class of generators \eqref{eqn:KGroupGenerator} of $K^+_{\bT'}(\cX)$ consisting of twists of structure sheaves of toric divisors and curves.

On the other hand, in Section \ref{sect:Integrals}, we considered central charges of relative cycles in the equivariant mirror curve $\tC_q$ defined by oscillatory integrals. The group $H_1(\tC_q, \Re (\hx) \gg 0; \bZ)$ of such cycles also admits an action of $M'$ via the deck transformation of the covering $p: \tC_q \to C_q$. Moreover, we constructed mirror cycles for the generators of $K^+_{\bT'}(\cX)$ above and evaluated the integrals on these cycles. Our construction and computation lead to the following result.

\begin{theorem}\label{thm:EquivMir}
There is a unique homomorphism
$$
    \mir^+_{\bT'}: K^+_{\bT'}(\cX) \to H_1(\tC_q, \Re (\hx) \gg 0; \bZ)
$$
such that for any $\cE \in K^+_{\bT'}(\cX)$, we have
\begin{equation}\label{eqn:EquivMir}
    2\pi\sqrt{-1}\int_{\mir^+_{\bT'}(\cE)} e^{-\hx/z} \hy d\hx = Z_{\bT'}(\cE) 
\end{equation}
under the mirror map $\btau=\btau(q)$. Moreover, $\mir^+_{\bT'}$ is $M'$-equivariant and is an isomorphism.
\end{theorem}

In other words, $\mir^+_{\bT'}$ identifies the two lattices embedded in $\bfS_{\bT'}$ via the respective central charges. 
Recall that the definition of the $K$-group $K^+_{\bT'}(\cX)$ depends on the choice of the cocharacter $\sv \in N'$ of $\bT'$. In fact, there is a wall-and-chamber structure on $N'_{\bR}$ and the definition depends only on the chamber that $\sv$ belongs to; see Remark \ref{apx-rmk:Cocharacter} for the general situation. On the other hand, the subset $\{\Re (\hx) \gg 0\} \subset \tC_q$ depends on $(u_1, u_2)$ which may be viewed as coordinates of an element in $N'$. The same wall-and-chamber structure applies in the sense that the relative homology group depends only on the chamber that $(u_1, u_2)$ belongs to. Indeed, underlying the isomorphism $\mir^\pm_{\bT'}$ we have $\sv = (u_1, u_2)$, or equivalently $u_1 = \inner{\su_1, \sv}$, $u_2 = \inner{\su_2, \sv}$. See Figure \ref{fig:wall-chamber} for an illustration in the example $\cX = K_{\bP^2}$. The parameters $\sv = (u_1, u_2)$ can be taken to be complex and the wall-and-chamber structure governs their real parts. Equation \eqref{eqn:EquivMir} holds for $(u_1, u_2)$ in an open region in $\bC^2$.

\begin{proof}
For a generator $\cE = \cO_{\cV(\sigma)} \otimes \cL$ of $K^+_{\bT'}(\cX)$ where $\cV(\sigma)$ is a toric divisor or curve and $\cL$ is a $\bT'$-equivariant line bundle, Propositions \ref{prop:DivisorCycle}, \ref{prop:CurveCycle} provide a cycle $\mir^+_{\bT'}(\cE) \in H_1(\tC_q, \Re (\hx) \gg 0; \bZ)$ such that \eqref{eqn:EquivMir} is satisfied. By Proposition \ref{prop:IntegralInjective}, equation \eqref{eqn:EquivMir} in fact uniquely characterizes the cycle. Extending the definition linearly gives the homomorphism $\mir^+_{\bT'}$.

Equation \eqref{eqn:EquivMir} further implies that $\mir^+_{\bT'}$ is $M'$-equivariant. The injectivity of $\mir^+_{\bT'}$ follows from the injectivity of the $\bT'$-equivariant central charge $K^+_{\bT'}(\cX) \to \bfS_{\bT'}$ which is a consequence of Proposition \ref{prop:PMCentralInjective}. The surjectivity will be treated in Lemma \ref{lem:MirSurj} below.
\end{proof}

\begin{remark}\label{rmk:MinusCase}\rm{
Analogously, there is an isomorphism
$$
    \mir^-_{\bT'}: K^-_{\bT'}(\cX) \to H_1(\tC_q, \Re (\hx) \ll 0; \bZ)
$$
satisfying similar properties. This also applies to Corollary \ref{cor:NonEquivMir} below in the non-equivariant case and Theorem \ref{thm:All-genus-mirror} below in higher genus.
}\end{remark}

\begin{figure}
    \begin{tikzpicture}
        \draw (-6.4,2.4) -- (-1.6,-2.4);
        \draw (-7,-1.5) -- (-1,1.5);
        \draw (-5.5,-3) -- (-2.5,3);


        \coordinate (S1) at (-1.2,2.8);

        \draw[draw = blue, fill = blue!15] (S1) -- ($(S1)+(0,-0.5)$) -- ($(S1)+(-0.5,0)$) -- (S1);
        \draw[blue] (S1) -- ($(S1)+(0.5,0.5)$);
        \draw[dashed] ($(S1)+(0,-0.5)$) -- ($(S1)+(0.3,-1.1)$);
        \draw[dashed] ($(S1)+(-0.5,0)$) -- ($(S1)+(-1.1,0.3)$);

        \coordinate (S2) at (-4.4,2.8);

        \draw[draw = blue!15, fill = blue!15] (S2) -- ($(S2)+(-0.5,0)$) -- ($(S2)+(-1.1,0.3)$) -- ($(S2)+(0.5,0.5)$) -- (S2);

        \draw[draw = blue, fill = blue!15] (S2) -- ($(S2)+(0,-0.5)$) -- ($(S2)+(-0.5,0)$) -- (S2);
        \draw[blue] (S2) -- ($(S2)+(0.5,0.5)$);
        \draw[dashed] ($(S2)+(0,-0.5)$) -- ($(S2)+(0.3,-1.1)$);
        \draw[blue] ($(S2)+(-0.5,0)$) -- ($(S2)+(-1.1,0.3)$);

        \coordinate (S3) at (-6.3,0.7);

        \draw[draw = blue, fill = blue!15] (S3) -- ($(S3)+(0,-0.5)$) -- ($(S3)+(-0.5,0)$) -- (S3);
        \draw[dashed] (S3) -- ($(S3)+(0.5,0.5)$);
        \draw[dashed] ($(S3)+(0,-0.5)$) -- ($(S3)+(0.3,-1.1)$);
        \draw[blue] ($(S3)+(-0.5,0)$) -- ($(S3)+(-1.1,0.3)$);

        \coordinate (S4) at (-6.3,-2.3);

        \draw[draw = blue!15, fill = blue!15] ($(S4)+(0,-0.5)$) -- ($(S4)+(-0.5,0)$) -- ($(S4)+(-1.1,0.3)$) -- ($(S4)+(0.3,-1.1)$) -- ($(S4)+(0,-0.5)$);

        \draw[draw = blue, fill = blue!15] (S4) -- ($(S4)+(0,-0.5)$) -- ($(S4)+(-0.5,0)$) -- (S4);
        \draw[dashed] (S4) -- ($(S4)+(0.5,0.5)$);
        \draw[blue] ($(S4)+(0,-0.5)$) -- ($(S4)+(0.3,-1.1)$);
        \draw[blue] ($(S4)+(-0.5,0)$) -- ($(S4)+(-1.1,0.3)$);

        \coordinate (S5) at (-3.3,-2.3);

        \draw[draw = blue, fill = blue!15] (S5) -- ($(S5)+(0,-0.5)$) -- ($(S5)+(-0.5,0)$) -- (S5);
        \draw[dashed] (S5) -- ($(S5)+(0.5,0.5)$);
        \draw[blue] ($(S5)+(0,-0.5)$) -- ($(S5)+(0.3,-1.1)$);
        \draw[dashed] ($(S5)+(-0.5,0)$) -- ($(S5)+(-1.1,0.3)$);

        \coordinate (S6) at (-1.2,-0.4);

        \draw[draw = blue!15, fill = blue!15] (S6) -- ($(S6)+(0,-0.5)$) -- ($(S6)+(0.3,-1.1)$) -- ($(S6)+(0.5,0.5)$) -- (S6);

        \draw[draw = blue, fill = blue!15] (S6) -- ($(S6)+(0,-0.5)$) -- ($(S6)+(-0.5,0)$) -- (S6);
        \draw[blue] (S6) -- ($(S6)+(0.5,0.5)$);
        \draw[blue] ($(S6)+(0,-0.5)$) -- ($(S6)+(0.3,-1.1)$);
        \draw[dashed] ($(S6)+(-0.5,0)$) -- ($(S6)+(-1.1,0.3)$);

        \draw[->, blue] (-4,0) -- (-3,0.1) node[right] {$\sv$};

        \node at (0,0) {$\longleftrightarrow$};

        \draw[->, orange] (4,0) -- (5,0.1) node[right] {\small $(u_1, u_2)$};

        \draw (6.4,-2.4) -- (1.6,2.4);
        \draw (7,1.5) -- (1,-1.5);
        \draw (5.5,3) -- (2.5,-3);

    
        \coordinate (C1) at (6.8,2.8);

        \draw ($(C1)+(-0.9,0.3)$) .. controls ($(C1)+(-0.5,0.1)$) and ($(C1)+(0,0.1)$) .. ($(C1)+(0.33,0.47)$);
        \draw ($(C1)+(0.3,-0.9)$) .. controls ($(C1)+(0.1,-0.5)$) and ($(C1)+(0.1,0)$) .. ($(C1)+(0.47,0.33)$);
        \draw ($(C1)+(-0.9,0.1)$) .. controls ($(C1)+(-0.5,-0.1)$) and ($(C1)+(-0.1,-0.5)$) .. ($(C1)+(0.1,-0.9)$);
        \draw ($(C1)+(-0.15,-0.1)$) ellipse [x radius=0.1, y radius=0.05];

        \draw ($(C1)+(-0.9,0.2)$) ellipse [x radius=0.05, y radius=0.1];
        \draw ($(C1)+(0.2,-0.9)$) ellipse [x radius=0.1, y radius=0.05];
        \draw[rotate=-45, draw=orange, very thick] ($(C1)+(-0.003,0.565)$) ellipse [x radius=0.1, y radius=0.05];

        \coordinate (C2) at (3.6,2.8);

        \draw ($(C2)+(-0.9,0.3)$) .. controls ($(C2)+(-0.5,0.1)$) and ($(C2)+(0,0.1)$) .. ($(C2)+(0.33,0.47)$);
        \draw ($(C2)+(0.3,-0.9)$) .. controls ($(C2)+(0.1,-0.5)$) and ($(C2)+(0.1,0)$) .. ($(C2)+(0.47,0.33)$);
        \draw ($(C2)+(-0.9,0.1)$) .. controls ($(C2)+(-0.5,-0.1)$) and ($(C2)+(-0.1,-0.5)$) .. ($(C2)+(0.1,-0.9)$);
        \draw ($(C2)+(-0.15,-0.1)$) ellipse [x radius=0.1, y radius=0.05];

        \draw[draw=orange, very thick] ($(C2)+(-0.9,0.2)$) ellipse [x radius=0.05, y radius=0.1];
        \draw ($(C2)+(0.2,-0.9)$) ellipse [x radius=0.1, y radius=0.05];
        \draw[rotate=-45, draw=orange, very thick] ($(C2)+(-0.003,0.565)$) ellipse [x radius=0.1, y radius=0.05];

        \coordinate (C3) at (1.7,0.7);

        \draw ($(C3)+(-0.9,0.3)$) .. controls ($(C3)+(-0.5,0.1)$) and ($(C3)+(0,0.1)$) .. ($(C3)+(0.33,0.47)$);
        \draw ($(C3)+(0.3,-0.9)$) .. controls ($(C3)+(0.1,-0.5)$) and ($(C3)+(0.1,0)$) .. ($(C3)+(0.47,0.33)$);
        \draw ($(C3)+(-0.9,0.1)$) .. controls ($(C3)+(-0.5,-0.1)$) and ($(C3)+(-0.1,-0.5)$) .. ($(C3)+(0.1,-0.9)$);
        \draw ($(C3)+(-0.15,-0.1)$) ellipse [x radius=0.1, y radius=0.05];

        \draw[draw=orange, very thick] ($(C3)+(-0.9,0.2)$) ellipse [x radius=0.05, y radius=0.1];
        \draw ($(C3)+(0.2,-0.9)$) ellipse [x radius=0.1, y radius=0.05];
        \draw[rotate=-45] ($(C3)+(-0.003,0.565)$) ellipse [x radius=0.1, y radius=0.05];

        \coordinate (C4) at (1.7,-2.3);

        \draw ($(C4)+(-0.9,0.3)$) .. controls ($(C4)+(-0.5,0.1)$) and ($(C4)+(0,0.1)$) .. ($(C4)+(0.33,0.47)$);
        \draw ($(C4)+(0.3,-0.9)$) .. controls ($(C4)+(0.1,-0.5)$) and ($(C4)+(0.1,0)$) .. ($(C4)+(0.47,0.33)$);
        \draw ($(C4)+(-0.9,0.1)$) .. controls ($(C4)+(-0.5,-0.1)$) and ($(C4)+(-0.1,-0.5)$) .. ($(C4)+(0.1,-0.9)$);
        \draw ($(C4)+(-0.15,-0.1)$) ellipse [x radius=0.1, y radius=0.05];

        \draw[draw=orange, very thick] ($(C4)+(-0.9,0.2)$) ellipse [x radius=0.05, y radius=0.1];
        \draw[draw=orange, very thick] ($(C4)+(0.2,-0.9)$) ellipse [x radius=0.1, y radius=0.05];
        \draw[rotate=-45] ($(C4)+(-0.003,0.565)$) ellipse [x radius=0.1, y radius=0.05];

        \coordinate (C5) at (4.7,-2.3);

        \draw ($(C5)+(-0.9,0.3)$) .. controls ($(C5)+(-0.5,0.1)$) and ($(C5)+(0,0.1)$) .. ($(C5)+(0.33,0.47)$);
        \draw ($(C5)+(0.3,-0.9)$) .. controls ($(C5)+(0.1,-0.5)$) and ($(C5)+(0.1,0)$) .. ($(C5)+(0.47,0.33)$);
        \draw ($(C5)+(-0.9,0.1)$) .. controls ($(C5)+(-0.5,-0.1)$) and ($(C5)+(-0.1,-0.5)$) .. ($(C5)+(0.1,-0.9)$);
        \draw ($(C5)+(-0.15,-0.1)$) ellipse [x radius=0.1, y radius=0.05];

        \draw ($(C5)+(-0.9,0.2)$) ellipse [x radius=0.05, y radius=0.1];
        \draw[draw=orange, very thick] ($(C5)+(0.2,-0.9)$) ellipse [x radius=0.1, y radius=0.05];
        \draw[rotate=-45] ($(C5)+(-0.003,0.565)$) ellipse [x radius=0.1, y radius=0.05];

        \coordinate (C6) at (6.8,-0.4);

        \draw ($(C6)+(-0.9,0.3)$) .. controls ($(C6)+(-0.5,0.1)$) and ($(C6)+(0,0.1)$) .. ($(C6)+(0.33,0.47)$);
        \draw ($(C6)+(0.3,-0.9)$) .. controls ($(C6)+(0.1,-0.5)$) and ($(C6)+(0.1,0)$) .. ($(C6)+(0.47,0.33)$);
        \draw ($(C6)+(-0.9,0.1)$) .. controls ($(C6)+(-0.5,-0.1)$) and ($(C6)+(-0.1,-0.5)$) .. ($(C6)+(0.1,-0.9)$);
        \draw ($(C6)+(-0.15,-0.1)$) ellipse [x radius=0.1, y radius=0.05];

        \draw ($(C6)+(-0.9,0.2)$) ellipse [x radius=0.05, y radius=0.1];
        \draw[draw=orange, very thick] ($(C6)+(0.2,-0.9)$) ellipse [x radius=0.1, y radius=0.05];
        \draw[rotate=-45, draw=orange, very thick] ($(C6)+(-0.003,0.565)$) ellipse [x radius=0.1, y radius=0.05];

    \end{tikzpicture}
    \caption{Wall-and-chamber structure on $N'_{\bR}$ for $\cX = K_{\bP^2}$. The walls are dual to the images of the three non-compact toric curves under the moment map $\mu_{\bT'_{\bR}}: \cX \to M'_{\bR}$. The blue parts on the left illustrate the images of $\cX^+$ under $\mu_{\bT'_{\bR}}$ for the different chambers. The orange parts on the right highlight the punctures (on $C_q$) contained in $\{\Re (\hx) \gg 0\}$.}
    \label{fig:wall-chamber}
\end{figure}
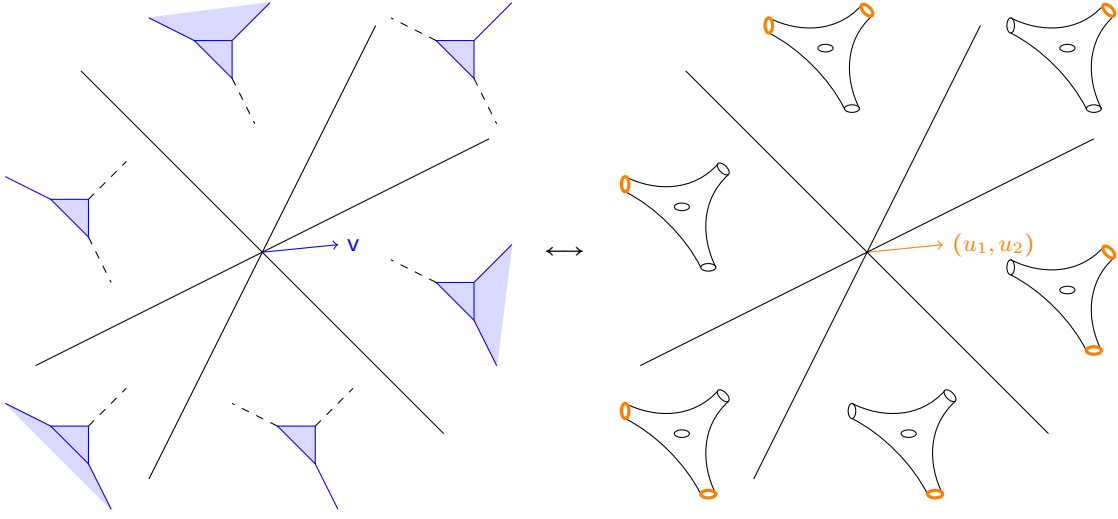

\subsection{Non-equivariant case}
We may take the ``non-equivariant limit'' of the map $\mir^+_{\bT'}$ constructed in Theorem \ref{thm:EquivMir} and obtain the following statement.

\begin{corollary}\label{cor:NonEquivMir}
There is a unique homomorphism
$$
    \mir^+: K^+(\cX) \to H_1(C_q, \Re (\hx) \gg 0; \bZ)
$$
that fits into the commutative diagram
$$
    \xymatrix{
        K^+_{\bT'}(\cX) \ar[r]^-{\mir^+_{\bT'}} \ar[d] & H_1(\tC_q, \Re (\hx) \gg 0; \bZ) \ar[d]^{p_*} \\
        K^+(\cX) \ar[r]^-{\mir^+} & H_1(C_q, \Re (\hx) \gg 0; \bZ).
    }
$$
Moreover, $\mir^+$ is an isomorphism.
\end{corollary}

Note that the two abelian groups $K^+(\cX)$, $H_1(C_q, \Re (\hx) \gg 0; \bZ)$ both have rank equal to $\dim_{\bC} H^*_{\CR}(\cX;\bC)$ (cf. Lemmas \ref{apx-lem:HPMDim}, \ref{apx-lem:PMChern}).

\begin{proof}
On a generator $\cE = \cO_{\cV(\sigma)} \otimes \cL$ of $K^+(\cX)$ as in \eqref{eqn:KGroupGenerator} where $\cL$ is a non-equivariant line bundle, we choose a $\bT'$-equivariant lift $\cE_{\bT'}$ of $\cE$ by endowing $\cL$ with a $\bT'$-equivariant structure, and set
$$
    \mir^+(\cE) := p_*(\mir^+_{\bT'}(\cE_{\bT'})).
$$
This definition in fact does not depend on the choice of $\bT'$-equivariant lift since $M'$ acts transitively on the choices of lifts, $\mir^+_{\bT'}$ is $M'$-equivariant, and the projection $p_*$ respects the $M'$-action on $H_1(\tC_q, \Re (\hx) \gg 0; \bZ)$. We thus obtain the desired homomorphism $\mir^+$. That $\mir^+$ is an isomorphism follows directly from the properties of $\mir^+_{\bT'}$ in Theorem \ref{thm:EquivMir}.
\end{proof}

\begin{remark}\label{rmk:NonEquivMirIso}\rm{
Taking the ``$-$''-case into account (Remark \ref{rmk:MinusCase}), we conjecture that the isomorphisms $\mir^\pm, \mir^\pm_{\bT'}$ intertwine the Euler characteristic pairings \eqref{eqn:EulerPairing} on the $K$-groups and the intersection pairings on the relative homology groups (given by Lefschetz duality).
}    
\end{remark}

\begin{remark}\rm{
Suppose $\cX$ is smooth. By Viro's patchworking technique for real algebraic curves and amoebas, when $q$ is sufficiently small, real, and with prescribed signs, the real locus $\bR C_q$ of the mirror curve $C_q$ has the maximum possible number of connected components $3+\fp$, and their images under the tropicalization map
$$
    \mu: C_q \subset (\bC^*)^2 \to \bR^2
$$
are precisely the connected components of the boundary of the image $\mu(C_q)$, or the amoeba; see e.g. \cite{IV96,Mikhalkin00,Viro01,Gu22}. The components of $\partial\mu(C_q)$ (or the regions they bound in $\bR^2 \setminus \mu(C_q)$) are in bijection with the $3+\fp$ lattice points in the polytope $P$. In this situation, we note that the (non-equivariant) mirror cycles on $C_q$ of the structure sheaves of the toric divisors $\cO_{\cD_j}$, $j = 1, \dots, 3+\fp$ we constructed are precisely the components of $\bR C_q$ or $\partial \mu(C_q)$ and this identification agrees with the above bijection. The (non-equivariant) mirror cycles of the structure sheaves of toric curves are (branched) double covers of chords in $\mu(C_q)$ between boundary components.
}
\end{remark}

\subsection{Compact support case}
Now we discuss the implications of our constructions for sheaves with compact support. Recall from \eqref{eqn:KNaturalMaps} that there is a natural map $K^c_{\bT'}(\cX) \to K^+_{\bT'}(\cX)$ which is injective by Lemma \ref{apx-lem:EquivKInjective}.

\begin{proposition}\label{prop:KcEquivMirror}
The image of $K^c_{\bT'}(\cX)$ under the composition
$$
    \xymatrix{
        K^c_{\bT'}(\cX) \ar[r] & K^+_{\bT'}(\cX) \ar[r]^-{\mir^+_{\bT'}} & H_1(\tC_q, \Re (\hx) \gg 0; \bZ)
    }
$$
is equal to the subgroup $H_1(\tC_q; \bZ)$ of absolute 1-cycles. In particular, we have an isomorphism
$$
    \mir^c_{\bT'}: K^c_{\bT'}(\cX) \to H_1(\tC_q; \bZ)
$$
under which \eqref{eqn:EquivMir} holds.
\end{proposition}

\begin{proof}
Here we show that the image is contained in $H_1(\tC_q; \bZ)$; the surjectivity will be treated in Lemma \ref{lem:MirSurj} below. By Lemma \ref{lem:KGroupGenerator}, $K^c_{\bT'}(\cX)$ is generated by sheaves of form $\cE = \cO_{\cV(\sigma)} \otimes \cL$ where $\sigma$ is minimal in $\Sigma^c$. We use $\cE$ to also denote its image in $K^+_{\bT'}(\cX)$ by an abuse of notation. As we observed, when $\cX$ is not affine, $\cV(\sigma)$ is a compact toric divisor or curve and it is clear from the construction that $\mir^+_{\bT'}(\cE)$ is contained in $H_1(\tC_q; \bZ)$.

Now consider the case when $\cX$ is affine and $\sigma = \sigma_0$ is the unique 3-cone. Let $\tau \in \Sigma(2)$ be a facet of $\sigma$ such that $\tau \in \Sigma^+$, and let $j$ be the unique index in $I_{\sigma}' \setminus I_\tau'$. Then $\fp_{\sigma} = \fl_\tau \cap \cD_{j}$ and we have the relation
$$
    [\cE] = [\cO_{\fl_\tau} \otimes \cL] - [\cO_{\fl_\tau} \otimes \cL(-\cD_j)]
$$
in $K^+_{\bT'}(\cX)$. As in Section \ref{sec:curve-cycle}, since $\tau$ is exterior, $\cO_{\fl_\tau} \otimes \cL$ is determined by the character $\chi := c_1^{\bT'}(\cL\vert_{\fp_{\si}}) \in M'_{\si}$ and $\cO_{\fl_\tau} \otimes \cL(-\cD)$ is determined by $c_1^{\bT'}(\cL(-\cD_j) \vert_{\fp_{\si}}) = \chi - \sw_{j, \sigma}$. On the other hand, the kernel of $\pi: M'_{\si} \to M'_{\tau}$ is spanned by $\sw_{j, \sigma}$. By the affine case discussed in Section \ref{sec:mirror-curve-affine}, the mirror cycles
$$
    \mir^+_{\bT'}(\cO_{\fl_\tau} \otimes \cL) = \gamma'_{\chi}, \qquad \mir^+_{\bT'}(\cO_{\fl_\tau} \otimes \cL(-\cD)) = \gamma'_{\chi - \sw_{j, \sigma}}
$$
end on the same two connected components of $\tC^\circ_{\tau}$. Thus their difference belongs to $H_1(\tC_q; \bZ)$.
\end{proof}

Now consider the following commutative diagram as in \cite[Section 4.4]{flz2020remodeling}
\begin{equation}\label{eqn:RelSeq}
    \xymatrix{
        0 \ar[r] & H_2(\bC^2; \bZ) = 0 \ar[r] \ar[d] & H_2(\bC^2, \tC_q; \bZ) \ar[r]^-\partial_-\cong \ar[d] & H_1(\tC_q; \bZ) \ar[r] \ar[d] & 0\\
        0 \ar[r] & H_2((\bC^*)^2; \bZ) \ar[r] & H_2((\bC^*)^2, C_q; \bZ) \ar[r]^-\partial & K_1(C_q; \bZ) \ar[r] & 0
    }
\end{equation}
where the rows are given by long exact sequences of homology and are exact. The vertical maps are given by the pushforward along the covering $p$. In particular, $K_1(C_q; \bZ)$ is the image of $p_*: H_1(\tC_q; \bZ) \to H_1(C_q; \bZ)$ and is identified with the kernel of $H_1(C_q; \bZ) \to H_1((\bC^*)^2; \bZ)$. Given any $\gamma \in K_1(C_q; \bZ)$, the integral
$
    \int_{\gamma} \hy d\hx
$
may be defined via choosing a lift to $H_1(\tC_q;\bZ)$ or $H_2((\bC^*)^2, C_q; \bZ)$ and is well-defined up to additive constants in $(2\pi \sqrt{-1})^2 \bZ$ \cite[Lemma 4.4]{flz2020remodeling}.


Composing the boundary map $\partial$ in the first row of \eqref{eqn:RelSeq} and the map $\mir^c_{\bT'}$ in Proposition \ref{prop:KcEquivMirror}, we obtain an isomorphism
$$
    \mir^c_{2, \bT'}: K^c_{\bT'}(\cX) \to H_2(\bC^2, \tC_q; \bZ)
$$
under which by \eqref{eqn:EquivMir}, for any $\cE \in K^c_{\bT'}(\cX)$, we have
\begin{equation}\label{eqn:Kc2EquivMir}
    -2\pi\sqrt{-1}\int_{\mir^c_{2, \bT'}(\cE)} e^{-\hx/z} d\hx \wedge d\hy = Z_{\bT'}(\cE).
\end{equation}

We now descend to the non-equivariant case. As explained in the proof of \cite[Proposition 14]{CLT13}, the results in \cite{Batyrev93,Stienstra98,KM10} produce an injective homomorphism
$$
    H_2((\bC^*)^2, C_q; \bZ) \to \bfS, \qquad \gamma \mapsto \int_\gamma \frac{dX}{X} \wedge \frac{dY}{Y}
$$
which is an isomorphism after tensoring the domain with $\bC$. Note that the subgroup $H_2((\bC^*)^2; \bZ)$ is isomorphic to $\bZ$ and generated by the class of the maximal compact subtorus $T^2 = U(1)^2 \subset (\bC^*)^2$. We have 
$$
    \int_{[T^2]} \frac{dX}{X} \wedge \frac{dY}{Y} = (2\pi\sqrt{-1})^2.
$$
Combining the above and Lemma \ref{lem:NECentralInjective} into the non-equivariant limit of \eqref{eqn:Kc2EquivMir}, we have the following result.

\begin{corollary}\label{cor:Kc2Mir}
There is a unique homomorphism
$$
    \mir^c_2: K^c(\cX) \to H_2((\bC^*)^2, C_q; \bZ)
$$
such that for any $\cE \in K^c(\cX)$, we have
$$    
    -2\pi\sqrt{-1}\int_{\mir^c_2(\cE)} \frac{dX}{X} \wedge \frac{dY}{Y} = Z^c(\cE).
$$
Moreover, $\mir^c_2$ is an isomorphism and for (any) $\sigma \in \Sigma(3)$, $\mir^c_2([\cF_\sigma]) = [T^2]$.
\end{corollary}

Similar to Theorem \ref{thm:EquivMir}, $\mir^c_2$ is an identification of two integral structures in $\bfS$. The composition of $\mir^c_2$ and the boundary map $\partial$ in the second row of \eqref{eqn:RelSeq} gives a surjective homomorphism
$$
    \mir^c: K^c(\cX) \to K_1(C_q; \bZ)
$$
such that for any $\cE \in K^c(\cX)$, we have
$$
    2\pi\sqrt{-1} \int_{\mir^c(\cE)} \hy d\hx \in Z^c(\cE) + (2\pi\sqrt{-1})^{3} \bZ.
$$
The kernel of $\mir^c$ is $\bZ[\cF_\sigma]$. The map $\mir^c$ is compatible with $\mir^c_{\bT'}$ in Proposition \ref{prop:KcEquivMirror} and $\mir^+$ in Corollary \ref{cor:NonEquivMir}.

\subsection{Hori-Vafa mirrors and Hosono's Conjecture}
We adapt the above discussion on the mirror curve to the \emph{Hori-Vafa mirror} of $\cX$, which is the family of Calabi-Yau 3-folds
$$
    \chcX_q := \{(u,v,X,Y)\in \bC^2 \times (\bC^*)^2: uv = H(X,Y,q)\}
$$
equipped with the holomorphic 3-form
$$
    \Omega_q := \Res_{\chcX_q} \left( \frac{1}{H(X,Y,q) - uv} du \wedge dv \wedge \frac{dX}{X} \wedge \frac{dY}{Y} \right).
$$
We refer to \cite[Section 4.4]{flz2020remodeling} for details on the reduction as genus-zero B-models from the Hori-Vafa mirror to the mirror curve. The projection to the $(X,Y)$-coordinates realizes $\chcX_q$ as a conic fibration over $(\bC^*)^2$ whose discriminant locus is the mirror curve $C_q = \{H(X,Y,q) = 0\}$. Based on this structure, \cite{Gross01,DK11,CLT13} constructed an isomorphism
$
    \alpha: H_2((\bC^*)^2, C_q; \bZ) \to H_3(\chcX_q; \bZ)
$
such that for any $\gamma \in H_2((\bC^*)^2, C_q; \bZ)$, we have
$$
    2\pi\sqrt{-1}\int_{\gamma} \frac{dX}{X} \wedge \frac{dY}{Y} = \int_{\alpha(\gamma)} \Omega_q.
$$
In particular, there is an injective homomorphism
$$
    H_3(\chcX_q; \bZ) \to \bfS, \qquad \gamma \mapsto \int_\gamma \Omega_q
$$
which is an isomorphism after tensoring the domain with $\bC$. Corollary \ref{cor:Kc2Mir} leads to the following result which resolves a conjecture of Hosono \cite[Conjecture 6.3]{Hosono06}.

\begin{theorem}\label{thm:Kc3Mir}
There is a unique homomorphism
$$
    \mir^c_3: K^c(\cX) \to H_3(\chcX_q; \bZ)
$$
such that for any $\cE \in K^c(\cX)$, we have
$$    
    -\int_{\mir^c_3(\cE)} \Omega_q = Z^c(\cE).
$$
Moreover, $\mir^c_3$ is an isomorphism.
\end{theorem}

As in Corollary \ref{cor:Kc2Mir}, $\mir^c_3$ is an identification of two integral structures in $\bfS$. We note that \cite[Conjecture 6.3]{Hosono06} also concerns the monodromies of the central charges over the parameter $q$ and we leave a treatment of this aspect to future work.

\subsection{Surjectivity}
Here we prove the surjectivity of the mirror cycle homomorphisms.

\begin{lemma}\label{lem:MirSurj}
The homomorphisms
$$
    \mir^+_{\bT'}: K^+_{\bT'}(\cX) \to H_1(\tC_q, \Re (\hx) \gg 0; \bZ), \quad 
    \mir^c_{\bT'}: K^c_{\bT'}(\cX) \to H_1(\tC_q; \bZ)
$$
in Theorem \ref{thm:EquivMir} and Proposition \ref{prop:KcEquivMirror} respectively are surjective.
\end{lemma}

\begin{proof}
We first establish both surjectivity statements in the case $\cX = [\bC^3/G_\sigma]$ is affine, where $\sigma = \sigma_0$ is the unique 3-cone. We choose a reference flag $\bff = (\tau, \sigma)$ and adopt the notation in Section \ref{sec:mirror-curve-affine}. As discussed there, depending on the signs of $\Re(w_1), \Re(w_2), \Re(w_3)$, the collection of mirror cycles $\{\gamma_\chi\}_{\chi \in M_\sigma'}$ or $\{\gamma'_\chi\}_{\chi \in M_\sigma'}$ recovers the collection of Lefschetz thimbles of $\hx$, which is a basis of $H_1(\tC_\sigma, \Re (\hx) \gg 0; \bZ)$. This proves the surjectivity of $\mir^+_{\bT'}$.

For $\mir^c_{\bT'}$, we view $H_1(\tC_\sigma; \bZ)$ as the kernel of the boundary map
$$
    \partial: H_1(\tC_\sigma, \Re (\hx) \gg 0; \bZ) \to H_0(\Re(\hx) \gg 0; \bZ).
$$ 
Up to cyclic symmetry there are two subcases. First suppose $\Re(w_1)>0, \Re(w_2)<0, \Re(w_3)<0$. The boundary map can be written as
$$
    \partial: \bigoplus_{\chi \in M'_\sigma} \bZ[\gamma'_\chi] \to \bigoplus_{\psi \in M'_{\tau_1}} \bZ e_\psi, \qquad [\gamma'_\chi] \mapsto -e_{\pi_1(\chi)} + e_{\pi_1(\chi)+1}
$$
where $e_\psi$ denotes the class of a point in the component of $\tC_{\tau_1}^\circ$ indexed by $\psi \in M'_{\tau_1}$ and $\pi_1$ is the projection $M'_\sigma \to M'_{\tau_1}$ whose kernel is $\bZ \sw_1$. As in the proof of Proposition \ref{prop:KcEquivMirror}, $K^c_{\bT'}(\cX)$ is generated by the sheaves $\cO_{\fp_\sigma} \otimes \cL$ whose mirror cycles take form 
$$
    [\gamma'_\chi] - [\gamma'_{\chi- \sw_1}], \qquad \chi \in M'_\sigma.
$$
It is a direct verification that such cycles generate the kernel of $\partial$.

In the second subcase suppose $\Re(w_1)>0, \Re(w_2)>0, \Re(w_3)<0$. The boundary map can be written as
$$
    \partial: \bigoplus_{\chi \in M'_\sigma} \bZ[\gamma_\chi] \to \bigoplus_{\psi \in M'_{\tau_1} \sqcup M'_{\tau_2}} \bZ e_\psi, \qquad [\gamma_\chi] \mapsto - e_{\pi_1(\chi)} + e_{\pi_2(\chi)}
$$
where $\pi_i: M'_\sigma \to M'_{\tau_i}$ is the projection whose kernel is $\bZ \sw_i$, $i = 1, 2$. Now since
$$
    [\cO_{\fp_\sigma} \otimes \cL] = [\cO_{\cD_3} \otimes \cL] - [\cO_{\cD_3} \otimes \cL(-\cD_1)] - [\cO_{\cD_3} \otimes \cL(-\cD_2)] + [\cO_{\cD_3} \otimes \cL(-\cD_1-\cD_2)],
$$
the mirror cycles of the sheaves $\cO_{\fp_\sigma} \otimes \cL$ take form 
$$
    [\gamma_\chi] - [\gamma_{\chi- \sw_1}] - [\gamma_{\chi- \sw_2}] + [\gamma_{\chi- \sw_1 - \sw_2}], \qquad \chi \in M'_\sigma.
$$
One again verifies that such cycles generate the kernel of $\partial$.

Now we establish the surjectivity statements for a general $\cX$. We consider $\mir^+_{\bT'}$ and the argument for $\mir^c_{\bT'}$ is similar. Consider the decomposition \eqref{eqn:equiv-curve-decomposition} of the mirror curve $\tC_q$. We apply the Mayer-Vietoris sequence to
$$
    (\tC_q, \Re (\hx) \gg 0) = \left(\bigcup_{\si\in \Si(3)} \tC_\sigma^\circ, \Re (\hx) \gg 0 \right) \cup \left(\bigcup_{\tau\in \Si(2)_c} \tC_\tau^\circ, \emptyset \right).
$$
Recall that each connected component of any $\tC_\tau^\circ$ or any $\tC_\sigma^\circ \cap \tC_\tau^\circ$ is contractible and has trivial $H_1$. The sequence reads
$$
    \displaystyle 0 \to \bigoplus_{\sigma \in \Sigma(3)} H_1(\tC_\sigma^\circ, \Re (\hx) \gg 0; \bZ) \to H_1(\tC_q, \Re (\hx) \gg 0; \bZ) \xrightarrow{\delta}
$$
$$
    \displaystyle \to H_0 \left( \left(\bigcup_{\si\in \Si(3)} \tC_\sigma^\circ \right) \cap \left( \bigcup_{\tau\in \Si(2)_c} \tC_\tau^\circ \right); \bZ \right) \xrightarrow{\iota} \bigoplus_{\sigma \in \Sigma(3)} H_0(\tC_\sigma^\circ, \Re (\hx) \gg 0; \bZ) \oplus  \bigoplus_{\tau \in \Sigma(2)_c} H_0(\tC_\tau^\circ; \bZ) \to \cdots .   
$$

For $\sigma \in \Sigma(3)$, the intersection of $\tC_\sigma^\circ$ with the subset $\{\Re (\hx) \gg 0\} \subset \tC_q$ may be empty or contained in $\tC_\tau^\circ$ for one or two facets $\tau$ of $\sigma$. Let $\sigma'$ be the intersection of all such $\tau$'s; if there is none, $\sigma' = \sigma$. Then the closed substack $\cV(\sigma')$ of $\cX$ is in fact contained in the affine chart $\cX_\sigma$. The discussion in the affine case above implies that the images of the sheaves $\cO_{\cV(\sigma')} \otimes \cL$ under $\mir^+_{\bT'}$ generate the subgroup $H_1(\tC_\sigma^\circ, \Re (\hx) \gg 0; \bZ)$. Therefore, it suffices to show that the image of $\mir^+_{\bT'}$ in $H_1(\tC_q, \Re (\hx) \gg 0; \bZ)$, when passing through the connecting homomorphism $\delta$, surjects onto the kernel of the map $\iota$ in the Mayer-Vietoris sequence.

We introduce the index sets
$$
    \tF := \{ \tilde{\bff} = (\bff, \psi): \bff = (\tau, \sigma) \in F(\Sigma), \tau \in \Sigma(2)_c, \psi \in M'_\tau\}, \quad 
    \Sigma(3)_s := \{ \sigma \in \Sigma(3): \tC_\sigma^\circ \cap \{\Re (\hx) \gg 0\} = \emptyset \}.
$$
Then the map $\iota$ can be written as
$$
    \iota: \bigoplus_{\tilde{\bff} \in \tF} \bZ e_{\tilde{\bff}} \to \bigoplus_{\sigma \in \Sigma(3)_s} \bZ e_\sigma \oplus \bigoplus_{\substack{\tau \in \Sigma(2)_c \\ \psi \in M'_\tau}} \bZ e_\psi, \qquad 
    e_{\tilde{\bff}} \mapsto e_\sigma - e_\psi.
$$
Here, for $\tilde{\bff} \in \tF$, $e_{\tilde{\bff}}$ denotes the class of a point in $\tC_\sigma^\circ \cap \tC_\tau^\circ$. For $\sigma \in \Sigma(3)_s$, $e_\sigma$ denotes the class of a point in $\tC_\sigma^\circ$. To unify notation, for $\sigma \not \in \Sigma(3)_s$ we set $e_\sigma = 0$ since the relative $H_0$ is trivial. In the notation of Sections \ref{sec:divisor-cycle}, \ref{sec:curve-cycle}, for the divisor $\cD_j$ with an $\bT'$-equivariant line bundle $\cL$ corresponding to a twisted polytope $\uchi$, we have
\begin{equation}\label{eqn:DivisorConnecting}
    \delta \circ \mir^+_{\bT'}(\cO_{\cD_j} \otimes \cL) = \sum_{i = 1}^m - e_{((\tau_{i+1}, \sigma_i),\psi_{i+1})} +  e_{((\tau_i, \sigma_i),\psi_i)} 
\end{equation}
where $\psi_{i+1} \in M'_{\tau_{i+1}}$ denotes common image of $\chi_i, \chi_{i+1}$ under \eqref{eqn:DivisorAdjProj}. For the curve $\fl_\tau$, $\tau \in \Sigma(2)_c$, with an $\bT'$-equivariant line bundle $\cL$ corresponding to a twisted polytope $\uchi$, we have
\begin{equation}\label{eqn:CurveConnecting}
    \delta \circ \mir^+_{\bT'}(\cO_{\fl_\tau} \otimes \cL) = - e_{((\tau, \sigma_1),\psi)} +  e_{((\tau, \sigma_1),\psi + 1)} - e_{((\tau, \sigma_2),\psi+1)} + e_{((\tau, \sigma_2),\psi)} 
\end{equation}
where $\psi \in M'_{\tau}$ denotes common image of $\chi_1$, $\chi_2$ under \eqref{eqn:CurveAdjProj} and the ``$+1$'' is chosen with respect to $(\tau, \sigma_1)$. We show that any element
$$
    \alpha = \sum_{\tilde{\bff} \in \tF} a_{\tilde{\bff}}e_{\tilde{\bff}} \in \ker(\iota)
$$
is a linear combination of elements of form \eqref{eqn:DivisorConnecting}, \eqref{eqn:CurveConnecting}. The coefficients $a_{\tilde{\bff}}$ in $\alpha$ satisfy the following conditions:
\begin{enumerate}[label=(\roman*)]
    \item For each $\sigma \in \Sigma(3)_s$, the sum of all $a_{\tilde{\bff}}$ with $\sigma$ appearing in $\tilde{\bff}$ is zero.
    
    \item For each $\tau \in \Sigma(2)_c$ and $\psi \in M'_\tau$, the sum of all $a_{\tilde{\bff}}$ with $\tau, \psi$ appearing in $\tilde{\bff}$ is zero.
\end{enumerate}
Consider the total order $\le$ on $\Sigma(3)$ determined by the values of $\fp_\sigma$ under the moment map $\mu_{\bT_{f, \bR}}$. For $\sigma < \sigma'$ with $\tau = \sigma \cap \sigma' \in \Sigma(2)_c$, we say that $\tau$ is an outgoing facet of $\sigma$ and incoming facet of $\sigma'$. Now consider a minimal $\sigma \in \Sigma(3)$ with respect to $\le$. By condition (i), we may modify $\alpha$ by elements of form \eqref{eqn:DivisorConnecting}, \eqref{eqn:CurveConnecting} such that each $a_{\tilde{\bff}}$ with $\sigma$ appearing in $\tilde{\bff}$ becomes zero. In doing so, for each outgoing facet $\tau \in \Sigma(2)_c$ of $\sigma$, each $a_{\tilde{\bff}}$ with $(\tau, \sigma)$ appearing in $\tilde{\bff}$ also becomes zero. Condition (ii) then implies that if $\sigma' \in \Sigma(3)$ is such that $\tau$ is incoming for $\sigma'$, then each $a_{\tilde{\bff}}$ with $(\tau, \sigma')$ appearing in $\tilde{\bff}$ becomes zero as well. We then proceed to the next 3-cone under $\le$. Noting that the coefficients from incoming facets have been nullified, we may repeat the above modification process. Eventually, after a traversal of all 3-cones, $\alpha$ is modified into zero.
\end{proof}


\section{All-genus descendant mirror symmetry}\label{sec:allgenus}
In this section, we prove all-genus mirror symmetry for equivariant descendant invariants of $\cX$. The higher-genus B-model of $\cX$ is provided by the Chekhov-Eynard-Orantin topological recursion on the mirror curve, which we first review.

\subsection{Differentials of the second kind}
The main reference of this subsection is \cite{Fay73}. Let $\Cbar_q$ be the compactification of the mirror curve $C_q$ as in Section \ref{sec:mirror-curve} and $\cap$ denote the intersection pairing $H_{1}(\Cbar_q;\bC)\times H_{1}(\Cbar_q;\bC)\to \bC$. We choose a symplectic basis $\{ A_i, B_i: i=1,\dots, \fg\}$ of $(H_{1}(\Cbar_q;\bC), \cap)$:
$$
A_i \cap A_j = B_i\cap B_j =0,\quad A_i\cap B_j = -B_j\cap A_i =\delta_{i, j},\quad i,j\in \{1,\dots,\fg\}.
$$


The \emph{fundamental differential of the second kind} on $\Cbar_q$ normalized by $A_1,\dots, A_{\fg}$ is a bilinear symmetric meromorphic differential $B(p_1,p_2)$ characterized by:
\begin{itemize}
\item $B(p_1,p_2)$ is holomorphic everywhere except for a double pole along the diagonal $p_1=p_2$, and,
if $z_1$, $z_2$ are local coordinates on $\Cbar_q\times \Cbar_q$ near $(p,p)$ then
$$
B(z_1,z_2)= \left(\frac{1}{(z_1-z_2)^2} + f(z_1,z_2) \right)dz_1 dz_2
$$
where $f(z_1, z_2)$ is holomorphic and $f(z_1,z_2)=f(z_2,z_1)$.
\item $\displaystyle{\int_{p_1\in A_i} B(p_1,p_2) =0}$, $i=1,\dots, \fg$.
\end{itemize}
It is also called the Bergman kernel in \cite{EO07,EO15}. We adopt the choice of A-cycles $A_1,\dots, A_\fg\in H_1(\Cbar;\bC)$ and B-cycles $B_1,\dots, B_\fg \in H_1(\Cbar_q;\bC)$ in \cite[Section 5.9]{flz2020remodeling}. We choose the fundamental differential of the second kind $B$ to be normalized by $A_1,\dots,A_\fg$.

Following \cite{Eynard11, EO15}, given any $\bsi \in I_\Si$, let
$
\zeta_{\bsi} =\sqrt{\hx-\check{u}^{\bsi}}
$
be a local holomorphic coordinate near the critical point $p_{\bsi} \in C_q$. 
Around $p_{\bsi}$ we have
$$
  \hx = \check{u}^{\bsi} +\zeta_\bsi^2, \qquad 
  \hy = \check{v}^{\bsi} +\sum_{d=1}^\infty h^{\bsi}_d \zeta_\bsi^d
$$
where
$$
  h_1^{\bsi} = \sqrt{ \frac{2}{\frac{d^2\hat{x}}{d\hat{y}^2}(p_{\bsi}) } }. 
$$

For any non-negative integer $d$,  define
$$
\theta_{\bsi}^d(p):= -(2d-1)!! 2^{-d}\Res_{p'\to p_{\bsi}}
B(p,p')\zeta_\bsi^{-2d-1}.
$$
Then $\theta_{\bsi}^d$ satisfies the following properties:
\begin{itemize}
  \item $\theta_{\bsi}^d$ is a meromorphic 1-form on $\Cbar_q$ with
  a single pole of order $2d+2$ at $p_{\bsi}$.

  \item In local coordinate $\zeta_{\bsi} =\sqrt{ \hat{x}-\check{u}^{\bsi}}$ near $p_{\bsi}$, we have
  $$
  \theta_{\bsi}^d = \left( \frac{-(2d+1)!!}{2^d \zeta_\bsi^{2d+2}}
  + f(\zeta_\bsi)\right) d\zeta_\bsi
  $$
  where $f(\zeta_\bsi)$ is analytic around $p_{\bsi}$.
  The residue of $\theta_{\bsi}$ at $p_{\bsi}$ is zero,
  so $\theta_\bsi$ is a differential of the second kind.
  
  \item $\int_{A_i} \theta_\bsi^d=0$, $i=1,\dots,\fg$.
\end{itemize}
The meromorphic 1-form $\theta_{\bsi}^d$ is uniquely characterized by the above
properties; $\theta_{\bsi}^d$ can be viewed as a section in
$H^0(\Cbar_q, \omega_{\Cbar_q}((2d+2) p_{\bsi}))$.

Following Eynard-Orantin \cite{EO15}, let
$$
f_{\bsi'}^{\spa \bsi}(u) :=
\frac{e^{u\check{u}^{\bsi}}}{2\sqrt{\pi u}}
\int_{\gamma_{\bsi}} e^{-u\hat{x}}\theta^0_{\bsi'} \quad \in \bC \formal{ u^{-1} }.
$$

\subsection{Topological recursion and the B-model graph sum}
\label{sec:eynard-orantin}

Let $\omega_{g,n}$ be defined recursively by the Chekhov-Eynard-Orantin topological recursion \cite{EO07}:
$$
\omega_{0,1}= \Phi = \hy d\hx,\qquad  \omega_{0,2}=B(p_1,p_2),
$$
and when $2g-2+n>0$,
\begin{align*}
\omega_{g,n}(p_1,\dots, p_n) = & \sum_{\bsi\in I_\Si}\Res_{p \to p_\bsi}
\frac{\int_{\xi = p}^{\bar{p}} B(p_n,\xi)}{2(\Phi(p)-\Phi(\bar{p}))}
\Bigl( \omega_{g-1,n+1}(p,\bar{p},p_1,\dots, p_{n-1}) \\
& + \sum_{g_1+g_2=g}
\sum_{ \substack{ I\cup J=\{1,..., n-1\} \\ I\cap J =\emptyset } }' \omega_{g_1,|I|+1} (p,p_I)\omega_{g_2,|J|+1}(\bar{p},p_J) \Bigr)
\end{align*}
where the prime in the notation $\sum_{g_1, g_2} \sum_{I, J}'$ means that any term with $(g_1 = 0, I = \emptyset)$ or $(g_2 = 0, J = \emptyset)$ is excluded.

We now state the graph sum formula of \cite{DOSS14} for the B-model invariants $\omega_{g,n}$. Recall that labeled graphs and the A-model graph sum are discussed in Section \ref{sec:Agraph}. We introduce some notation.
\begin{itemize}
\item  For any $\bsi\in I_{\Sigma}$, we define
$$
  \check{h}^{\bsi}_{k} :=\frac{(2k-1)!!}{2^{k-1}}h^\bsi_{2k-1}. 
$$
Then
$$
\check{h}^{\bsi}_k = [u^{1-k}]\frac{u^{3/2}}{\sqrt{\pi}} e^{u\check{u}^{\bsi}}
\int_{p\in \gamma_{\bsi}}e^{-u \hat{x}(p)}\Phi(p).
$$

\item For any $\bsi,\bsi'\in I_\Sigma$, we expand
$$
B(p_1,p_2) =\left( \frac{\delta_{\bsi,\bsi'}}{ (\zeta_\bsi-\zeta_{\bsi'})^2}
+ \sum_{k,l\in \bZ_{\geq 0}} B^{\bsi,\bsi'}_{k,l} \zeta_\bsi^k \zeta_{\bsi'}^l \right) d\zeta_\bsi d\zeta_{\bsi'}
$$
near $p_1=p_{\bsi}$ and $p_2=p_{\bsi'}$, and define
$$
  \check{B}^{\bsi,\bsi'}_{k,l} := \frac{(2k-1)!! (2l-1)!!}{2^{k+l+1}} B^{\bsi,\bsi'}_{2k,2l}.
$$
Then
\[
\check{B}^{i,j}_{k,l}=[u^{-k}v^{-l}]\left(\frac{uv}{u+v}(\delta_{\bsi,\bsi'}
-\sum_{\bgamma\in I_\Si} f^{\ \bsi}_{\bgamma}(u)f^{\ \bsi'}_{\bgamma}(v))\right)
=[z^{k}w^{l}]\left(\frac{1}{z+w}(\delta_{\bsi,\bsi'}
-\sum_{\bgamma\in I_\Si} f^{\ \bsi}_{\bgamma}(\frac{1}{z})f^{\ \bsi'}_{\bgamma}(\frac{1}{w}))\right).
\]

\end{itemize}

Given a labeled graph $\vGa \in \bGa_{g,n}(\cX)$ with
$L^o(\Ga)=\{l_1,\dots,l_n\}$, 
we define its \emph{B-model weight} to be
\begin{align*}
w_B^{\bu}(\vGa) =& (-1)^{g(\vGa)-1}\prod_{v\in V(\Gamma)} \left(\frac{h^{\bsi(v)}_1}{\sqrt{-2}}\right)^{2-2g-\val(v)} \left\langle \prod_{h\in H(v)} \tau_{k(h)} \right\rangle_{g(v)}
\prod_{e\in E(\Gamma)} \check{B}^{\bsi(v_1(e)),\bsi(v_2(e))}_{k(e),l(e)}  \\
& \cdot \prod_{l\in \cL^1(\Gamma)}(\check{\cL}^1)^{\bsi(l)}_{k(l)}
\prod_{i=1}^n (\check{\cL}^\bu)^{\bsi(l_i)}_{k(l_i)}(l_i)
\end{align*}
where
\begin{itemize}
\item (dilaton leaf)
$$
(\check{\cL}^1)^{\bsi}_k = \frac{-1}{\sqrt{-2}}\check{h}^{\bsi}_k;
$$
\item (descendant leaf)
$$
(\check{\cL}^\bu)^{\bsi}_k(l_i) =  \frac{1}{\sqrt{-2}} \theta_{\bsi}^k(p_i).
$$
\end{itemize}

In our notation \cite[Theorem 3.7]{DOSS14} is equivalent to:
\begin{theorem}[Dunin-Barkowski--Orantin--Shadrin--Spitz \cite{DOSS14}] \label{thm:DOSS}
For $2g-2+n>0$, we have
$$
\omega_{g,n} = \sum_{\vGa \in \bGa_{g,n}(\cX)}\frac{w_B^{\bu}(\vGa)}{|\Aut(\vGa)|}.
$$
\end{theorem}

\subsection{Identification of graph sums}
We now state the identification of the A- and B-model graph sums obtained in \cite{flz2020remodeling}. For $\bsi \in I_\Si$ and $k \in \bZ_{\ge 1}$, define
$$
  \hxi_{\bsi}^k := (-1)^k \left(\frac{d}{d\hx} \right)^{k-1} \frac{\theta_{\bsi}^0}{d\hx}
$$
which is a meromorphic function on $\Cbar_q$. Moreover, define
$$
  \htheta_{\bsi}^0 := \theta_{\bsi}^0, \qquad \qquad \htheta_{\bsi}^k := d \hxi_{\bsi}^k, \quad k \in \bZ_{\ge 1}.
$$
As shown in \cite[Proposition 6.6]{flz2020remodeling}, the differentials $\{\htheta_{\bsi}^k\}_{k \in \bZ_{\ge 0}}$ are related to $\{\theta_{\bsi}^k\}_{k \in \bZ_{\ge 0}}$, which define the B-model descendant leaf, via the \emph{B-model $R$-matrix}.

Moreover, for $i=1,\dots,n$ and $\bsi\in I_\Si$, introduce formal variables
$$
\tilde{\bu}_i^\bsi(z)=\sum_{a\geq 0}(\tilde{u}_i)^\bsi_az^a:=\sum_{\bsi'\in I_\Si}
\left(\frac{\bu_i^{\bsi'}(z)}{\sqrt{\Delta^{\bsi'}(\btau)} }
S^{\widehat{\underline{\bsi}} }_{\spa \widehat{\underline{\bsi'}}}(z)\right)_+.
$$
Define the flat coordinates $\overline \bu_i^\bsi$ by
$$
\sum_{\bsi\in I_\Si} \bu^\bsi_i(z) \phi_\bsi(q)= \sum_{\bsi\in I_\Si} \overline \bu_i^\bsi(z) \phi_\bsi(0).
$$
Recall that $(S^{\hubsi}_{\ \brho})$ is unitary, i.e. $\sum_{\brho} S^{\hubrho}_{\ \bsi}(z) S^{\hubrho}_{\ \bsi'}(-z)=\frac{1}{\Delta^{\bsi}}\delta_{\bsi\bsi'}$. Therefore we have
\begin{equation}\label{eqn:tu}
	\sum_{\bsi\in I_\Si} \bigl( S^{\hubsi}_{\ \brho}(-z) \tilde \bu^{\bsi}_i(z)\bigr)_+=\sum_{\bsi\in I_\Si} \left( \sum_{\bsi'\in I_\Si} S^{\hubsi}_{\ \bsi'}(z) S^{\hubsi}_{\ \brho}(-z) \overline \bu^{\bsi'}_i(z) \right)=\frac{\overline \bu^{\brho}_i(z)}{\Delta^{\brho}}.
\end{equation}

With the above notation, \cite[Theorem 7.2]{flz2020remodeling} provided the following identification of weights in the graph sums (Theorems \ref{thm:Zong}, \ref{thm:DOSS}).

\begin{theorem}\label{thm:graph-match}
Let $2g-2+n>0$. For any $\vGa\in \bGa_{g,n}(\cX)$, we have
$$
w^\bu_B(\vGa)|_{ \frac{1}{\sqrt{-2}}\htheta_{\bsi}^a(p_i)=-(\tilde{u}_i)^\bsi_a}
=(-1)^{g(\vGa)-1+n} w^\bu_A(\vGa)
$$
under the mirror map $\btau = \btau(q)$.
\end{theorem}

\subsection{Unstable cases}
In preparation for proving all-genus mirror symmetry, we use the genus-zero, 1-pointed descendant mirror theorem (Theorem \ref{thm:EquivMir}) to deduce the following lemma on descendant leaves.

\begin{lemma}\label{lem:LeafInt}
For $\bsi \in I_\Si$ and $k \in \bZ_{\ge 0}$, and any $\cE \in K^+_{\bT'}(\cX)$, we have
$$  
  2\pi\sqrt{-1} \int_{p\in \mir^+_{\bT'}(\cE)} e^{- \hx(p)/z} \frac{(-\htheta_{\bsi}^k(p))}{\sqrt{-2}} = (-1)^{k+1} z^{-k-2} \left\llangle \hat{\phi}_{\bsi}(\btau(q)), \frac{\kappa_z^{\bT'}(\cE)}{z+\hpsi} \right\rrangle^{\cX, \bT'}_{0,2}  = (-1)^k z^{-k-2} S^{\ \kappa_z^{\bT'}(\cE)}_{\hubsi}(-z)
$$
under the mirror map $\btau = \btau(q)$.
\end{lemma}

\begin{proof}
We prove the first equality, starting with $k = 0$. Theorem \ref{thm:EquivMir} implies that
\begin{equation}\label{eqn:LeafBaseEqn}
  2\pi\sqrt{-1} \int_{p\in \mir^+_{\bT'}(\cE)} e^{- \hx(p)/z} \Phi = -\left\llangle \one, \frac{\kappa_z^{\bT'}(\cE)}{z+\hpsi} \right\rrangle^{\cX, \bT'}_{0,2}. 
\end{equation}
We proceed as in the derivation of \cite[Equation (7.4)]{flz2020remodeling}. By the isomorphism \eqref{eqn:FrobIso} and \cite[Proposition 6.3]{flz2020remodeling}, with the $\bST$-basis
$$
  \text{$\one$}, \quad H_1, \dots, H_{\fp}, \quad H_{a_1} \star_{\btau} H_{b_1}, \dots, H_{a_{\fg}} \text{$\star_{\btau}$} H_{b_{\fg}}
$$
of $QH^*_{\bT'}(\cX)$ where $a_1, b_1, \dots, a_{\fg}, b_{\fg} \in \{1, \dots, \fp\}$, if
$$
  \phi_{\bsi}(\btau(q)) = \sum_{i = 1}^{\text{$\fg$}} A_{\bsi}^i(q) H_{a_i} \star_{\btau} H_{b_i} - \sum_{a = 1}^{\text{$\fp$}} B_{\bsi}^a(q) H_a + C_{\bsi}(q) \one,
$$
then
$$
  \frac{h_1^{\bsi}\htheta_{\bsi}^0}{2} = \sum_{i = 1}^{\text{$\fg$}} A_{\bsi}^i(q) \frac{\partial^2 \Phi}{\partial\tau_{a_i}\partial\tau_{b_i}} + \sum_{a = 1}^{\fp} B_{\bsi}^a(q) d\left(\frac{\frac{\partial\Phi}{\partial\tau_a}}{d\hx}\right) + C_{\bsi}(q) d\left(\frac{d\hy}{d\hx}\right).
$$
Moreover, it is computed in \cite[Section 7.1]{flz2020remodeling} that
$$
  h_1^{\bsi}(q)= \sqrt{ \frac{-2}{\Delta^\bsi(\btau)} } \bigg|_{\btau = \btau(q)}.
$$
Thus multiplying the above by $\sqrt{\Delta^\bsi(\btau)} \big|_{\btau = \btau(q)}$ provides a similar relation between $\hat{\phi}_{\bsi}(\btau(q))$ and $-\frac{\htheta_{\bsi}^0}{\sqrt{-2}}$. Now applying the differential operator
$$
  \sqrt{\Delta^\bsi(\btau)} \big|_{\btau = \btau(q)}  \left( z^2 \sum_{i = 1}^{\text{$\fg$}} A_{\bsi}^i(q) \frac{\partial^2 }{\partial\tau_{a_i}\partial\tau_{b_i}} + z \sum_{a = 1}^{\fp} B_{\bsi}^a(q) \frac{\partial}{\partial\tau_a} + C_{\bsi}(q) \right)
$$
to the two sides of \eqref{eqn:LeafBaseEqn} yields the $k=0$ case of the lemma; here the left hand side follows from integration by parts. The case $k \ge 1$ then follows by inductively applying integration by parts.
\end{proof}

We may use Lemma \ref{lem:LeafInt} to deduce the following analog of Theorem \ref{thm:EquivMir} for the genus-zero, 2-pointed case.

\begin{proposition}\label{prop:02Case}
For any $\cE_1, \cE_2 \in K^+_{\bT'}(\cX)$, we have
$$
    (2\pi\sqrt{-1})^2 \int_{p_1\in\mir^+_{\bT'}(\cE_1)} \int_{p_2\in\mir^+_{\bT'}(\cE_2)} e^{- (\hx(p_1)/z_1+\hx(p_2)/z_2)} \omega_{0,2}(p_1,p_2) = - \left\llangle \frac{\kappa_{z_1}^{\bT'}(\cE_1)}{z_1(z_1 + \hpsi)}, \frac{\kappa_{z_2}^{\bT'}(\cE_2)}{z_2(z_2 + \hpsi)} \right\rrangle^{\cX, \bT'}_{0,2}
$$
under the mirror map $\btau = \btau(q)$.
\end{proposition}

\begin{proof}
We compute
\begin{align*}
  & (2\pi\sqrt{-1})^2 \int_{p_1\in\mir^+_{\bT'}(\cE_1)} \int_{p_2\in\mir^+_{\bT'}(\cE_2)} e^{- (\hx(p_1)/z_1+\hx(p_2)/z_2)} \omega_{0,2}(p_1,p_2) \\
  & = - \frac{z_1z_2}{2(z_1+z_2)} \sum_{\bsi \in I_\Si} (2\pi\sqrt{-1})^2 \int_{p_1\in\mir^+_{\bT'}(\cE_1)} \int_{p_2\in\mir^+_{\bT'}(\cE_2)} e^{- (\hx(p_1)/z_1+\hx(p_2)/z_2)} \htheta_{\bsi}^0(p_1)\htheta_{\bsi}^0(p_2) \\
  & = \frac{z_1^{-1}z_2^{-1}}{z_1+z_2} \sum_{\bsi \in I_\Si} \left\llangle \hat{\phi}_{\bsi}(\btau(q)), \frac{\kappa_{z_1}^{\bT'}(\cE_1)}{z_1+\hpsi} \right\rrangle^{\cX, \bT'}_{0,2} \left\llangle \hat{\phi}_{\bsi}(\btau(q)), \frac{\kappa_{z_2}^{\bT'}(\cE_2)}{z_2+\hpsi} \right\rrangle^{\cX, \bT'}_{0,2}\\
  & = - \left\llangle \frac{\kappa_{z_1}^{\bT'}(\cE_1)}{z_1(z_1 + \hpsi)}, \frac{\kappa_{z_2}^{\bT'}(\cE_2)}{z_2(z_2 + \hpsi)} \right\rrangle^{\cX, \bT'}_{0,2}.
\end{align*}
Here, the first equality follows from the treatment of oscillatory integrals of $\omega_{0,2}$ in the proof of \cite[Lemma 6.9]{flz2020remodeling}; the second equality follows from Lemma \ref{lem:LeafInt}; the third equality follows from identity \eqref{eqn:SIdentity}.
\end{proof}

\subsection{All-genus descendant mirror theorem}
Finally, we obtain all-genus descendant invariants of $\cX$ via oscillatory integrals of the topological recursion invariants $\omega_{g,n}$ along the mirror cycles provided by the isomorphism
$$
    \mir^+_{\bT'}: K^+_{\bT'}(\cX) \to H_1(\tC_q, \Re (\hx) \gg 0; \bZ).
$$

\begin{definition}\label{def:A-model-Zgn} 
Given $\cE_1, \dots, \cE_n \in K^+_{\bT'}(\cX)$, we define
$$
    Z^+_{g,n}(\cE_1, \dots, \cE_n) := \double{\frac{\kappa_{z_1}^{\bT'}(\cE_1)}{z_1(z_1 + \hpsi)},\dots,\frac{\kappa_{z_n}^{\bT'}(\cE_n)}{z_n(z_n + \hpsi)} }^{\cX, \bT'}_{g,n}.
$$  
\end{definition}

In particular, $Z^+_{0,1}(\cE_1) = Z_{\bT'}(\cE_1)$. 

\begin{definition} \label{def:B-model-Zgn} 
Given $\gamma_1, \dots, \gamma_n \in H_1(\tC_q, \Re (\hx) \gg 0; \bZ)$, we define
$$
    \chZ^+_{g,n}(\gamma_1, \dots, \gamma_n) := (2\pi\sqrt{-1})^n \int_{p_1\in\gamma_1}\cdots\int_{p_n\in\gamma_n } e^{- (\hx(p_1)/z_1+\cdots+\hx(p_n)/z_n)} \omega_{g,n}(p_1,\dots,p_n).
$$
\end{definition}

\begin{theorem}\label{thm:All-genus-mirror}
For any $g \in \bZ_{\ge 0}$, $n \in \bZ_{\ge 1}$, and $\cE_1, \dots, \cE_n \in K^+_{\bT'}(\cX)$, we have
$$
  \chZ^+_{g,n}\left(\mir^+_{\bT'}(\cE_1), \dots, \mir^+_{\bT'}(\cE_n) \right) = (-1)^{g-1+n} Z^+_{g,n}(\cE_1, \dots, \cE_n)
$$
under the mirror map $\btau = \btau(q)$.
\end{theorem}

\begin{proof}
The unstable cases $(g,n) = (0,1)$, $(0,2)$ are Theorem \ref{thm:EquivMir} and Proposition \ref{prop:02Case} respectively. For the stable case $2g-2+n>0$, by Theorem \ref{thm:graph-match}, we have
\begin{align*}
& (2\pi\sqrt{-1})^n \int_{p_1\in\mir^+_{\bT'}(\cE_1)}\cdots\int_{p_n\in\mir^+_{\bT'}(\cE_n)} e^{- (\hx(p_1)/z_1+\cdots+\hx(p_n)/z_n)} \omega_{g,n}(p_1,\dots,p_n)\\
&= (2\pi\sqrt{-1})^n \int_{p_1\in\mir^+_{\bT'}(\cE_1)}\cdots\int_{p_n\in\mir^+_{\bT'}(\cE_n)} e^{- (\hx(p_1)/z_1+\cdots+\hx(p_n)/z_n)} (-1)^{g-1+n} \sum_{\bsi_i,a_i} \left\llangle
\prod_{i=1}^n\tau_{a_i}(\phi_{\bsi_i}(0)) \right\rrangle^{\cX, \bT'}_{g,n} \\
& \quad\quad \cdot \prod_{i=1}^n (\overline u_i)^{\bsi_i}_{a_i} \bigg|_{{(\tilde u_i)}^\bsi_k=-\frac{1}{\sqrt{-2}}\htheta_{\bsi}^k(p_i)}\\
&= (2\pi\sqrt{-1})^n \int_{p_1\in\mir^+_{\bT'}(\cE_1)}\cdots\int_{p_n\in\mir^+_{\bT'}(\cE_n)} e^{- (\hx(p_1)/z_1+\cdots+\hx(p_n)/z_n)} (-1)^{g-1+n} \sum_{\bsi_i,a_i} \left\llangle
\prod_{i=1}^n\tau_{a_i}(\phi_{\bsi_i}(0)) \right\rrangle^{\cX, \bT'}_{g,n} \\
&\quad\quad \cdot \prod_{i=1}^n \Delta^{\bsi_i}\sum_{\brho\in I_\Si} \sum_{k\in \bZ_{\ge 0}}(-1)^{a_i-k} \left([z_i^{a_i-k}] S_{\ \bsi_i}^{\hubrho}(z_i)\right) \frac{(-\htheta_{\brho}^k(p_i))}{\sqrt{-2}}\\
&= (-1)^{g-1+n} \sum_{\bsi_i,a_i} \left\llangle
\prod_{i=1}^n\tau_{a_i}(\phi_{\bsi_i}(0)) \right\rrangle^{\cX, \bT'}_{g,n}\prod_{i=1}^n \Delta^{\bsi_i}\sum_{\brho\in I_\Si} \sum_{k\in \bZ_{\geq 0}} (-1)^{a_i} \left([z_i^{a_i-k}]  S_{\ \bsi_i}^{\hubrho}(z_i)\right) z_i^{-k-2} S^{\ \kappa_{z_i}^{\bT'}(\cE_i)}_{\hubrho}(-z_i)  \\
&= (-1)^{g-1+n} \sum_{\bsi_i,a_i} \left\llangle
\prod_{i=1}^n\tau_{a_i}(\phi_{\bsi_i}(0)) \right\rrangle^{\cX, \bT'}_{g,n}  \prod_{i=1}^n  \Delta^{\bsi_i} (-1)^{a_i}z_i^{-a_i-2} \left(\phi_{\bsi_i}(0), \kappa_{z_i}^{\bT'}(\cE_i) \right)_{\cX, \bT'} \\
&= (-1)^{g-1+n}  \left\llangle \frac{\kappa_{z_1}^{\bT'}(\cE_1)}{z_1(z_1 + \hpsi)},\dots,\frac{\kappa_{z_n}^{\bT'}(\cE_n)}{z_n(z_n + \hpsi)} \right\rrangle^{\cX, \bT'}_{g,n}.
\end{align*}
Here, the second equality follows from \eqref{eqn:tu}; the third equality follows from Lemma \ref{lem:LeafInt}; the fourth equality follows from the unitary condition of the $\cS$-operator.
\end{proof}

\appendix


\section{\texorpdfstring{$K$}{K}-theory with bounded below/above support}\label{appdx:Bounded}
In this section, we introduce $K$-groups with bounded below/above support for toric orbifolds. Our approach follows and generalizes the results of Borisov-Horja \cite{BH06, BH15} on the usual $K$-groups and $K$-groups with compact support. This section supplements the discussion in Section \ref{sect:Sheaves} on the case of toric Calabi-Yau 3-orbifolds.


\subsection{Geometric setup}
We introduce some notation to be used within this section. Let $N$ be a lattice of rank $r \in \bZ_{\ge 1}$ and $M$ be the dual lattice. Let $C$ be an $r$-dimensional finite rational convex polyhedral cone in $N_{\bR} := N \otimes \bR$, which is not required to be strongly convex and can be the entire $N_{\bR}$. Let $\Sigma$ be a simplicial fan with support $|\Sigma| = C$, and let
$X$ be the $r$-dimensional simplicial toric variety associated to $\Sigma$. Then $X$ admits the action of
an $r$-dimensional algebraic torus $\bT$, whose character lattice is identified with $M$. 
 Let $\cX$ be the induced $r$-dimensional toric orbifold which is the canonical smooth toric Deligne-Mumford stack associated to $\Sigma$ \cite{FMN10}. We will use the same notational conventions for toric geometry as for the case of toric Calabi-Yau 3-orbifolds in the rest of the paper, unless otherwise specified.

Let 
$$
    \pi: \cX \to X_0
$$ 
be the composition of the $\bT$-equivariant map $\cX\to X$ from $\cX$ to its coarse moduli space $X$ and the $\bT$-equivariant 
map $X\to X_0$ from $X$ to its affinization
$$
X_0 := \Spec\left(H^0(X, \cO_X)\right) =  \Spec(\bC[C^\vee \cap M])
$$
where $C^\vee \subseteq M_{\bR} := M \otimes \bR$ is the dual cone of $C$. We call the unique $\bT$-fixed point
in $X_0$ the {\em origin} of $X_0$. The affinization map $X\to X_0$ is projective and $X^{\bT}$ is non-empty, so $X$ is a semi-projective toric variety in the sense of \cite{HS02}. In particular, if  $C=N_\bR$ then $C^\vee=0$, so $X_0=\Spec(\bC)$ is a point, $X$ is a projective toric variety, and $\cX$ is a proper toric orbifold.

The connected components of the inertia stack of $\cX$ are indexed by $\Box(\cX)$. For $v \in \Box(\cX)$, let $\sigma(v) \in \Sigma$ be the minimal cone such that $v \in \Box(\sigma)$, and let 
$$
    N_v := N / N_{\sigma(v)} = N \bigg/ \sum_{i \in I'_{\sigma(v)}} \bZ b_i
$$
which is a finitely generated abelian group possibly with torsion. Then the twisted sector $\cX_v$ indexed by $v$ is the toric Deligne-Mumford stack corresponding to the stacky fan $\Sigma_v$ which is the image of the star of $\sigma(v)$ in $\Sigma$ under the projection $N \to N_v$.

\subsection{Derived categories and $K$-groups}\label{apx-sect:KGroup}
Let $D(\cX) := D^b(\Coh(\cX))$ denote the bounded derived category of coherent sheaves on $\cX$ and $K(\cX) := K_0(D(\cX))$ denote its Grothendieck group. Moreover, let $\Sigma^c$ denote the subset of cones in $\Sigma$ whose interior is contained in the interior of $|\Sigma| = C$. 
Then
$$
    \cX^c := \bigcup_{\sigma \in \Sigma^c} \cV(\sigma),
$$
which is the preimage of the origin under $\pi: \cX \to X_0$, is a proper closed substack of $\cX$. We call 
$\cX^c$ the {\em core} of $\cX$. (The core of the semi-projective variety $X$ is the preimage of the origin under
the affinization map $X\to X_0$; see Definition 3.1 and Theorem 3.2 of \cite{HS02}.)
Let $D^c(\cX)$ denote the full subcategory of $D(\cX)$ consisting of complexes of sheaves whose cohomology sheaves are supported on $\cX^c$, and
$$
    K^c(\cX) := K_0(D^c(\cX))
$$
denote its Grothendieck group. The group $K^c(\cX)$ admits an action of the ring $K(\cX)$ by tensor product. The inclusion of categories gives a natural $K(\cX)$-linear map $K^c(\cX) \to K(\cX)$. 

Borisov-Horja \cite{BH06, BH15} provided combinatorial presentations of these $K$-groups. First, $K(\cX)$ is presented as the quotient of $\bZ[R_1^{\pm 1}, \dots, R_{r+\fp'}^{\pm 1}]$ by the relations
\begin{itemize}
    \item $\prod_{i \in I} (1-R_i)$, for $I \not \in \Sigma$;
    \item $\prod_{i=1}^{r+\fp'} R_i^{\inner{u, b_i}} - 1$, for $u \in M$.
\end{itemize}
Here and below, by a slight abuse of notation, by $I \in \Sigma$ we mean that $I = I'_\sigma$ for some $\sigma \in \Sigma$. Moreover, $K^c(\cX)$ is presented as the $K(\cX)$-module with generators
$$
    G_I, \qquad \text{for $I = I'_\sigma$ with $\sigma \in \Sigma^c$}
$$
and relations
\begin{itemize}
    \item $(1-R_i^{-1}) G_I - G_{I \cup \{i\}}$, for $i \not \in I$, $I \sqcup \{i\} \in \Sigma$;
    \item $(1-R_i^{-1}) G_I$, for $i \not \in I$, $I \sqcup \{i\} \not \in \Sigma$.
\end{itemize}
The natural map $K^c(\cX) \to K(\cX)$ is given by $G_I \mapsto \prod_{i \in I} (1-R_i^{-1})$.

There is an Euler characteristic pairing between $D(\cX)$ and $D^c(\cX)$ given by
$$    
    \inner{\cE, \cE^c} := \sum_k (-1)^k \dim_{\bC} \Ext^k(\cE, \cE^c)
$$
which descends to a non-degenerate pairing
$$
    \chi: K(\cX) \times K^c(\cX) \to \bZ.
$$
Combinatorially, \cite[Section 4]{BH15} gave an explicit description of the Euler characteristic map
\begin{equation}\label{apx-eqn:EulerMap}
    \chi: K^c(\cX) \to \bZ
\end{equation}
whose composition of $\chi$ with the multiplication $K(\cX) \times K^c(\cX) \to K^c(\cX)$,
$$
    \left( \prod_{i=1}^{r+\fp'} R_i^{\alpha_i}, \prod_{i=1}^{r+\fp'} R_i^{\beta_i}G_I \right) \mapsto \prod_{i=1}^{r+\fp'} R_i^{\beta_i - \alpha_i}G_I
$$
gives the Euler characteristic pairing \cite[Theorem 5.2, Lemma 5.3]{BH15}. Note that the input from the $K(\cX)$-factor is dualized.

We will also consider the complexified versions $K := K(\cX) \otimes \bC, K^c := K^c(\cX) \otimes \bC$.

\subsection{Stanley-Reisner cohomology and combinatorial Chern character}\label{apx-sect:CombChern}
The Chen-Ruan cohomology of a semi-projective smooth toric Deligne-Mumford stack coincides with the Stanley-Reisner cohomology \cite{BCS05,JT08}. By \cite{BH06}, as a $\bC$-vector space (and module over the untwisted sector), the Stanley-Reisner cohomology of $\cX$ can be described as
$$
    H = \bigoplus_{v \in \Box(\cX)} H_v
$$
where the factor on the untwisted sector is the quotient of $\bC[D_1, \dots, D_{r+\fp'}]$ by the relations
\begin{itemize}
    \item $\prod_{i \in I} D_i$, for $I \not \in \Sigma$;
    \item $\sum_{i=1}^{r+\fp'} \inner{u, b_i}D_i$, for $u \in M$.
\end{itemize}
The definition of the factor $H_v$ on the twisted sector is similar and based on the closed substack $\cX_v$. Let $Z_v$ denote the ideal generated by the second type of relations above.

Borisov-Horja \cite{BH06} defined a combinatorial Chern character map
$$
    \tch: K \xrightarrow{\sim} H
$$
which is a $\bC$-vector space isomorphism. The projection $\tch_v: K \to H_v$ to the twisted sector $H_v$ is defined by
$$
    \tch_v(R_i) = \begin{cases}
        1 & \text{if } \{i\} \cup I'_{\sigma(v)} \not \in \Sigma,\\
        \exp D_i & \text{if } i \not \in I'_{\sigma(v)}, \{i\} \cup I'_{\sigma(v)} \in \Sigma,\\
        \exp \left(2\pi \sqrt{-1} c_i(v) \right) \prod_{j \not \in \sigma(v)} \tch_v(R_j)^{\inner{m_i, b_j}} & \text{if } i \in I'_{\sigma(v)},   
    \end{cases}
$$
where in the last case $m_i$ is any element of $M_{\bQ}$ such that $\inner{m_i, b_i} = -1$ and $\inner{m_i, b_j} = 0$ for all other $j \in \sigma(v)$. Note that the twisted Chern character $\tch_z$ defined in Section \ref{sect:Sheaves} is a version modified by factors of $-\frac{2\pi\sqrt{-1}}{z}$. As pointed out in \cite[Remark 5.4]{BH06}, although $\tch$ generalizes the usual Chern character for smooth toric varieties, it is not a ring isomorphism if $H$ is endowed the ring structure given by Stanley-Reisner cohomology. As in \cite{BH15}, from now on we equip $H$ with the ring structure induced by that on $K$ via $\tch$.

In the compact support case, analogous results were obtained in \cite{BH15}. Let
$$
    H^c = \bigoplus_{v \in \Box(\cX)} H^c_v
$$
where on the untwisted sector, $H^c_v$ is the $H_v$-module presented as the quotient of the free $\bC[D_1, \dots, D_{r+\fp'}]$-module with generators
$$
    F_I, \qquad \text{for $I = I'_\sigma$ with $\sigma \in \Sigma^c$}
$$
and relations
\begin{itemize}
    \item $D_i F_I - F_{I \cup \{i\}}$, for $i \not \in I$, $I \sqcup \{i\} \in \Sigma$;
    \item $D_i F_I$, for $i \not \in I$, $I \sqcup \{i\} \not \in \Sigma$;
    \item elements given by the action of the ideal $Z_v$.
\end{itemize}
The definition of the factors on the twisted sectors is similar. The natural map $H^c_v \to H_v$ is given by $F_I \mapsto \prod_{i \in I} D_i$.

As shown in \cite[Section 2]{BH15}, for each $v$, $H^c_v$ is the dualizing module of $H_v$ and has 1-dimensional socle. Evaluation at the socle is the unique linear function $\int: H^c_v \to \bC$ which, in the case of the untwisted sector, takes value on $\frac{1}{|G_\sigma|}$ on $F_I$ for any $I = I'_\sigma$, $\sigma \in \Sigma(r)$. The composition of the multiplication $H_v \times H^c_v \to H^c_v$ and $\int$ gives a non-degenerate pairing $H_v \times H^c_v \to \bC$. Summing over all twisted sectors, we have an evaluation 
\begin{equation}\label{apx-eqn:SocleMap}
    \int: H^c \to \bC
\end{equation}
and a multiplication $H \times H^c \to H^c$ whose composition gives a non-degenerate pairing
$$
    (-,-): H \times H^c \to \bC.
$$

Borisov-Horja \cite[Section 3]{BH15} defined a combinatorial Chern character with compact support
$$
    \tch^c: K^c \xrightarrow{\sim} H^c
$$
which is a $\bC$-vector space isomorphism. The projection $\tch^c_v: K^c \to H^c_v$ to the twisted sector $H^c_v$ is defined by
$$
    \tch^c_v \left(\prod_{i=1}^{r+\fp'}R_i^{\alpha_i} G_I \right) = \begin{cases}
        0 & \text{if } I \cup I'_{\sigma(v)} \not \in \Sigma,\\
        \displaystyle \prod_{i=1}^{r+\fp'} \tch_v(R_i)^{\alpha_i} \prod_{i \in I \setminus I'_{\sigma(v)}} \frac{1 - \exp (-D_i)}{D_i} \prod_{i \in I \cap I'_{\sigma(v)}} (1 - \tch_v(R_i)^{-1}) F_{\Ibar} & \text{if } I \cup I'_{\sigma(v)} \in \Sigma,   
    \end{cases}
$$
where in the latter case $\Ibar = I \setminus I'_{\sigma(v)}$ is the set of indices of the cone in the induced fan $\Sigma_v$ corresponding to $I$. With respect to the ring isomorphism $\tch: K \xrightarrow{\sim} H$, $\tch^c$ is an isomorphism of modules \cite[Proposition 3.10]{BH15}. There is a commutative diagram
$$
    \xymatrix{
        K^c \ar[r] \ar[d]_{\tch^c} & K \ar[d]^{\tch} \\
        H^c \ar[r] & H.
    }
$$
Moreover, $\tch^c$ intertwines the Euler characteristic map \eqref{apx-eqn:EulerMap} and the evaluation \eqref{apx-eqn:SocleMap} via a combinatorial Hirzebruch-Riemann-Roch theorem \cite[Proposition 4.5]{BH15}.

\subsection{Sheaves with bounded below/above support}\label{apx-sect:BoundedDef}
Now we come to the main construction of the section. 
Let $\sv \in N$ be a non-zero primitive cocharacter of $\bT$. Let $\bT_0 \cong \bC^*$ be the 1-dimensional subtorus of $\bT$ determined by $\sv$ and $\bT_{0, \bR} \cong U(1)$ be the maximal compact subtorus. Let
$$
    \mu_{\bT_{0, \bR}}: \cX \to \bR
$$
denote the moment map where we use the identification of the dual Lie algebra of $\bT_{0, \bR}$ with $\bR$ provided by $\sv$. We choose $\sv$ generically such that for any flag $(\tau, \sigma) \in F(\Sigma)$, the weight of the $\bT_0$-action on the tangent line at $\fp_\sigma$ along $\fl_\tau$ is non-zero. Equivalently, the image of any $\fl_\tau$ under $\mu_{\bT_{0, \bR}}$ is non-constant.

Let $\Sigma^+$ (resp. $\Sigma^-$) denote the subset of cones $\sigma \in \Sigma$ such that $\mu_{\bT_{0, \bR}}(\cV(\sigma))$ is a bounded below (resp. above) subset of $\bR$. By definition, if $\sigma \in \Sigma^\pm$ and $\sigma' \supseteq \sigma$, then $\sigma' \in \Sigma^\pm$ also. Moreover, we have
$$
    \Sigma^+ \cap \Sigma^- = \Sigma^c.
$$
In particular, if $\sigma \in \Sigma$ is such that $I'_{\sigma} = I'_{\sigma^+} \cup I'_{\sigma^-}$ with $\sigma^\pm \in \Sigma^\pm$, then $\sigma \in \Sigma^c$. 

Let 
$$
    \cX^\pm := \bigcup_{\sigma \in \Sigma^\pm} \cV(\sigma)
$$
whose irreducible components are $\bT$-invariant closed substacks of $\cX$. Note that $\cX^+ \cap \cX^- = \cX^c$. Let $D^\pm(\cX)$ denote the full subcategory of $D(\cX)$ consisting of complexes of sheaves whose cohomology sheaves are supported on $\cX^\pm$, and
$$
    K^\pm(\cX) := K_0(D^\pm(\cX))
$$
denote its Grothendieck group. The group $K^\pm(\cX)$ is a module over the ring $K(\cX)$. The inclusion of categories induce natural $K(\cX)$-linear maps
$$
    K^c(\cX) \to K^\pm(\cX) \to K(\cX).
$$

\begin{remark}\label{apx-rmk:Cocharacter}\rm{
There is a wall-and-chamber structure on $N_{\bR}$ and the construction of $D^\pm(\cX)$, $K^\pm(\cX)$ depends only on the chamber that $\sv$ belongs to. The walls are spanned by facets of the cone $C$, which are dual to the non-compact $\bT$-invariant lines of $\cX$, and the genericity assumption on $\sv$ implies that it does not lie on any wall. The cone $C$ is itself a chamber, and if $\sv \in C$, one checks that
$$
    \Sigma^+ = \Sigma, \qquad \Sigma^- = \Sigma^c.
$$
Thus our discussion on $K^\pm(\cX)$ below will generalize the relation between $K(\cX)$ and $K^c(\cX)$ discussed above. In the special case where $C = N_{\bR}$, i.e. there is only one chamber, $\Sigma = \Sigma^c = \Sigma^\pm$ for any choice of $\sv$ and all $K$-groups are equal.
}\end{remark}

We provide a combinatorial description of $K^\pm(\cX)$ as in \cite{BH15}. Temporarily, let 
$
    K^\pm_{\comb}(\cX)
$
be the $K(\cX)$-module with generators
$$
    G_I, \qquad \text{for $I = I'_\sigma$ with $\sigma \in \Sigma^\pm$}
$$
and relations
\begin{itemize}
    \item $(1-R_i^{-1}) G_I - G_{I \cup \{i\}}$, for $i \not \in I$, $I \sqcup \{i\} \in \Sigma$;
    \item $(1-R_i^{-1}) G_I$, for $i \not \in I$, $I \sqcup \{i\} \not \in \Sigma$.
\end{itemize}
The natural map $K^c(\cX) \to K(\cX)$ factors through a natural map $K^c(\cX) \to K^\pm_{\comb}(\cX)$ given by $G_I \mapsto G_I$. We will prove the following analogue of \cite[Theorem 5.2]{BH15}.

\begin{theorem}\label{apx-thm:KpmPresentation}
We have
$
    K^\pm_{\comb}(\cX) \cong K^\pm(\cX).
$
\end{theorem}

The proof strategy is also analogous and is based on considering the natural map
\begin{equation}\label{apx-eqn:MuMap}
    \mu^\pm: K^\pm_{\comb}(\cX) \to K^\pm(\cX), \qquad \prod_{i=1}^{r+\fp'} R_i^{\alpha_i}G_I \mapsto  \iota_{\sigma, *} \cO_{\cV(\sigma)} \otimes \cO_{\cX}\left( \sum_{i=1}^{r+\fp'} \alpha_i \cD_i \right)
\end{equation}
where $\sigma \in \Sigma^\pm$ is such that $I = I'_\sigma$. We will show that $\mu^\pm$ is an isomorphism. The surjectivity is given by the following statement, which follows from the proof of \cite[Theorem 5.2]{BH15}.

\begin{lemma}\label{apx-lem:KGroupGenerator}
For $\star \in \{\emptyset, c, +, -\}$, the group $K^\star(\cX)$ is additively generated by sheaves of the form
$$
    \iota_{\sigma, *} \cL
$$
where $\sigma \in \Sigma^*$ is minimal and $\cL$ is a line bundle on $\cV(\sigma)$.
\end{lemma}

We will prove the injectivity of $\mu^\pm$ in Section \ref{apx-sect:PMEulerPairing} below.

\subsection{Cohomology with bounded below/above support}
Let
$$
    H^\pm = \bigoplus_{v \in \Box(\cX)} H^\pm_v,
$$
where on the untwisted sector, $H^\pm_v$ is the $H_v$-module presented as the quotient of the free $\bC[D_1, \dots, D_{r+\fp'}]$-module with generators
$$
    F_I, \qquad \text{for $I = I'_\sigma$ with $\sigma \in \Sigma^\pm$}
$$
and relations
\begin{itemize}
    \item $D_i F_I - F_{I \cup \{i\}}$, for $i \not \in I$, $I \sqcup \{i\} \in \Sigma$;
    \item $D_i F_I$, for $i \not \in I$, $I \sqcup \{i\} \not \in \Sigma$;
    \item elements given by the action of the ideal $Z_v$.
\end{itemize}
The definition of the factors on the twisted sectors is similar. The natural map $H^c_v \to H_v$ factors through a natural map $H^c_v \to H^\pm_v$ given by $F_I \mapsto F_I$.

\begin{lemma}\label{apx-lem:HPMDim}
We have 
$
    \dim_{\bC} H^\pm = \dim_{\bC} H = \dim_{\bC} H^c =: E.
$
\end{lemma}

\begin{proof}
We show equality of dimension on the untwisted sector $v \in \Box(\cX)$ and the argument for the twisted sectors is similar. The argument is inspired by that in \cite[Remark 2.5]{BH15} and describes $H_v^\pm$ in the context of \emph{relative} Stanley-Reisner theory, for which we refer to \cite{Stanley96}. Moreover, we may assume that $\Sigma^c \neq \Sigma^+ \neq \Sigma$ since otherwise $H_v^+ = H_v$ or $H_v^c$. This is possible only if $C \neq N_{\bR}$ and $r \ge 2$.

Let $R_v := \bC[D_1, \dots, D_{r+\fp'}]$. In this proof we use $\dim$ (resp. $\dim_{\bC}$) to denote the dimension as an $R_v$-module (resp. a $\bC$-vector space).

Let $A_v$ be the quotient of $R_v$ by the relations $\prod_{i \in I} D_i$ for $I \not \in \Sigma$. (We have $H_v = A_v /Z_vA_v$.) In other words, $A_v$ is the face ring $\bC[\Delta]$ of the simplicial complex $\Delta$ associated to the simplicial fan $\Sigma$, i.e. $\Delta$ is defined on the set of 1-cones in $\Sigma$ and is the set of subsets of 1-cones which span cones in $\Sigma$. Topologically, $\Delta$ is a triangulation of an $(r-1)$-ball.



Let $A_v^+$ be the $A_v$-module defined using the presentation of $H_v^+$ above except that the relations do not include the elements in $Z_v$. (We have $H_v^+ = A_v^+/Z_vA_v^+$.) On the other hand, let $\Gamma^+$ be the simplicial subcomplex of $\Delta$ associated to the subfan $\Sigma \setminus \Sigma^+$. Since $\Sigma^c \neq \Sigma^+ \neq \Sigma$, topologically $\Gamma^+$ is a triangulation of a non-empty, simply connected subset of $\partial \Delta$. The genericity of $\sv$ further ensures that $\Gamma^+$ is $(r-2)$-dimensional, so that it triangulates an $(r-2)$-ball. Now $A_v^+$ can be described as the relative face module associated to the relative simplicial complex $(\Delta, \Gamma^+)$:
$$
    A_v^+ = \bC[\Delta, \Gamma^+] := \ker \bC[\Delta] \twoheadrightarrow \bC[\Gamma^+].
$$
By \cite[Corollary III.7.3(iii)]{Stanley96}, $A_v^+$ is Cohen-Macaulay. Now $\dim \bC[\Delta] = r$ and $\dim \bC[\Gamma^+] = r-1$ \cite[Theorem II.1.3]{Stanley96}, which imply that $\dim A_v^+ = r$. Then $\depth$ $A_v^+ = r$ also.

To the $(r-1)$-dimensional relative simplicial complex $(\Delta, \Gamma^+)$, one associates the \emph{$f$-vector} $(f_{-1}, f_0, \dots, f_{r-1})$ where $f_{i-1}$ is the number of $(i-1)$-dimensional faces in $\Delta \setminus \Gamma^+$, or equivalently, the number of $i$-cones in $\Sigma^+$. Note that $\Sigma^+$ contains all the maximal cones in $\Sigma$ and thus
$$
    f_{r-1} = |\Sigma(r)|.
$$
Moreover, consider the $\bZ$-grading on $A_v^+$ defined by $\deg F_I = |I|$ for all $I$ and $\deg D_i = 1$ for all $i$. Let $\left(A_v^+ \right)_i$ denote the homogeneous degree-$i$ component of $A_v^+$. Then to $(\Delta, \Gamma^+)$ one also associates the \emph{$h$-vector} $(h_0, h_1, \dots, h_r)$ defined such that the Hilbert series of $\bC[\Delta, \Gamma^+] = A_v^+$ is of form
$$
    P_{A_v^+}(t) := \sum_{i \in \bZ_{\ge 0}} \dim_{\bC} \left(A_v^+ \right)_i t^i  = \frac{h_0 + h_1 t + \cdots h_r t^r}{(1-t)^r}.
$$
The $f$- and $h$-vectors are recovered from one another by
$$
    h_k = \sum_{i = 0}^k (-1)^{k-i} {r-i \choose k-i} f_{i-1}, \qquad f_{i-1} = \sum_{k=0}^i {r-k \choose i-k} h_k.
$$
In particular,
$
    h_0 + h_1 + \cdots + h_r = f_{r-1} = |\Sigma(r)|.
$

Now take a $\bZ$-basis $u_1, \dots, u_r$ of $M$ and consider the elements
$$
    \theta_j := \sum_{i=1}^{r+\fp'} \inner{u_j, b_i}D_i
$$
which are linear in $R_v$ and generate the ideal $Z_v$. Moreover, since $\dim A_v^+/Z_vA_v^+ = 0 = \dim A_v^+ - r$, $(\theta_1, \dots, \theta_r)$ is a linear sequence of parameters for $A_v^+$, and is furthermore a regular sequence since $A_v^+$ is Cohen-Macaulay \cite[Theorem I.5.9]{Stanley96}. Therefore, the Hilbert series of $H_v^+ = A_v^+/Z_vA_v^+$ is obtained from that of $A_v^+$ by
$$
    P_{H_v^+}(t) = (1-t)^r P_{A_v^+}(t) = h_0 + h_1 t + \cdots h_r t^r.
$$
In particular,
$$
    \dim_{\bC} H_v^+ = h_0 + h_1 + \cdots + h_r = |\Sigma(r)|.
$$
\end{proof}

Consider the multiplication $H^+_v \times H^-_v \to H^c_v$ defined by
$$
    \left( \prod_{i=1}^{r+\fp'} D_i^{\alpha_i}F_{I^+}, \prod_{i=1}^{r+\fp'} D_i^{\beta_i}F_{I^-} \right) \mapsto 
    \begin{cases}
        \displaystyle \prod_{i=1}^{r+\fp'} D_i^{\alpha_i + \beta_i} \prod_{i \in I^+ \cap I^-} D_i F_{I^+ \cup I^-} & \text{if } I^+ \cup I^- \in \Sigma,\\
        0 & \text{otherwise,}
    \end{cases}
$$
which after composition with the evaluation \eqref{apx-eqn:SocleMap} gives a pairing
\begin{equation}\label{apx-eqn:HPMPairing}
    (-,-): H^+ \times H^- \to \bC.
\end{equation}
We now construct bases of $H^\pm$ and evaluate the pairing on them. The elements in the bases will be indexed by the set
$$
    \fI := \{(\sigma, v) : \sigma \in \Sigma(r), v \in \Box(\sigma)\}
$$
whose cardinality is $E$. Let $\sigma \in \Sigma(r)$ and consider the untwisted sector $v \in \Box(\sigma)$; the construction for the twisted sectors is similar. The genericity of $\sv$ implies that there is a unique partition $I'_{\sigma} = I'_{\sigma^+} \sqcup I'_{\sigma^-}
$ such that $\sigma^\pm \in \Sigma^\pm$. Writing $I^\pm := I'_{\sigma^\pm}$, we set
$$
    a_{(\sigma, v)}^\pm := F_{I^\pm} \qquad \in H^\pm.
$$


\begin{lemma}\label{apx-lem:HPMBasis}
The $E$-by-$E$ matrix obtained by taking the pairing \eqref{apx-eqn:HPMPairing} between the collections $\{a_{(\sigma, v)}^\pm\}$ is invertible. In particular, each collection is a basis and the pairing \eqref{apx-eqn:HPMPairing} is non-degenerate.
\end{lemma}

\begin{proof}
We consider the pairing matrix under a fixed ordering of the elements $\{a_{(\sigma, v)}^\pm\}$, which is given by a total order on the index set $\fI$ chosen as follows. First, the pairs $(\sigma, v)$ are ordered according to the total order $\le$ on $\Sigma(r)$ determined by the values of $\fp_\sigma$ under the moment map $\mu_{\bT_{0, \bR}}$. Then for each $\sigma \in \Sigma(r)$, the $|\Box(\sigma)|$ pairs $(\sigma, v)$ are ordered in a fixed arbitrary way. The square submatrix given by
$$
    \left( a_{(\sigma, v_1)}^+, a_{(\sigma, v_2)}^- \right), \qquad v_1, v_2 \in \Box(\sigma)
$$
is diagonal with non-zero diagonal entries given by the reciprocals of the orders of the isotropy groups of $\fp_\sigma$ in the twisted sectors. Now, if $\sigma_1 > \sigma_2$, the genericity of $\sv$ implies that $\cV(\sigma_1^+) \cap \cV(\sigma_2^-) = \emptyset$, or equivalently $I'_{\sigma_1^+} \cup I'_{\sigma_2^-} \not  \in \Sigma$. This implies that
$$
    \left( a_{(\sigma_1, v_1)}^+, a_{(\sigma_2, v_2)}^- \right) = 0
$$
for any $v_1, v_2$. Therefore, the $E$-by-$E$ pairing matrix is upper-triangular with non-zero diagonal entries, and is thus invertible.
\end{proof}

\begin{remark}\rm{
Our construction of the bases is motivated by the Maulik-Okounkov \cite{MO19} stable envelopes for holomorphic symplectic varieties. Consider the case $\cX$ is smooth. The moment map $\mu_{\bT_{0, \bR}}: \cX \to \bR$ given by the choice of non-zero cocharacter $\sv$ of $\bT$ (or cocharacter chamber; see Remark \ref{apx-rmk:Cocharacter}) is a Morse function on $\cX$. Then in the above, $\cV(\sigma^\pm)$ is the closure of the attracting/repelling set of the critical point $\fp_\sigma$. We note that \cite{MO19} uses \emph{transitive closures} of attracting/repelling sets which is equally viable for our purpose.
}\end{remark}

\begin{remark}\rm{
Let $v \in \Box(\cX)$. Since the ring $H_v$ is Artinian local, its dualizing module $H_v^c$ is also its Matlis module, i.e. the injective hull of its residue field $\bC$. The proof of Lemma \ref{apx-lem:HPMBasis} shows that the multiplication $H^+_v \times H^-_v \to H^c_v$ is a non-degenerate pairing of $H_v$-modules, i.e. the induced map
$$
    H_v^\pm \to \Hom_{H_v}(H_v^\mp, H_v^c)
$$
is injective. Applying the exact functor $\Hom_{H_v}(-, H_v^c)$, we see that the map
$$
    H_v^\mp = \Hom_{H_v}(\Hom_{H_v}(H_v^\mp, H_v^c), H_v^c) \to \Hom_{H_v}(H_v^\pm, H_v^c)
$$
is surjective. Thus the maps are isomorphisms, the pairing $H^+_v \times H^-_v \to H^c_v$ is perfect, and $H_v^\pm$ are Matlis dual to each other.
} 
\end{remark}

Let $K^\pm  := K^\pm_{\comb}(\cX) \otimes \bC$. Similar to $\tch$ and $\tch^c$, we may define
$$
    \tch^\pm: K^\pm \to H^\pm
$$
where the projection $\tch_v^\pm: K^\pm \to H_v^\pm$ to the twisted sector $H_v^\pm$ is defined by
$$
    \tch_v^\pm \left(\prod_{i=1}^{r+\fp'}R_i^{\alpha_i} G_I \right) = \begin{cases}
        0 & \text{if } I \cup I'_{\sigma(v)} \not \in \Sigma,\\
        \displaystyle \prod_{i=1}^{r+\fp'} \tch_v(R_i)^{\alpha_i} \prod_{i \in I \setminus I'_{\sigma(v)}} \frac{1 - \exp (-D_i)}{D_i} \prod_{i \in I \cap I'_{\sigma(v)}} (1 - \tch_v(R_i)^{-1}) F_{\Ibar} & \text{if } I \cup I'_{\sigma(v)} \in \Sigma.   
    \end{cases}
$$
The proof of \cite[Proposition 3.10]{BH15} implies the following statement.

\begin{lemma}\label{apx-lem:PMChern}
The map $\tch^\pm$ is an isomorphism of $\bC$-vector spaces and an isomorphism of modules with respect to the ring isomorphism $\tch: K \xrightarrow{\sim} H$.
\end{lemma}

\subsection{Euler characteristic pairings}\label{apx-sect:PMEulerPairing}
On the categorical level, there is a natural pairing between $D^+(\cX)$ and $D^-(\cX)$ given by
\begin{equation}\label{apx-eqn:CatPairingFormula}
    \inner{\cE^+, \cE^-} := \sum_k (-1)^k \dim_{\bC} \Ext^k(\cE^+, \cE^-).
\end{equation}
The pairing is well-defined since by spectral sequences, the $\Ext$ groups are computed in terms of cohomology groups of local $\Ext$ sheaves from sheaves supported on $\cX^+$ to sheaves supported on $\cX^-$, which are finite dimensional because $\cX^+ \cap \cX^- = \cX^c$ is proper. Moreover, the pairing descends to
\begin{equation}\label{apx-eqn:CatPairing}
    \chi: K^+(\cX) \times K^-(\cX) \to \bZ.
\end{equation}
Combinatorially, consider the multiplication $K^+_{\comb}(\cX) \times K^-_{\comb}(\cX) \to K^c(\cX)$,
$$
    \left( \prod_{i=1}^{r+\fp'} R_i^{\alpha_i}G_{I^+}, \prod_{i=1}^{r+\fp'} R_i^{\beta_i}G_{I^-} \right) \mapsto \prod_{i=1}^{r+\fp'} R_i^{\beta_i - \alpha_i} (-1)^{|I^+|} \prod_{i \in I^+} R_i \prod_{i \in I^+ \cap I^-} (1 - R_i^{-1}) G_{I^+ \cup I^-}
$$
where the $K^+_{\comb}(\cX)$-factor is dualized. This composed with the Euler characteristic map $\chi: K^c(\cX) \to \bZ$ \eqref{apx-eqn:EulerMap} gives a pairing
\begin{equation}\label{apx-eqn:CombPairing}
    \chi_{\comb}: K^+_{\comb}(\cX) \times K^-_{\comb}(\cX) \to \bZ.
\end{equation}

We now compute the categorical pairing and compare it to the combinatorial version. The computation technique follows \cite{BO19}. For an index set $I$, consider the Koszul complex
$$
    \cK_I^\bullet = \bigotimes_{i \in I} \left(\cO_{\cX}(-\cD_i) \to \cO_{\cX} \right) =  \left(\cE_I^{|I|} \to \cdots \to \cE_I^1 \to \cE_I^0 \right) 
$$
where for $s = 0, \dots, |I|$,
$$
    \cE_I^s = \bigoplus_{S \subseteq I, |S| = s} \cO_{\cX}\left( - \sum_{i \in S} \cD_i \right)
$$
and has cohomological degree $-s$. In particular,
$$
    \cE_I^0 = \cO_{\cX}, \qquad \cE_I^{|I|} = \cO_{\cX}\left( - \sum_{i \in I} \cD_i \right) =: \cL_I^\vee.
$$
In the Grothendieck group, we have
$$
    \sum_{s \in \bZ_{\ge 0}} (-1)^s [\cE_I^s] = \prod_{i \in I} \left( [\cO_{\cX}] - [\cO_{\cX}(-\cD_i)] \right).
$$

\begin{lemma}\label{apx-lem:EulerPairingCompute}
For line bundles $\cL_1, \cL_2$ on $\cX$ and $\sigma^\pm \in \Sigma^\pm$, setting $I^\pm = I'_{\sigma^\pm}$ and $\sigma \in \Sigma^c$ with $I'_\sigma = I^+ \cup I^-$, we have
$$
    \chi \left(\iota_{\sigma^+, *} \cO_{\cV(\sigma^+)} \otimes \cL_1, \iota_{\sigma^-, *} \cO_{\cV(\sigma^-)} \otimes \cL_2 \right) = (-1)^{|I^+|} \sum_{s \in \bZ_{\ge 0}} (-1)^{s} \chi \left(\iota_{\sigma, *} \cO_{\cV(\sigma)} \otimes \cE_{I^+ \cap I^-}^s \otimes \cL_{I^+} \otimes \cL_1^\vee \otimes \cL_2 \right).
$$
In particular, the pairing $\chi$ agrees with $\chi_{\comb}$ on the images of $\mu^\pm$ \eqref{apx-eqn:MuMap}.
\end{lemma}

\begin{proof}
The Koszul complex $\cK_{I^+}^\bullet$ is a locally free resolution of $\iota_{\sigma^+, *} \cO_{\cV(\sigma^+)}$. By the local-to-global spectral sequence for $\Ext$ groups, we have
\begin{align*}
    & \chi \left(\iota_{\sigma^+, *} \cO_{\cV(\sigma^+)} \otimes \cL_1, \iota_{\sigma^-, *} \cO_{\cV(\sigma^-)} \otimes \cL_2 \right) \\
    & = \sum_{s, i \in \bZ_{\ge 0}} (-1)^{s+i} h^i(\cX, \cHom(\cE_{I^+}^s, \iota_{\sigma^-, *} \cO_{\cV(\sigma^-)} \otimes \cL_1^\vee \otimes \cL_2)) \\
    & = \sum_{s, i \in \bZ_{\ge 0}} (-1)^{s+i} h^i \left(\cV(\sigma^-), \left(\cE_{I^+}^s\right)^\vee \otimes \cL_1^\vee \otimes \cL_2 \big|_{\cV(\sigma^-)}\right) \\
    & = (-1)^{|I^+|} \sum_{s, i \in \bZ_{\ge 0}} (-1)^{s+i} h^i \left(\cV(\sigma^-), \cE_{I^+}^s \otimes \cL_{I^+} \otimes \cL_1^\vee \otimes \cL_2 \big|_{\cV(\sigma^-)}\right) \\
    & = (-1)^{|I^+|} \sum_{s, i \in \bZ_{\ge 0}} (-1)^{s+i} h^i \left(\cV(\sigma), \cE_{I^+ \cap I^-}^s \otimes \cL_{I^+} \otimes \cL_1^\vee \otimes \cL_2 \big|_{\cV(\sigma)}\right) \\
    & = (-1)^{|I^+|} \sum_{s \in \bZ_{\ge 0}} (-1)^{s} \chi \left(\iota_{\sigma, *} \cO_{\cV(\sigma)} \otimes \cE_{I^+ \cap I^-}^s \otimes \cL_{I^+} \otimes \cL_1^\vee \otimes \cL_2 \right).
\end{align*}
\end{proof}

\begin{lemma}\label{apx-lem:EulerPairingND}
The pairing $\chi_{\comb}$ \eqref{apx-eqn:CombPairing} is non-degenerate.
\end{lemma}

\begin{proof}
We follow the proof of \cite[Corollary 4.6]{BH15}. The non-degeneracy of $\chi_{\comb}$ is equivalent to that of another pairing 
$$
    \chi'_{\comb}: K^+_{\comb}(\cX) \times K^-_{\comb}(\cX) \to \bZ
$$
defined by the multiplication $K^+_{\comb}(\cX) \times K^-_{\comb}(\cX) \to K^c(\cX)$,
$$
    \left( \prod_{i=1}^{r+\fp'} R_i^{\alpha_i}G_{I^+}, \prod_{i=1}^{r+\fp'} R_i^{\beta_i}G_{I^-} \right) \mapsto \prod_{i=1}^{r+\fp'} R_i^{\alpha_i + \beta_i} \prod_{i \in I^+ \cap I^-} (1 - R_i^{-1}) G_{I^+ \cup I^-}
$$
where the $K^+_{\comb}(\cX)$-factor is not dualized. On the other hand, the isomorphisms $\tch^\pm$ intertwine $\chi'_{\comb}$ (extended over $\bC$) and the cohomological pairing \eqref{apx-eqn:HPMPairing} which is non-degenerate by Lemma \ref{apx-lem:HPMBasis}. Thus $\chi'_{\comb}$ is also non-degenerate.
\end{proof}
    
\begin{proof}[Proof of Theorem \ref{apx-thm:KpmPresentation}]
The theorem follows from that the map $\mu^\pm$ \eqref{apx-eqn:MuMap} is an isomorphism. The surjectivity is given in Section \ref{apx-sect:BoundedDef} and the injectivity follows from Lemmas \ref{apx-lem:EulerPairingCompute}, \ref{apx-lem:EulerPairingND}.
\end{proof}

\begin{remark} \rm{
It would be interesting to investigate whether the pairing $\chi$ \eqref{apx-eqn:CatPairing} is unimodular (cf. \cite[Remark 4.7]{BH15}).
}
\end{remark}

\subsection{Equivariant case}\label{apx-sect:Equivariant}
We now consider the equivariant case with respect to a subtorus $\bT' \subseteq \bT$. Let $D_{\bT'}(\cX) := D^b(\Coh_{\bT'}(\cX))$ denote the bounded derived category of $\bT'$-equivariant coherent sheaves on $\cX$ and $K_{\bT'}(\cX) := K_0(D_{\bT'}(\cX))$ denote its Grothendieck group. For $\star \in \{c, +, -\}$, let $D^\star_{\bT'}(\cX)$ denote the full subcategory of $D_{\bT'}(\cX)$ consisting of complexes of sheaves whose cohomology sheaves are supported on $\cX^\star$, and $K^\star_{\bT'}(\cX) := K_0(D^\star_{\bT'}(\cX))$ denote its Grothendieck group. We have the following straightforward generalization of Lemma \ref{apx-lem:KGroupGenerator}.

\begin{lemma}\label{apx-lem:EquivKGroupGenerator}
For $\star \in \{\emptyset, c, +, -\}$, the group $K^\star_{\bT'}(\cX)$ is additively generated by sheaves of the form
$$
    \iota_{\sigma, *} \cL
$$
where $\sigma \in \Sigma^\star$ is minimal and $\cL$ is a $\bT'$-equivariant line bundle on $\cV(\sigma)$.
\end{lemma}

The group $K^\star_{\bT'}(\cX)$ is a module over the ring $K_{\bT'}(\cX)$. The inclusions $\Sigma^c \subseteq \Sigma^\pm \subseteq \Sigma$ induce natural $K_{\bT'}(\cX)$-linear maps
\begin{equation}\label{apx-eqn:EquivKNatural} 
    K^c_{\bT'}(\cX) \to K^\pm_{\bT'}(\cX) \to K_{\bT'}(\cX).
\end{equation}

\begin{lemma}\label{apx-lem:EquivKInjective}
Assume that
\begin{equation}\label{apx-eqn:TActionCond}
    \cX^{\bT'} = \cX^{\bT} = \bigsqcup_{\sigma \in \Sigma(r)} \fp_\sigma.
\end{equation}
Then the natural maps \eqref{apx-eqn:EquivKNatural} are injective.
\end{lemma}

Condition \eqref{apx-eqn:TActionCond} implies that the tangent $\bT'$-weights at any $\fp_\sigma$ along any torus-invariant curve $\fl_\tau$ is non-trivial, so that we may apply localization on the $\bT'$-equivariant $K$-groups of toric substacks of $\cX$ \cite{Thomason87, Thomason92}. The condition is met for instance when $\cX$ is Calabi-Yau and $\bT'$ is the Calabi-Yau subtorus. Note that in the non-equivariant case the lemma fails in general.

\begin{proof}
In view of Remark \ref{apx-rmk:Cocharacter}, it suffices to show the injectivity of $K^+_{\bT'}(\cX) \to K_{\bT'}(\cX)$, which we denote by $i$. We will make use of a filtration of the support $\cX^+$ by closed substacks constructed as follows. Consider the total order $\sigma_1 < \cdots < \sigma_{|\Sigma(r)|}$ on $\Sigma(r)$ defined by $\bT_0$ as in the proof of Lemma \ref{apx-lem:HPMBasis}, and for each $i$ let $\sigma_i^+ \in \Sigma^+$ be the cone associated to $\sigma_i$ there. Setting
$$
    \cV_i := \bigcup_{j \ge i} \cV(\sigma_j^+),
$$
we have a descending chain of closed substacks
$$
    \cX^+ = \cV_1^+ \supseteq \cdots \supseteq \cV_{|\Sigma(r)|}^+  \supset \cV_{|\Sigma(r)|+1}^+ = \emptyset.
$$

Now take $F \in K^+_{\bT'}(\cX)$ such that $i(F) = 0 \in K_{\bT'}(\cX)$. By definition, $F$ is supported on $\cV_1$. Note that $\cV(\sigma_1^+)$ is an irreducible component of $\cV_1$ whose complement is contained in $\cV_2$. We may write
\begin{equation}\label{eqn:FDecompose}
    F = \iota_{\sigma_1^+, *}(F_1) + F'
\end{equation}
for some $F_1 \in K_{\bT'}(\cV(\sigma_1^+))$ and $F' \in K_{\bT'}^+(\cX)$ such that $F'$ is supported on $\cV_2$. Note that $\fp_{\sigma_1}$ is the only $\bT'$-fixed point in $\cV_1 \setminus \cV_2$. We have a sequence of maps
$$
    \xymatrix{
        K_{\bT'}(\fp_{\sigma_1}) \ar[r] & K_{\bT'}(\cV(\sigma_1^+)) \ar[r]^{\iota_{\sigma_1^+, *}} & K_{\bT'}^+(\cX) \ar[r]^{i} & K_{\bT'}(\cX).
    }
$$
whose composition is the pushforward $\iota_{\sigma_1,*}: K_{\bT'}(\fp_{\sigma_1}) \to K_{\bT'}(\cX)$. By $\bT'$-equivariant localization on $\cX$, the pushforward $K_{\bT'}(\cX^{\bT'}) \to K_{\bT'}(\cX)$ is an isomorphism after localizing at the finite collection of elements $[\cL_\chi] - [\cO_{\cX}]$ where $\chi$ ranges through the non-trivial characters appearing in the tangent $\bT'$-weights at the fixed points and $\cL_\chi$ is the corresponding line bundle. A similar statement holds for $\bT'$-equivariant localization on $\cV(\sigma_1^+)$. Therefore in the decomposition \eqref{eqn:FDecompose}, since $i(F) = 0$ and the support of $F'$ is disjoint from $\fp_{\sigma_1}$, we have $F_1 \big|_{\fp_{\sigma_1}} = 0$. Thus $F_1$ is supported on the maximal $\bT'$-invariant closed substack of $\cV(\sigma_1^+)$ that is disjoint from $\fp_{\sigma_1}$, which is in fact contained in $\cV_2$. This implies that $F$ is supported on $\cV_2$. Continuing the argument inductively, we eventually arrive at $F = 0$.
\end{proof}

We expect much of the discussion on the non-equivariant case $K^\star(\cX)$, $\star \in \{\emptyset, c, +, -\}$, in this appendix to be generalizable to the $\bT'$-equivariant case, including a combinatorial characterization of $K^\star_{\bT'}(\cX)$ similar to Theorem \ref{apx-thm:KpmPresentation}. Moreover, the formula \eqref{apx-eqn:CatPairingFormula} defines a pairing between $D^+_{\bT'}(\cX)$ and $D^-_{\bT'}(\cX)$ which descends to $\chi_{\bT'}: K^+_{\bT'}(\cX) \times K^-_{\bT'}(\cX) \to \bZ$, and we expect to generalize the study in Section \ref{apx-sect:PMEulerPairing}. We defer a detailed treatment to future work.

\section{Oscillatory integrals solve equivariant GKZ}\label{appdx:PF}
In this section, we prove Proposition \ref{prop:IntSolvesPF} that oscillatory integrals on the mirror curve are solutions to the $\bT'$-equivariant GKZ-type system \eqref{eqn:GKZ}. It is clear from the definitions that the integrals are annihilated by $\bE$, and thus we focus on the operators $\bD_\beta^{\bT'}$ for $\beta \in \bL$. We work with integrals on $C_q$ and consider a flat relative cycle
$
    \gamma \in H_1(C_q, \Re (\hx) \gg 0; \bZ).
$
The statement about integrals on $\tC_q$ can be reduced to this case by projecting relative cycles under the covering map $\tC_q \to C_q$.

Recall from Section \ref{sec:mirror-curve} that the equation of the mirror curve $C_q$ is
$$
    H(X, Y, q) = X^{\fr}Y^{-\fs} + Y^{\fm} + 1 + \sum_{i = 4}^{3 + \fp} a_i(q)X^{m_i}Y^{n_i}
$$
where $a_i(q) = \prod_{a = 1}^{\fp} q_a^{s_{ai}}$. We start with some elementary computations.

\begin{lemma}\label{lem:LinearAlgebra}
Let $j \in \{4, \dots, 3+\fp\}$. For $i = 1, 2, 3$ we have
$$
\sum_{a = 1}^{\fp} m_1^{(a)}s_{aj} = - \frac{1}{\fr}m_j, \quad 
\sum_{a = 1}^{\fp} m_2^{(a)}s_{aj} = - \frac{\fs}{\fr \fm} m_j - \frac{1}{\fm}n_j, \quad
\sum_{a = 1}^{\fp} m_3^{(a)}s_{aj} = - 1 + \frac{\fs + \fm}{\fr \fm} m_j + \frac{1}{\fm}n_j.
$$
For $i \in \{4, \dots, 3+\fp\}$ we have 
$
    \sum_{a = 1}^{\fp} m_i^{(a)}s_{aj} = \begin{cases}
        1 & \text{if } i = j,\\
        0 & \text{if } i \neq j.
    \end{cases}
$
\end{lemma}

\begin{proof}
The case $i \in \{4, \dots, 3+\fp\}$ directly follows from the definitions. The case $i = 1$ case follows from reading off the coefficient of $D_j$ from
$$
    \sum_{j = 4}^{3+\fp} (-m_j) D_j = \fr D_1 = \sum_{a = 1}^{\fp} \fr m_1^{(a)} H_a = \sum_{j = 4}^{3+\fp} \sum_{a = 1}^{\fp} \fr m_1^{(a)}s_{aj} D_j.
$$
The case $i = 2$ follows from the case $i = 1$ and
$$
    \sum_{j = 4}^{3+\fp} (-n_j) D_j = -\fs D_1 + \fm D_2 = \sum_{a = 1}^{\fp} (-\fs m_1^{(a)} + \fm m_2^{(a)})H_a = \sum_{j = 4}^{3+\fp} \sum_{a = 1}^{\fp} (-\fs m_1^{(a)} + \fm m_2^{(a)})s_{aj} D_j.
$$
The case $i = 3$ follows from the cases $i = 1, 2$ and
$$
    \sum_{j = 4}^{3+\fp} - D_j = D_1 + D_2 + D_3 = \sum_{a = 1}^{\fp} (m_1^{(a)} + m_2^{(a)} + m_3^{(a)}) H_a = \sum_{j = 4}^{3+\fp} \sum_{a = 1}^{\fp} (m_1^{(a)} + m_2^{(a)} + m_3^{(a)})s_{aj} D_j.
$$
\end{proof}

Consider
$$
    H_X := \frac{\partial H}{\partial \log X} = \fr X^{\fr}Y^{-\fs} + \sum_{i = 4}^{3 + \fp} m_ia_i(q)X^{m_i}Y^{n_i}, \quad
    H_Y := \frac{\partial H}{\partial \log Y} = -\fs X^{\fr}Y^{-\fs} + \fm Y^{\fm} + \sum_{i = 4}^{3 + \fp} n_ia_i(q)X^{m_i}Y^{n_i}.
$$

\begin{lemma}\label{lem:PartialH}
For $i = 1, 2, 3$ we have
$$
    \partial_1 H = X^{\fr}Y^{-\fs} - \frac{1}{\fr}H_X, \quad
    \partial_2 H = Y^{\fm} - \frac{\fs}{\fr \fm} H_X - \frac{1}{\fm}H_Y, \quad
    \partial_3 H = 1 - H + \frac{\fs + \fm}{\fr \fm} H_X + \frac{1}{\fm}H_Y.
$$
For $i \in \{4, \dots, 3+\fp\}$ we have $\partial_i H = a_i(q)X^{m_i}Y^{n_i}$.
\end{lemma}

\begin{proof}
For $i = 1, \dots, 3 + \fp$, we have
$$
    \partial_i H =  \sum_{a = 1}^{\fp} m_i^{(a)} \partial_a H = \sum_{a = 1}^{\fp} m_i^{(a)} \sum_{j = 4}^{3+\fp} s_{aj}a_j(q)X^{m_j}Y^{n_j} = \sum_{j = 4}^{3+\fp} \left( \sum_{a = 1}^{\fp} m_i^{(a)}s_{aj}\right) a_j(q)X^{m_j}Y^{n_j}.
$$
Below we apply Lemma \ref{lem:LinearAlgebra}. The case $i \in \{4, \dots, 3 + \fp\}$ directly follows. For $i = 1$, we have
$$
    \sum_{j = 4}^{3+\fp} \left( \sum_{a = 1}^{\fp} m_1^{(a)}s_{aj}\right) a_j(q)X^{m_j}Y^{n_j} = \sum_{j = 4}^{3+\fp} -\frac{1}{\fr} m_ja_j(q)X^{m_j}Y^{n_j} = X^{\fr}Y^{-\fs} - \frac{1}{\fr} H_X.
$$
For $i = 2$, we have
$$
    \sum_{j = 4}^{3+\fp} \left( \sum_{a = 1}^{\fp} m_2^{(a)}s_{aj}\right) a_j(q)X^{m_j}Y^{n_j} = \sum_{j = 4}^{3+\fp} \left(-\frac{\fs}{\fr \fm} m_j - \frac{1}{\fm}n_j \right)a_j(q)X^{m_j}Y^{n_j} = Y^{\fm} - \frac{\fs}{\fr \fm} H_X - \frac{1}{\fm}H_Y.
$$
Finally for $i = 3$, we have
$$
    \sum_{j = 4}^{3+\fp} \left( \sum_{a = 1}^{\fp} m_3^{(a)}s_{aj}\right) a_j(q)X^{m_j}Y^{n_j} = \sum_{j = 4}^{3+\fp} \left(-1+\frac{\fs + \fm}{\fr \fm} m_j + \frac{1}{\fm}n_j \right)a_j(q)X^{m_j}Y^{n_j}
    = 1 - H + \frac{\fs + \fm}{\fr \fm} H_X + \frac{1}{\fm}H_Y.
$$
\end{proof}

We denote
$
    \omega := \frac{dXdY}{XY}, \Phi = \hy d\hx = \log Y d\left(\log X + \frac{u_2}{u_1} \log Y \right).
$
Let
$$
    \Res = \Res_{C_q}: H^2((\bC^*)^2 \setminus C_q) \to H^1(C_q)
$$
denote the Poincar\'e residue operator. As computed in \cite[Section 5.3]{flz2020remodeling}, for any $a = 1, \dots, \fp$, we have
$$
    \partial_a \Phi = \partial_a \left(\log Y d\log X \right)= \Res \left(\frac{\partial_a H}{H}\omega \right).
$$
Therefore for any $i = 1, \dots, 3+\fp$, we have
$$
    \partial_i \Phi = \Res \left(\frac{\partial_i H}{H}\omega \right).
$$
Moreover, for any $(m, n) \in \bZ^2$, we have
$$
    \Res \left(\frac{X^mY^n}{H}\omega \right) = -\frac{X^mY^n}{H_Y} d\log X.
$$
In particular, we have
$$
    \Res \left(\frac{H_X}{H}\omega \right) = -\frac{H_X}{H_Y}d\log X = d\log Y = -u_1 d\hy, \quad
    \Res \left(\frac{H_Y}{H}\omega \right) = -\frac{H_Y}{H_Y}d\log X = -d\log X = \frac{1}{u_1}d\hx - u_2d\hy.
$$

\begin{lemma}\label{lem:Equiv1}
For $i = 1, \dots, 3+\fp$, we have
$$
    \left(\partial_i - \frac{w^i}{z} \right) \int_\gamma e^{-\hx/z}\Phi = \int_\gamma e^{-\hx/z} \Res \left(\frac{a_i(q)X^{m_i}Y^{n_i}}{H} \omega \right).
$$
\end{lemma}

\begin{proof}
We apply Lemma \ref{lem:PartialH} to
$
    \partial_i \Phi = \Res \left(\frac{\partial_i H}{H}\omega \right).
$
For $i \in \{4, \dots, 3 + \fp\}$, where $w^i = 0$, the lemma directly follows. For $i = 1$, we have
\begin{align*}
    \left(\partial_1 - \frac{w^1}{z} \right) \int_\gamma e^{-\hx/z} \Phi &= \int_\gamma e^{-\hx/z}\Res \left(\frac{X^{\fr}Y^{-\fs} - \frac{1}{\fr} H_X}{H}\omega \right) - \frac{u_1}{\fr}\frac{e^{-\hx/z}}{z}\hy d\hx \\
    & = \int_\gamma e^{-\hx/z}\Res \left(\frac{X^{\fr}Y^{-\fs}}{H}\omega \right) + \frac{u_1}{\fr}d\left(e^{-\hx/z}\hy\right) \\
    & = \int_\gamma e^{-\hx/z}\Res \left(\frac{X^{\fr}Y^{-\fs}}{H}\omega \right).
\end{align*}
For $i = 2$, we have
\begin{align*}
    \left(\partial_2 - \frac{w^2}{z}\right) \int_\gamma e^{-\hx/z}\Phi & = \int_\gamma e^{-\hx/z}\Res \left(\frac{Y^{\fm} - \frac{\fs}{\fr \fm} H_X - \frac{1}{\fm}H_Y}{H}\omega \right) - \left(\frac{\fs u_1}{\fr \fm} + \frac{u_2}{\fm}  \right)\frac{e^{-\hx/z}}{z}\hy d\hx\\
    & = \int_\gamma e^{-\hx/z}\Res \left(\frac{Y^{\fm}}{H}\omega \right) + \left(\frac{\fs u_1}{\fr \fm} + \frac{u_2}{\fm}  \right)d\left(e^{-\hx/z}\hy\right) + \frac{z}{\fm u_1}d\left(e^{-\hx/z}\right) \\
    & = \int_\gamma e^{-\hx/z}\Res \left(\frac{Y^{\fm}}{H}\omega \right).
\end{align*}
Finally for $i = 3$, we have
\begin{align*}
    \left(\partial_3 - \frac{w^3}{z} \right) \int_\gamma \Phi & = \int_\gamma e^{-\hx/z} \Res \left(\frac{1 - H + \frac{\fs + \fm}{\fr \fm} H_X + \frac{1}{\fm}H_Y}{H}\omega \right) + \left(\frac{(\fs + \fm) u_1}{\fr \fm} + \frac{u_2}{\fm} \right) \frac{e^{-\hx/z}}{z}\hy d\hx \\
    & = \int_\gamma e^{-\hx/z}\Res \left(\frac{1}{H}\omega - \omega \right) -\left(\frac{(\fs + \fm) u_1}{\fr \fm} + \frac{u_2}{\fm} \right)d\left(e^{-\hx/z}\hy\right) - \frac{z}{\fm u_1}d\left(e^{-\hx/z}\right) \\
    & = \int_\gamma e^{-\hx/z}\Res \left(\frac{1}{H}\omega \right).
\end{align*}
\end{proof}

\begin{lemma}\label{lem:Equiv2}
Let $k \in \bZ_{\ge 1}$ with a splitting
$
    k = \sum_{i = 1}^{3 + \fp} k_i
$
where $k_i \in \bZ_{\ge 0}$ for all $i$. For $i = 1, \dots, 3+\fp$, we have
\begin{align*}
    \left( \partial_i - k_i - \frac{w^i}{z} \right) & \int_\gamma (-1)^{k-1}(k-1)! e^{-\hx/z}\Res \left(\frac{\prod_{j = 1}^{3 + \fp} \left(a_j(q)X^{m_j}Y^{n_j}\right)^{k_j}}{H^k} \omega \right) \\
    = & \int_\gamma (-1)^{k}k! e^{-\hx/z}\Res \left(\frac{a_i(q)X^{m_i}Y^{n_i}\prod_{j = 1}^{3 + \fp} \left(a_j(q)X^{m_j}Y^{n_j}\right)^{k_j}}{H^{k+1}} \omega \right).
\end{align*}
In other words, applying the operator $\partial_i - k_i - \frac{w^i}{z}$ increases the index $k_i$ (and thus $k$) by 1.
\end{lemma}

\begin{proof}
We have
\begin{align*}
    & \left(\partial_i - k_i - \frac{w^i}{z} \right) \frac{e^{-\hx/z}\prod_{j = 1}^{3 + \fp} \left(a_j(q)X^{m_j}Y^{n_j}\right)^{k_j}}{H^k}  \\
    & = \left( \sum_{a = 1}^{\fp} m_i^{(a)} \sum_{j = 4}^{3+\fp} k_js_{aj} H  -k \partial_i H - k_iH - \frac{w^i}{z}H\right) \frac{e^{-\hx/z}\prod_{j = 1}^{3 + \fp} \left(a_j(q)X^{m_j}Y^{n_j}\right)^{k_j}}{H^{k+1}}\\
    & = \left(\left(-k_i - \frac{w^i}{z} + \sum_{j = 4}^{3+\fp} k_j \left(\sum_{a = 1}^{\fp} m_i^{(a)} s_{aj}\right)   \right) H  -k \partial_i H\right) \frac{e^{-\hx/z}\prod_{j = 1}^{3 + \fp} \left(a_j(q)X^{m_j}Y^{n_j}\right)^{k_j}}{H^{k+1}} .
\end{align*}
We proceed by applying Lemmas \ref{lem:LinearAlgebra}, \ref{lem:PartialH}. For $i \in \{4, \dots, 3 + \fp\}$, where $w^i = 0$, we have
$$
    \left(-k_i + \sum_{j = 4}^{3+\fp} k_j \left(\sum_{a = 1}^{\fp} m_i^{(a)} s_{aj}\right) \right) H  -k \partial_i H = \left(-k_i + k_i\right) H - k a_i(q)X^{m_i}Y^{n_i} = - k a_i(q)X^{m_i}Y^{n_i}
$$
which implies the lemma. For $i = 1$, we have
\begin{align*}
    & \left(\left(-k_1 - \frac{w^1}{z} + \sum_{j = 4}^{3+\fp} k_j \left(\sum_{a = 1}^{\fp} m_1^{(a)} s_{aj}\right) \right) H  -k \partial_1 H \right)  \frac{e^{-\hx/z}\prod_{j = 1}^{3 + \fp} \left(a_j(q)X^{m_j}Y^{n_j}\right)^{k_j}}{H^{k+1}} \\
    & =  \left( \left(-k_1 - \frac{u_1}{\fr z}  - \frac{1}{\fr} \sum_{j = 4}^{3+\fp} k_j m_j \right) H  -k \left( X^{\fr}Y^{-\fs} - \frac{1}{\fr}H_X \right)  \right)  \frac{e^{-\hx/z}\prod_{j = 1}^{3 + \fp} \left(a_j(q)X^{m_j}Y^{n_j}\right)^{k_j}}{H^{k+1}}\\
    & = \left( -\frac{k X^{\fr}Y^{-\fs}}{H} - \frac{1}{\fr} \frac{\partial}{\partial \log X} \right) \frac{e^{-\hx/z}\prod_{j = 1}^{3 + \fp} \left(a_j(q)X^{m_j}Y^{n_j}\right)^{k_j}}{H^k}.
\end{align*}
Then
\begin{align*}
    & \Res \left( \left( -\frac{k X^{\fr}Y^{-\fs}}{H} - \frac{1}{\fr} \frac{\partial}{\partial \log X} \right) \frac{e^{-\hx/z}\prod_{j = 1}^{3 + \fp} \left(a_j(q)X^{m_j}Y^{n_j}\right)^{k_j}}{H^k} \omega \right)\\
    & = \Res \left( -k X^{\fr}Y^{-\fs} \frac{e^{-\hx/z}\prod_{j = 1}^{3 + \fp} \left(a_j(q)X^{m_j}Y^{n_j}\right)^{k_j}}{H^{k+1}} \omega -  d \left( \frac{e^{-\hx/z}\prod_{j = 1}^{3 + \fp} \left(a_j(q)X^{m_j}Y^{n_j}\right)^{k_j}}{H^k} \frac{d\log Y}{\fr} \right) \right)\\
    & = \Res \left( -k X^{\fr}Y^{-\fs} \frac{e^{-\hx/z}\prod_{j = 1}^{3 + \fp} \left(a_j(q)X^{m_j}Y^{n_j}\right)^{k_j}}{H^{k+1}} \omega \right)
\end{align*}
which implies the lemma. For $i = 2$, we have
\begin{align*}
    & \left(\left(-k_2 - \frac{w^2}{z} + \sum_{j = 4}^{3+\fp} k_j \left(\sum_{a = 1}^{\fp} m_2^{(a)} s_{aj}\right) \right) H  -k \partial_2 H \right)  \frac{e^{-\hx/z}\prod_{j = 1}^{3 + \fp} \left(a_j(q)X^{m_j}Y^{n_j}\right)^{k_j}}{H^{k+1}} \\
    & =  \left( \left(-k_2 - \frac{\fs u_1}{\fr \fm z} - \frac{\fs}{\fr \fm}\sum_{j = 4}^{3+\fp} k_j m_j - \frac{u_2}{\fm z} - \frac{1}{\fm} \sum_{j = 4}^{3+\fp} k_j n_j  \right) H  -k \left( Y^{\fm} - \frac{\fs}{\fr \fm}H_X - \frac{1}{\fm}H_Y \right)  \right)  \\
        & \quad \cdot \frac{e^{-\hx/z}\prod_{j = 1}^{3 + \fp} \left(a_j(q)X^{m_j}Y^{n_j}\right)^{k_j}}{H^{k+1}}\\
    & = \left( -\frac{k Y^{\fm}}{H} - \frac{\fs}{\fr\fm} \frac{\partial}{\partial \log X} - \frac{1}{\fm} \frac{\partial}{\partial \log Y}  \right) \frac{e^{-\hx/z}\prod_{j = 1}^{3 + \fp} \left(a_j(q)X^{m_j}Y^{n_j}\right)^{k_j}}{H^k}.
\end{align*}
Then
\begin{align*}
    & \Res \left( \left(-\frac{k Y^{\fm}}{H} - \frac{\fs}{\fr\fm} \frac{\partial}{\partial \log X} - \frac{1}{\fm} \frac{\partial}{\partial \log Y} \right) \frac{e^{-\hx/z}\prod_{j = 1}^{3 + \fp} \left(a_j(q)X^{m_j}Y^{n_j}\right)^{k_j}}{H^k} \omega \right) \\
    & = \Res \left( -k Y^{\fm} \frac{e^{-\hx/z}\prod_{j = 1}^{3 + \fp} \left(a_j(q)X^{m_j}Y^{n_j}\right)^{k_j}}{H^{k+1}} \omega -  d \left( \frac{e^{-\hx/z}\prod_{j = 1}^{3 + \fp} \left(a_j(q)X^{m_j}Y^{n_j}\right)^{k_j}}{H^k} \left(\frac{\fs d\log Y}{\fr\fm} + \frac{d\log X}{\fm} \right) \right) \right) \\
    & = \Res \left( -k Y^{\fm} \frac{e^{-\hx/z}\prod_{j = 1}^{3 + \fp} \left(a_j(q)X^{m_j}Y^{n_j}\right)^{k_j}}{H^{k+1}} \omega \right)
\end{align*}
which implies the lemma. Finally for $i = 3$, we have
\begin{align*}
    & \left(\left(-k_3 - \frac{w^3}{z} + \sum_{j = 4}^{3+\fp} k_j \left(\sum_{a = 1}^{\fp} m_3^{(a)} s_{aj}\right) \right) H  -k \partial_3 H \right)  \frac{e^{-\hx/z}\prod_{j = 1}^{3 + \fp} \left(a_j(q)X^{m_j}Y^{n_j}\right)^{k_j}}{H^{k+1}} \\
    & =  \left( \left(-k_3 - \sum_{j = 4}^{3+\fp} k_j + \frac{(\fs + \fm)u_1}{\fr \fm z} +  \frac{\fs + \fm}{\fr \fm}\sum_{j = 4}^{3+\fp} k_j m_j + \frac{u_2}{\fm z}+ \frac{1}{\fm} \sum_{j = 4}^{3+\fp} k_j n_j  \right) H  -k \left( 1 - H + \frac{\fs + \fm}{\fr \fm}H_X + \frac{1}{\fm}H_Y \right)  \right) \\
    & \quad \cdot \frac{e^{-\hx/z}\prod_{j = 1}^{3 + \fp} \left(a_j(q)X^{m_j}Y^{n_j}\right)^{k_j}}{H^{k+1}}\\
    & = \left( - \frac{k}{H} + \frac{\fs + \fm}{\fr\fm} \frac{\partial}{\partial \log X} + \frac{1}{\fm} \frac{\partial}{\partial \log Y}  \right) \frac{e^{-\hx/z}\prod_{j = 1}^{3 + \fp} \left(a_j(q)X^{m_j}Y^{n_j}\right)^{k_j}}{H^k}.
\end{align*}
Then
\begin{align*}
    & \Res \left( \left( - \frac{k}{H} + \frac{\fs + \fm}{\fr\fm} \frac{\partial}{\partial \log X} + \frac{1}{\fm} \frac{\partial}{\partial \log Y} \right) \frac{e^{-\hx/z}\prod_{j = 1}^{3 + \fp} \left(a_j(q)X^{m_j}Y^{n_j}\right)^{k_j}}{H^k} \omega \right) \\
    & = \Res \left( -k \frac{e^{-\hx/z}\prod_{j = 1}^{3 + \fp} \left(a_j(q)X^{m_j}Y^{n_j}\right)^{k_j}}{H^{k+1}} \omega +  d \left( \frac{e^{-\hx/z}\prod_{j = 1}^{3 + \fp} \left(a_j(q)X^{m_j}Y^{n_j}\right)^{k_j}}{H^k} \left(\frac{(\fs f+\fm) d\log Y}{\fr\fm} + \frac{d\log X}{\fm} \right) \right)  \right) \\
    & = \Res \left( -k \frac{e^{-\hx/z}\prod_{j = 1}^{3 + \fp} \left(a_j(q)X^{m_j}Y^{n_j}\right)^{k_j}}{H^{k+1}} \omega \right)
\end{align*}
which implies the lemma.
\end{proof}


\begin{proof}[Proof of Proposition \ref{prop:IntSolvesPF}]
For $\beta \in \bL$, the proposition follows from iteratively applying Lemmas \ref{lem:Equiv1}, \ref{lem:Equiv2} to compute $\bD_{\beta}^{\bT'} \int_\gamma e^{-\hx/z}\Phi$ and at the end using
$$
    q^\beta = \prod_{i = 4}^{\fp} a_i(q)^{\inner{D_i, \beta}}.
$$
\end{proof}

\printbibliography

\end{document}